\documentclass[11pt, a4paper]{amsart}
\pdfoutput=1
\usepackage{tikz,enumitem,rotating,latexsym,bm,stmaryrd}
\usetikzlibrary{positioning,intersections,shapes.geometric,calc,decorations.pathreplacing}
\usepackage{graphicx}
\usepackage{amsmath,amsthm,amsfonts,amssymb,mathrsfs,pb-diagram}
\usepackage[pagebackref=true,bookmarks=true,colorlinks=true,linktoc=page,citecolor=darkgreen,linkcolor=blue,urlcolor=mediumblue]{hyperref}
\usepackage{caption}
\usepackage{lipsum,wasysym}
\usepackage{mathtools}
\usepackage[a4paper,margin=1in]{geometry}

\usepackage[color]{changebar}
\cbcolor{red}

\usepackage{color}
\definecolor{mediumblue}{rgb}{0.0, 0.0, 0.8}
\colorlet{darkgreen}{green!50!black}

\renewcommand*{\backref}[1]{}
\renewcommand*{\backrefalt}[4]{%
 \ifcase #1 No citations.
 \or [Page #2.]
 \else [Pages #2.]
 \fi%
}
\usepackage{cleveref}

\renewcommand{\geq}{\geqslant}
\renewcommand{\leq}{\leqslant}
\renewcommand{\trianglerighteq}{\trianglerighteqslant}
\renewcommand{\trianglelefteq}{\trianglelefteqslant}

\tikzset{wei/.style=
{red,double=red,double
distance=0.5pt}}

\tikzset{wei2/.style={red,double=red,double
distance=0.5pt}}
 
\allowdisplaybreaks
\numberwithin{equation}{section}
\usepackage{scalefnt}

\newtheorem{thm}{Theorem}[section]
\newtheorem{cor}[thm]{Corollary}

\newtheorem{lem}[thm]{Lemma}
\newtheorem{prop}[thm]{Proposition}
\newtheorem*{prop*}{Proposition}
\newtheorem*{thm*}{Theorem}

\newtheorem*{cor*}{Corollary}
\newtheorem*{conj*}{Conjecture}
\newtheorem*{move1}{Move 1}
\newtheorem*{move1'}{Move $\mathbf{1^*}$}
\newtheorem*{move2}{Move 2}

\theoremstyle{remark}

\newtheorem{rmk}[thm]{Remark}

\newtheorem*{Acknowledgements*}{Acknowledgements}

\theoremstyle{definition}
\newtheorem{defn}[thm]{Definition}
\newtheorem{eg}[thm]{Example}

\crefname{defn}{Definition}{Definitions}
\crefname{thm}{Theorem}{Theorems}
\crefname{prop}{Proposition}{Propositions}
\crefname{lem}{Lemma}{Lemmas}
\crefname{cor}{Corollary}{Corollaries}
\crefname{conj}{Conjecture}{Conjectures}
\crefname{section}{Section}{Sections}
\crefname{subsection}{Subsection}{Subsections}
\crefname{eg}{Example}{Examples}
\crefname{figure}{Figure}{Figures}
\crefname{rem}{Remark}{Remarks}
\crefname{rmk}{Remark}{Remarks}

\Crefname{defn}{Definition}{Definitions}
\Crefname{thm}{Theorem}{Theorems}
\Crefname{prop}{Proposition}{Propositions}
\Crefname{lem}{Lemma}{Lemmas}
\Crefname{cor}{Corollary}{Corollaries}
\Crefname{conj}{Conjecture}{Conjectures}
\Crefname{section}{Section}{Sections}
\Crefname{subsection}{Subsection}{Subsections}
\Crefname{eg}{Example}{Examples}
\Crefname{figure}{Figure}{Figures}
\Crefname{rem}{Remark}{Remarks}
\Crefname{rmk}{Remark}{Remarks}

\newcommand{\degr}{\mathrm{deg}}

\newcommand{\Add}{{\rm Add}}
\newcommand{\Remb}{{\rm Rem}}

\newcommand{\SStd}{\operatorname{SStd}}

\newcommand{\TSStd}{\operatorname{\mathcal{T}}}
\newcommand{\Shape}{\operatorname{Shape}} 

 \newcommand{\SSTS}{\mathsf{S}}  
\newcommand{\SSTT}{\mathsf{T}}  
\newcommand{\SSTU}{\mathsf{U}}  
\newcommand{\SSTV}{\mathsf{V}}  
\newcommand{\SSTQ}{\mathsf{Q}}  
\newcommand{\SSTR}{\mathsf{R}}  
\newcommand{\ZZ}{{\mathbb Z}}
\newcommand{\NN}{{\mathbb N}}
\newcommand{\g}{g}
\newcommand{\CC}{{\mathbb C}}
\newcommand{\RR}{{\mathbb R}}
\newcommand{\la}{{\lambda}}
\newcommand\mptn[2]{\mathscr{P}^{#1}_{#2}}

\newcommand\bbn{\mathbb{N}}

\newcommand\calm{\mathcal{M}}

\newcommand{\dom}{\trianglerighteqslant}

\DeclareMathOperator{\Ext}{Ext}

\let\<=\langle
\let\>=\rangle

\newcommand\Dec[1][A]{\mathbf{D}_{#1}(t)}

\tikzset{
    ultra thin/.style= {line width=0.05pt},
    very thin/.style=  {line width=0.2pt},
    thin/.style=       {line width=0.1pt},
    semithick/.style=  {line width=0.6pt},
    thick/.style=      {line width=0.8pt},
    very thick/.style= {line width=1.2pt},
    ultra thick/.style={line width=1.6pt}
}

\newcommand\Dim[2][t]{\text{\rm Dim}_{#1}#2}

\def\Item{\item\abovedisplayskip=0pt\abovedisplayshortskip=5pt~\vspace*{-\baselineskip}}

\begin{document}

\title[Kleshchev's decomposition numbers for diagrammatic Cherednik algebras]{Kleshchev's decomposition numbers \\ for diagrammatic Cherednik algebras}

\author{C.~Bowman}
\email{Chris.Bowman.2@city.ac.uk}
\address{Department of Mathematics,
City University London,
Northampton Square,
London,
EC1V 0HB,
UK}

\author{L.~Speyer}
\email{l.speyer@ist.osaka-u.ac.jp}
\address{Department of Pure and Applied Mathematics,
Graduate School of Information Science and Technology,
Osaka University,
Suita, Osaka 565-0871,
Japan}

\begin{abstract}
We construct  a family of graded isomorphisms between certain subquotients of diagrammatic Cherednik algebras as the 
quantum characteristic, multicharge, level, degree, and weighting are allowed to vary; this provides new structural information even in the case of the classical $q$-Schur algebra. This also allows us to prove some of the first results concerning the (graded) decomposition numbers of these algebras over fields of arbitrary characteristic.
\end{abstract}

\maketitle

\section*{Introduction}

Fix $\Bbbk$ an algebraically closed field of characteristic $p\geq 0$. Given a complex reflection group of type $G(l,1,n)$, a \emph{quantum characteristic} $e\in\{2,3,\ldots\}\cup \{\infty\}$, and an \emph{$e$-multicharge} $\kappa\in (\ZZ/e\ZZ)^l$ we have an associated cyclotomic Hecke algebra, $\mathcal{H}_n(\kappa)$.
In the semisimple case, the simple modules of $\mathcal{H}_n(\kappa)$ are labelled by the set of $l$-multipartitions, $\mptn ln$.
In the non-semisimple case, Ariki's categorification theorem, \cite{ariki96}, implies that  for each possible \emph{weighting} $\theta\in\mathbb{R}^l$ we have a corresponding parameterising set, $\Theta \subset \mptn ln$,  of the simple modules of $\mathcal{H}_n(\kappa)$.

We wish to study these Hecke algebras via an analogue of classical  {Schur--Weyl duality}.
The appropriate language for this is provided by Rouquier's formalism of \emph{quasi-hereditary covers} \cite{rouq08}.
In \cite[Section 3]{Webster}, it is shown that the \emph{diagrammatic Cherednik algebra}, $A(n,\theta,\kappa)$, is a {quasi-hereditary cover} of $\mathcal{H}_n(\kappa)$ and that the simple modules of $A(n,\theta,\kappa)$ which survive under the \emph{Schur functor} are precisely those which are labelled by $\Theta\subset \mptn ln$.  In particular the decomposition matrix of $\mathcal{H}_n(\kappa)$ appears  as a submatrix of that of $A(n,\theta,\kappa)$.
 
%
Over the complex field, the graded decomposition numbers of $A(n,\theta,\kappa)$ 
 are related to Uglov's canonical bases of higher level Fock spaces \cite{losev,RSVV,Webster}.  
By using Uglov's construction \cite{Uglov}, one can in principle give  an algorithm for computing the decomposition matrix over $\mathbb{C}$. However, in practice this algorithm is extremely slow.
Moreover, the picture deteriorates  drastically when  we consider fields of prime characteristic, where almost nothing is known.

%
%
In the case that $l=1$ the above specialises to the study of  the symmetric group and the Schur algebra of the general linear group (and their quantisations).  Some of the most interesting results here have sprung from generalising  Kleshchev's description of the decomposition numbers labelled by pairs of partitions which differ only by adding and removing a single node \cite{klesh1box}.
This was graded and  generalised to the Hecke algebra of the symmetric group by Chuang, Miyachi, Tan, and Teo (see \cite{cmt08,tanteo13}) as follows.
Fix $\gamma$  a (multi)partition with no removable $i$-nodes and  let $\Gamma$ denote the set of all partitions which may be obtained  by adding a total of $m$ $i$-nodes to $\gamma$.
Given $\lambda,\mu \in \Gamma$,  
the graded decomposition number $d_{ \lambda\mu}(t)$ 
is given in terms of \emph{nested sign sequences}.
As well as being one of the few results which holds in positive characteristic, this result is of interest over $\mathbb{C}$ as it provides a closed formula for $d_{\lambda\mu}(t)$, and so is computationally more efficient than the LLT algorithm.

In the main numerical result of this paper, we generalise the above to arbitrary diagrammatic Cherednik algebras.
Over $\mathbb{C}$, we show that the graded decomposition numbers $d_{\lambda\mu}(t)$ for $\lambda,\mu \in \Gamma$ of $A(n,\theta,\kappa)$ can be calculated in terms of nested sign sequences, see \cref{tanteolironcrispy2}.  
We then show, under a mild restriction on  $\kappa$, that the corresponding  (submatrices of the)  adjustment matrices for $A(n,\theta,\kappa)$ 
 are trivial, thereby calculating  the graded decomposition numbers $d_{\lambda\mu}(t)$ of $A(n,\theta,\kappa)$, for $\lambda,\mu \in \Gamma$, over  fields  of arbitrary characteristic; see \cref{charp}.

This is done by proving a stronger, structural result  over fields of arbitrary characteristic.
 Given $\gamma$ a multipartition, the set $\Gamma$ is closed under the dominance order  and so there is a strong relationship between the diagrammatic Cherednik algebra, $A$, and a certain subquotient $A_{\Gamma}$; in particular the graded decomposition numbers and certain higher extension groups are preserved.
We define a sequence $\chi(\gamma)$ associated to a multipartition $\gamma$, and show that if  $\gamma$  and $\overline\gamma$ are arbitrary multipartitions (which need not have the same level or degree) with $\chi(\gamma)$ \emph{equivalent to} $\chi(\overline\gamma)$, then the corresponding subquotients $A_{\Gamma}$  and  $A_{\overline\Gamma}$ are isomorphic as graded $\Bbbk$-algebras.
This allows us to compare diagrammatic Cherednik algebras as the 
quantum characteristic, multicharge, level, degree, and weighting are all allowed to vary.
This provides new structural information even in the case of the classical Schur algebras of type $G(1,1,n)$ (see \cref{magicexample}) and their higher level counterparts, the cyclotomic ($q$-)Schur algebras of Dipper, James and Mathas \cite{DJM}.

In \cite{ct16} it is shown that the results of \cite{cmt08,tanteo13} actually hold in more generality.  As long as  we never add or remove nodes whose residues differ by 1, then the graded  decomposition numbers (over $\CC$) can be written as the product of the decomposition numbers for the individual residues.
In this paper, we lift this result to the structural level and prove that it holds over fields of arbitrary characteristic  (and generalise it to arbitrary diagrammatic Cherednik algebras) by showing that the algebras involved decompose as tensor products according to residue, see \cref{tensordecomp2}.  

The paper is structured as follows.  In \cref{lCasec} we recall the definition of the diagrammatic Cherednik algebra defined by Webster in  \cite{Webster} and the combinatorics underlying its representation theory.  
In \cref{nestsec} we recall the combinatorics of  nested sign sequences from \cite{tanteo13}.    
In \cref{subqsec} we define the subquotient algebras in which we are interested and construct cellular bases of these algebras.  
In \cref{isosec} we construct a family of graded isomorphisms between the subquotient algebras.  We first illustrate  how one can deduce the decomposition numbers of these algebras over $\mathbb{C}$ using only the isomorphism on the level of graded vector spaces.  We then lift this to an isomorphism of graded $\Bbbk$-algebras and hence calculate  the decomposition numbers  over an arbitrary field $\Bbbk$.  
In \cref{sec5}, we then construct the isomorphism which decomposes  the adjacency-free subquotient algebras as tensor products of the smaller algebras corresponding to the individual residues.

\begin{Acknowledgements*}
The authors would like to thank Joe Chuang, Anton Cox, Kai Meng Tan, and Ben Webster for helpful conversations during the preparation of this manuscript. We also thank the referee for their helpful comments.
%
The authors are grateful for the financial support received from 
the {Royal Commission for the Exhibition of 1851}, the London Mathematical Society, and the Japan 
 Society for the Promotion of Science. 
\end{Acknowledgements*}

\section{The diagrammatic Cherednik algebra}\label{lCasec}

In this section we define the diagrammatic Cherednik algebras and recall the combinatorics underlying their representation theory.

\subsection{Graded cellular algebras with highest weight theories.}
We shall study the diagrammatic Cherednik  algebras through the following framework.

\begin{defn}\label{defn1}
Suppose that $A$ is a $\ZZ$-graded $\Bbbk$-algebra which is of finite rank over $\Bbbk$.
We say that $A$ is a \emph{graded cellular algebra with a highest weight theory} if the following conditions hold.

The algebra is equipped with a \emph{cell datum} 
$(\Lambda,\TSStd,C,\degr)$, where $(\Lambda,\trianglerighteq )$ is the \emph{weight poset}.
For each $\lambda,\mu\in\Lambda$   such that $\lambda \trianglerighteq \mu$, we have a finite set, denoted $\TSStd(\lambda,\mu)$, and we let $\TSStd(\lambda)=\cup_\mu \TSStd(\lambda,\mu)$.
There exist maps
\[
C:{\coprod_{\lambda\in\Lambda}\TSStd(\lambda)\times \TSStd(\lambda)}\to A 
      \quad\text{and}\quad
     \degr: {\coprod_{\lambda\in\Lambda}\TSStd(\lambda)} \to \ZZ
\]
such that $C$ is injective. We denote $C(\SSTS,\SSTT) = c^\lambda_{\SSTS\SSTT}$ for $\SSTS,\SSTT\in\TSStd(\lambda)$.  
 We require that $A$ satisfies properties (1)--(6), below.  
  \begin{enumerate}
    \item Each element $c^\lambda_{\SSTS,\SSTT}$ is homogeneous
	of degree 
	\[
	\degr(c^\lambda_{\SSTS,\SSTT}) = \degr(\SSTS)+\degr(\SSTT),
	\]
	for $\lambda\in\Lambda$ and $\SSTS,\SSTT\in \TSStd(\lambda)$.
    \item The set $\{c^\lambda_{\SSTS,\SSTT}\mid\SSTS,\SSTT\in \TSStd(\lambda), \,
      \lambda\in\Lambda \}$ is a  
      $\Bbbk$-basis of $A$.
    \item  If $\SSTS,\SSTT\in \TSStd(\lambda)$, for some
      $\lambda\in\Lambda$, and $a\in A$ then 
    there exist scalars $r_{\SSTS,\SSTU}(a)$, which do not depend on
    $\SSTT$, such that 
      \[
      ac^\lambda_{\SSTS,\SSTT} = \; \sum_{\mathclap{\SSTU\in
      \TSStd(\lambda)}} \; r_{\SSTS,\SSTU}(a)c^\lambda_{\SSTU,\SSTT}\pmod 
      {A^{\vartriangleright  \lambda}},
      \]
      where $A^{\vartriangleright  \lambda}$ is the $\Bbbk$-submodule of $A$ spanned by
      \[
      \{c^\mu_{\SSTQ,\SSTR}\mid\mu \vartriangleright  \lambda\text{ and }\SSTQ,\SSTR\in \TSStd(\mu )\}.
      \]
    \item  The $\mathbb{C}$-linear map $*:A\to A$ determined by
      $(c^\lambda_{\SSTS,\SSTT})^*=c^\lambda_{\SSTT,\SSTS}$, for all $\lambda\in\Lambda$ and
      all $\SSTS,\SSTT\in\TSStd(\lambda)$, is an anti-isomorphism of $A$.
\item The identity $1_A$ of $A$ has a decomposition $1_A = \sum_{\lambda \in \Lambda} 1_\lambda$ into pairwise orthogonal idempotents $1_\lambda$.
\item For $\SSTS\in\TSStd(\lambda,\mu)$, $\SSTT \in\TSStd(\lambda,\nu)$, we have that $1_\mu c^\lambda_{\SSTS,\SSTT}1_\nu =  c^\lambda_{\SSTS,\SSTT}$.
There exists a unique element $\SSTT^\lambda \in \TSStd(\lambda,\lambda)$, and $c_{\SSTT^\lambda,\SSTT^\lambda}^\lambda = 1_\lambda$.
\end{enumerate}
\end{defn}

All results in this section follow from \cite{hm10}.
Suppose that $A$ is a graded cellular algebra with a highest weight theory. Given any $\lambda\in\Lambda$, the \emph{graded standard module} $\Delta(\lambda)$ is the graded left $A$-module
\[
\Delta(\lambda)=\bigoplus_{\begin{subarray}c \mu\in\Lambda \\ z\in\ZZ \end{subarray}}\Delta_\mu(\lambda)_z,
\]
where $\Delta_\mu(\lambda)_z$ is the $\mathbb{C}$ vector-space  with basis
$\{c^\lambda_{\SSTS } \mid \SSTS\in \TSStd(\lambda,\mu)\text{ and }\deg(\SSTS)=z\}$.
The action of $A$ on $\Delta(\lambda)$ is given by
\[
a c^\lambda_{ \SSTS  }  = \sum_{\mathclap{\SSTU \in \TSStd(\lambda)}} \; r_{\SSTS,\SSTU}(a) c^\lambda_{\SSTU},
\]
where the scalars $r_{\SSTS,\SSTU}(a)$ are the scalars appearing in condition (3) of Definition \ref{defn1}.

Suppose that $\lambda \in \Lambda$. There is a bilinear form $\langle\ ,\ \rangle_\lambda$ on $\Delta(\lambda) $ which
is determined by
\[c^\lambda_{\SSTU,\SSTS}c^\lambda_{\SSTT,\SSTV}\equiv
  \langle c^\lambda_\SSTS,c^\lambda_\SSTT\rangle_\lambda c^\lambda_{\SSTU,\SSTV}\pmod{A^{\triangleright \lambda}},\]
for any $\SSTS,\SSTT, \SSTU,\SSTV\in \TSStd(\lambda  )$.
Let $t$ be an indeterminate over $\ZZ_{\geq 0}$. If $M=\oplus_{z\in\ZZ}M_z$ is
a free graded $\mathbb{C}$-module, then its \emph{graded dimension} is the Laurent polynomial
\[
\Dim{(M)}=\sum_{k\in\ZZ}(\dim_{\mathbb{C}}M_k)t^k.
\]

If $M$ is a graded $A$-module and $k\in\ZZ$, define $M\langle k \rangle$ to be the same module with $(M\langle k \rangle)_i = M_{i-k}$ for all $i\in\ZZ$. We call this a \emph{degree shift} by $k$.
If $M$ is a graded $A$-module and $L$ is a graded simple module let $[M:L\langle k\rangle]$ be the
multiplicity of 
$L\langle k\rangle$ as a graded composition factor of $M$, for $k\in\ZZ$.

Suppose that $A$ is a graded cellular algebra with a highest weight theory.
For every $\lambda \in \Lambda$,  define  $L(\lambda)$ to be  the quotient of the corresponding standard module $\Delta(\lambda)$ by the radical of the bilinear form $\langle\ ,\ \rangle_\lambda$.  This module is simple, and every simple module is of the form $L(\lambda)\langle k \rangle$ for some $k \in \ZZ$, $\lambda \in \Lambda$.  We let $L_\mu(\lambda)$ denote the $\mu$-weight space $1_\mu L(\lambda)$.
The \emph{graded decomposition matrix} of $A$ is the matrix $\Dec=(d_{\lambda\mu}(t))$, where
\[
d_{\lambda\mu}(t)=\sum_{k\in\ZZ} [\Delta(\lambda):L(\mu)\langle k\rangle]\,t^k,
\]
for $\lambda,\mu\in\Lambda$.
\begin{prop}[\cite{hm10}, Proposition 2.18]\label{humathasprop}
If $\lambda,\mu\in\Lambda$ then $\Dim{(L_\mu(\lambda))} \in \ZZ_{\geq 0}[t+t^{-1}]$.
\end{prop}

\subsection{Combinatorial preliminaries}
Fix integers $l,n\in\ZZ_{\geq 0} $, $\g\in \RR_{> 0}$ and $e\in\{3,4,\dots\}\cup\{\infty\}$. 
We define a \emph{weighting} $\theta = (\theta_1,\dots, \theta_l) \in \RR^l$ to be any $l$-tuple such that
$\theta_i-\theta_j$ is not an integer multiple of  $\g$ for $1\leq i< j \leq l$.
Let $\kappa$ denote an $e$-\emph{multicharge} $\kappa = (\kappa_1,\dots,\kappa_l)\in (\ZZ/e\ZZ)^l$.

\begin{rmk}
We say that a weighting $\theta\in\RR^l$ is \emph{well-separated} for $A(n,\theta,\kappa)$ if $|\theta_i-\theta_j| > ng$ for all $1\leq i < j \leq l$.
We say that a weighting $\theta\in\RR^l$ is a \emph{FLOTW} weighting for $A(n,\theta,\kappa)$ if $0<|\theta_i-\theta_j| < g$ for all $1\leq i < j \leq l$.
\end{rmk}

\begin{defn}
An \emph{$l$-multipartition} $\la=(\la^{(1)},\dots,\la^{(l)})$ of $n$ is an $l$-tuple of partitions such that $|\la^{(1)}|+\dots+ |\la^{(l)}|=n$. We will denote the set of $l$-multiparititons of $n$ by $\mptn ln$.
\end{defn}

We define the \emph{Russian array} as follows.
For each $1\leq k\leq l$, we place a point on the real line at $\theta_k$ and consider the region bounded by half-lines starting at $\theta_k$ at angles $3\pi/4 $ and $\pi/4$.
We tile the resulting quadrant with a lattice of squares, each with diagonal of length $2 \g $.

Let $\lambda=(\lambda^{(1)},\lambda^{(2)},\ldots ,\lambda^{(l)}) \in \mptn ln$.
The \emph{Young diagram} $[\lambda]$ is defined to be the set
\[
\{(r,c,k)\in\bbn\times\bbn\times\{1,\dots,l\} \mid c\leq\la^{(k)}_r\}.
\]
We refer to elements of $[\la]$ as nodes (of  $[\la]$ or $\la$). 
We define the \emph{residue} of a node $(r,c,k) \in [\lambda]$ to be $\kappa_k+c-r \pmod e$, and refer to $(r,c,k)$ as an \emph{$i$-node} if it has residue $i$.  

We  define an addable (respectively removable) node  of $\lambda$ to be any node which can be added to   (respectively removed from) the diagram $[\lambda]$ to obtain the Young diagram of a multipartition.  Given $S\subset \ZZ/e\ZZ$ we let 
$\Remb_S(\lambda)$ (respectively $\Add_S(\lambda)$) denote  the set of  removable (respectively addable) $i$-nodes of $\lambda$ for all $i\in S$.

For each node of $[\la]$ we draw a box in the plane; we shall draw our Young diagrams in a mirrored-Russian convention. We place the first node of component $m$ at $\theta_m$ on the real line, with rows going northwest from this node, and columns going northeast. The diagram is tilted ever-so-slightly in the clockwise direction so that the top vertex of the box $(r,c,k)$ (that is, the box in the $r$th row and $c$th column of the $k$th component of $ [\lambda] $) has $x$-coordinate $\theta_k + g(r-c) + (r+c)\epsilon$.  

Here the tilt $\epsilon$ is chosen so that $n\epsilon$ is absolutely small with respect to $g$ (so that $\epsilon \ll g/n$) and with respect to the weighting (so that $g$ does not divide any number in the interval  $|\theta_i-\theta_j| + (-n\epsilon,+n\epsilon)$ for $1\leq i < j\leq l$).

We define a \emph{loading}, $\mathbf{i}$, to be an element of $(\mathbb{R} \times (\ZZ/e\ZZ))^n$ such that no real number occurs with multiplicity greater than one.
Given a multipartition $\lambda \in \mptn ln$  we have an associated loading, $\mathbf{i}_\lambda^\theta$ (or simply $\mathbf{i}_\lambda$ when $\theta$ is clear from the context) given by the projection of the top vertex of each box $(r,c,k)\in[\la]$ to its $x$-coordinate $\mathbf{i}_{(r,c,k)}\in\RR$, and attaching to each point the {residue} $\kappa_k+c-r \pmod e$ of this node. Note that the above conditions on $\epsilon$ are designed to ensure that no two nodes have the same $x$-coordinate, so that $\mathbf{i}_\la$ is really a loading.

We let $D_\lambda$ denote the underlying ordered subset of $\RR$ given by the points of the loading. Given $a \in D_\lambda$, we abuse notation and let $a$ denote the corresponding node of $ \lambda $ (that is, the node whose top vertex projects onto $x$-coordinate $a \in \RR$).
The \emph{residue sequence} of $\lambda$ is given by reading the residues of the nodes of $ \lambda $ according to the ordering given by $D_\lambda$.

\begin{eg}
Let $l=2$,  $\g=1$, $\epsilon=1/100$, and $\theta=(0,0.5)$.
The bipartition $((2,1),(1^3))$ has Young diagram and corresponding loading $\mathbf{i}_\lambda$ given in \cref{ALoading}.
The residue sequence of $\lambda$ is $(\kappa_1+1,\kappa_1,\kappa_2,\kappa_1-1,\kappa_2-1,\kappa_2-2)$, and the ordered set $D_\lambda$ is $\{-0.97,0.02, 0.52, 1.03, 1.53 ,2.54 \}$.
The node $x={-0.97}$ in $  \lambda $ can be identified with the node in the first row and second column of  $\lambda^{(1)}$.

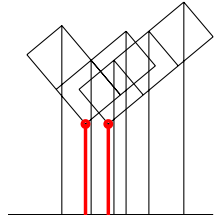
\begin{figure}[ht]\captionsetup{width=0.9\textwidth}
\[
\begin{tikzpicture}[scale=0.6]
 {\clip (-1.7,-2) rectangle ++(5,5);
\path [draw,name path=upward line] (-1.7,-2) -- (2.8,-2);
  \draw 
      (0, 0)              coordinate (a1)
            -- +(40:1) coordinate (a2)  
      (a2)   -- +(130:1) coordinate (a3)
      (a3)   -- +(220:1) coordinate (a4)       
      (a4)   -- +(310:1) coordinate (a5)     
(a3)  -- +(270:10)
      (a2)             
            -- +(40:1) coordinate (b2)  
      (b2)   -- +(130:1) coordinate (b3)
      (b3)   -- +(220:1) coordinate (b4)       
      (b4)   -- +(310:1) coordinate (b5)   
      (a4)   -- +(130:1) coordinate (c1)            
      (c1) -- +(40:1)coordinate (c2)    
            (c2) -- +(-50:1)coordinate (c3)            
      (c2)  -- +(270:10)
      (b3)  -- +(270:10)
        ; 
      \draw[wei2](a1) -- +(270:10);
      \draw[wei2](a1)  circle (2pt);
     \path   (0,0)  -- +(0:0.5) coordinate (d)   
      ;
        \draw[wei2](d) -- +(270:10);

        \draw 
      (d)              
               -- +(40:1) coordinate (d2)  
                (d2)   -- +(130:1) coordinate (d3)
                 (d3)   -- +(220:1) coordinate (d4)       
      (d4)   -- +(310:1) coordinate (d5) 
            (d3) -- +(270:10)
               (d2)              
               -- +(40:1) coordinate (e2)  
                  (e2)   -- +(130:1) coordinate (e3)
                     (e3)   -- +(220:1) coordinate (e4)       
(e3)  -- +(270:10)      
(e2)              
               -- +(40:1) coordinate (f2)  
                  (f2)   -- +(130:1) coordinate (f3)
                    (f3)   -- +(220:1) coordinate (f4)       
 (f3)  -- +(270:10)      
      ;  
        \draw[wei2](d)  circle (2pt); }
    \end{tikzpicture}
    \]
\caption{The diagram and loading of the bipartition $((2,1),(1^3))$ for $l=2$, $g=1$, $\theta=(0,0.5)$.}
\label{ALoading}
\end{figure}
   \end{eg}

\begin{defn} Let $\lambda,\mu \in \mptn l n$.  A $\lambda$-tableau of weight $\mu$ is a
  bijective map $\SSTT : [\lambda] \to D_\mu$ which respects residues.  
  In other words,  we fill a given node $(r,c,k)$ of the diagram
  $[\lambda]$ with a real number $d$ from $D_\mu$ (without
  multiplicities) so that the residue attached to the real number $d$ in the loading $\mathbf{i}_\mu$
  is equal to $\kappa_k+c-r \pmod e$.
\end{defn}

\begin{defn}\label{sdfjhkhjkfhjdsgfhjkhjkhjgkfjshkdkjhfgkjhgdfs}
A $\lambda$-tableau, $\SSTT$, of shape $\lambda$ and weight $\mu$ is said to be semistandard if
\begin{itemize}
\item $\SSTT(1,1,k)>\theta_k$,
\item $\SSTT(r,c,k)> \SSTT(r-1,c,k) +\g$,
\item $\SSTT(r,c,k)> \SSTT(r,c-1,k) -\g$.
\end{itemize}
We denote the set of all  semistandard tableaux of shape $\lambda$ and weight $\mu$ by $\SStd(\lambda,\mu)$. Given $\SSTT \in \SStd(\lambda,\mu)$, we write $\Shape(\SSTT)=\lambda$.
\end{defn}

\begin{eg} Fix $\kappa=(0)$, $\theta=(0)$ and $g=1$ and let $\epsilon \to0$.
For $e=4$, there is a unique $\SSTS\in \SStd((3,1),(2,1^2))$. 
This tableau is the leftmost depicted in   \cref{semistandard}, below.  
The diagram depicts a partition of shape $(3,1)$ whose boxes are 
filled with   integers.  These integers are obtained from the $x$-coordinates 
 of the nodes of the Young diagram 
$(2,1^2)$.  To see this, note that 
\[
\mathbf{i}_{(1,2,1)}= -1+3\epsilon\qquad
\mathbf{i}_{(1,1,1)}=  0+2\epsilon \qquad
\mathbf{i}_{(2,1,1)}=  1+3\epsilon \qquad  
\mathbf{i}_{(3,1,1)}=  2+4\epsilon 
\]
and by letting $\epsilon \to 0$ we obtain the entries of the tableau.  One can check that this is the only tableau which satisfies the conditions in \cref{sdfjhkhjkfhjdsgfhjkhjkhjgkfjshkdkjhfgkjhgdfs}.  
 Similarly, for $e=5$, there is a unique $\SSTT\in \SStd((6,1^4),(5,1^5))$ and  a unique $\SSTU \in {\SStd((6,2^2,1^2),(5,3^2,1^3))}$.
These   semistandard tableaux are depicted in Figure \ref{semistandard}, below.

In all cases we let $\epsilon \to 0$ to make the loadings easier to read.
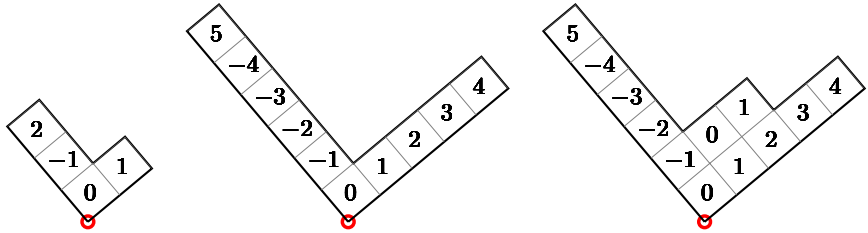
\begin{figure}[h]
\begin{center}\scalefont{0.8}
\begin{tikzpicture}[scale=1.1]
   \path (0,0) coordinate (origin);
    \draw[wei2] (0,0)   circle (2pt);
  \draw[thick] (origin) 
   --++(130:3*0.5)
   --++(40:1*0.5)
  --++(-50:2*0.5)	
  --++(40:1*0.5)
  --++(-50:1*0.5)
     --++(220:2*0.5);
     \clip 
   (origin) 
   --++(130:3*0.5)
   --++(40:1*0.5)
  --++(-50:2*0.5)	
  --++(40:1*0.5)
  --++(-50:1*0.5)
     --++(220:2*0.5);  
   \path (40:1cm) coordinate (A1);
  \path (40:2cm) coordinate (A2);
  \path (40:3cm) coordinate (A3);
  \path (40:4cm) coordinate (A4);
  \path (130:1cm) coordinate (B1);
  \path (130:2cm) coordinate (B2);
  \path (130:3cm) coordinate (B3);
  \path (130:4cm) coordinate (B4);
  \path (A1) ++(130:3cm) coordinate (C1);
  \path (A2) ++(130:2cm) coordinate (C2);
  \path (A3) ++(130:1cm) coordinate (C3);
  \foreach \i in {1,...,19}
  {
    \path (origin)++(40:0.5*\i cm)  coordinate (a\i);
    \path (origin)++(130:0.5*\i cm)  coordinate (b\i);
    \path (a\i)++(130:4cm) coordinate (ca\i);
    \path (b\i)++(40:4cm) coordinate (cb\i);
    \draw[thin,gray] (a\i) -- (ca\i)  (b\i) -- (cb\i); 
   \draw  (origin)++(40:0.25)++(130:0.25) node {${0}$} ; 
   \draw  (origin)
   ++(130:0.5) 
   ++(40:0.25)++(130:0.25) node {${-1}$} ; 
      \draw  (origin)
   ++(130:2*0.5) 
   ++(40:0.25)++(130:0.25) node {${2}$} ; 
      \draw  (origin)
    ++(40:0.5)
   ++(40:0.25)++(130:0.25) node {${1}$} ; 
  }
   \end{tikzpicture}
\quad
\begin{tikzpicture}[scale=1.1]
   \path (0,0) coordinate (origin);
     \draw[wei2] (0,0)   circle (2pt);
  \draw[thick] (origin) 
   --++(130:6*0.5)
   --++(40:1*0.5)
  --++(-50:5*0.5)	
  --++(40:4*0.5)
  --++(-50:1*0.5)
     --++(220:5*0.5);
     \clip 
 (origin) 
   --++(130:6*0.5)
   --++(40:1*0.5)
  --++(-50:5*0.5)	
  --++(40:4*0.5)
  --++(-50:1*0.5)
     --++(220:5*0.5); 
   \path (40:1cm) coordinate (A1);
  \path (40:2cm) coordinate (A2);
  \path (40:3cm) coordinate (A3);
  \path (40:4cm) coordinate (A4);
  \path (130:1cm) coordinate (B1);
  \path (130:2cm) coordinate (B2);
  \path (130:3cm) coordinate (B3);
  \path (130:4cm) coordinate (B4);
  \path (A1) ++(130:3cm) coordinate (C1);
  \path (A2) ++(130:2cm) coordinate (C2);
  \path (A3) ++(130:1cm) coordinate (C3);
  \foreach \i in {1,...,19}
  {
    \path (origin)++(40:0.5*\i cm)  coordinate (a\i);
    \path (origin)++(130:0.5*\i cm)  coordinate (b\i);
    \path (a\i)++(130:4cm) coordinate (ca\i);
    \path (b\i)++(40:4cm) coordinate (cb\i);
    \draw[thin,gray] (a\i) -- (ca\i)  (b\i) -- (cb\i); 
   \draw  (origin)++(40:0.25)++(130:0.25) node {${0}$} ; 
   \draw  (origin)
   ++(130:0.5) 
   ++(40:0.25)++(130:0.25) node {${-1}$} ; 
            \draw  (origin)
    ++(40:0.5)
   ++(40:0.25)++(130:0.25) node {${1}$} ; 
      \draw  (origin)
   ++(130:0.5) 
   ++(40:0.25)++(130:0.25) node {${-1}$} ; 
      \draw  (origin)
   ++(130:2*0.5) 
   ++(40:0.25)++(130:0.25) node {${-2}$} ; 
      \draw  (origin)
    ++(40:0.5)
   ++(40:0.25)++(130:0.25) node {${1}$} ; 
      \draw  (origin)
   ++(130:3*0.5) 
   ++(40:0.25)++(130:0.25) node {${-3}$} ; 
      \draw  (origin)
   ++(130:4*0.5) 
   ++(40:0.25)++(130:0.25) node {${-4}$} ; 
      \draw  (origin)
   ++(130:5*0.5) 
   ++(40:0.25)++(130:0.25) node {${5}$} ; 
    \draw  (origin)
    ++(40:2*0.5)
   ++(40:0.25)++(130:0.25) node {${2}$} ; 
    \draw  (origin)
    ++(40:3*0.5)
   ++(40:0.25)++(130:0.25) node {${3}$} ; 
     \draw  (origin)
    ++(40:4*0.5)
   ++(40:0.25)++(130:0.25) node {${4}$} ;  }
   \end{tikzpicture}
\quad
\begin{tikzpicture}[scale=1.1]
   \path (0,0) coordinate (origin);
     \draw[wei2] (0,0)   circle (2pt);
  \draw[thick] (origin) 
   --++(130:6*0.5)
   --++(40:1*0.5)
  --++(-50:4*0.5)	
  --++(40:2*0.5)
    --++(-50:1*0.5)
      --++(40:2*0.5)
  --++(-50:1*0.5)
     --++(220:5*0.5);
     \clip 
 (origin) 
   --++(130:6*0.5)
   --++(40:1*0.5)
  --++(-50:4*0.5)	
  --++(40:2*0.5)
    --++(-50:1*0.5)
      --++(40:2*0.5)
  --++(-50:1*0.5)
     --++(220:5*0.5);
   \path (40:1cm) coordinate (A1);
  \path (40:2cm) coordinate (A2);
  \path (40:3cm) coordinate (A3);
  \path (40:4cm) coordinate (A4);
  \path (130:1cm) coordinate (B1);
  \path (130:2cm) coordinate (B2);
  \path (130:3cm) coordinate (B3);
  \path (130:4cm) coordinate (B4);
  \path (A1) ++(130:3cm) coordinate (C1);
  \path (A2) ++(130:2cm) coordinate (C2);
  \path (A3) ++(130:1cm) coordinate (C3);
  \foreach \i in {1,...,19}
  {
    \path (origin)++(40:0.5*\i cm)  coordinate (a\i);
    \path (origin)++(130:0.5*\i cm)  coordinate (b\i);
    \path (a\i)++(130:4cm) coordinate (ca\i);
    \path (b\i)++(40:4cm) coordinate (cb\i);
    \draw[thin,gray] (a\i) -- (ca\i)  (b\i) -- (cb\i); 
   \draw  (origin)++(40:0.25)++(130:0.25) node {${0}$} ; 
   \draw  (origin)++(40:0.75)++(130:0.75) node {${0}$} ; 
   \draw  (origin)
   ++(130:0.5) 
   ++(40:0.25)++(130:0.25) node {${-1}$} ; 
            \draw  (origin)
    ++(40:0.5)
   ++(40:0.25)++(130:0.25) node {${1}$} ; 
      \draw  (origin)
   ++(130:0.5) 
   ++(40:0.25)++(130:0.25) node {${-1}$} ; 
      \draw  (origin)
   ++(130:2*0.5) 
   ++(40:0.25)++(130:0.25) node {${-2}$} ; 
      \draw  (origin)
    ++(40:0.5)
   ++(40:0.25)++(130:0.25) node {${1}$} ; 
         \draw  (origin)
    ++(40:1)    ++(130:0.5)
   ++(40:0.25)++(130:0.25) node {${1}$} ; 
      \draw  (origin)
   ++(130:3*0.5) 
   ++(40:0.25)++(130:0.25) node {${-3}$} ; 
      \draw  (origin)
   ++(130:4*0.5) 
   ++(40:0.25)++(130:0.25) node {${-4}$} ; 
      \draw  (origin)
   ++(130:5*0.5) 
   ++(40:0.25)++(130:0.25) node {${5}$} ; 
    \draw  (origin)
    ++(40:2*0.5)
   ++(40:0.25)++(130:0.25) node {${2}$} ; 
    \draw  (origin)
    ++(40:3*0.5)
   ++(40:0.25)++(130:0.25) node {${3}$} ; 
     \draw  (origin)
    ++(40:4*0.5)
   ++(40:0.25)++(130:0.25) node {${4}$} ;  }
   \end{tikzpicture}
\end{center}
\caption{ Three semistandard tableaux 
$\SSTS\in \SStd((3,1),(2,1^2))$ and 
 $\SSTT\in \SStd((6,1^4),(5,1^5))$
 and 
 $\SSTU\in \SStd((6,2^2,1^2),(5,2^2,1^3))$.  
}
\label{semistandard}
\end{figure}

\end{eg}
\begin{rmk}
For many more examples of the  combinatorics of loadings and  tableaux, we refer the reader to \cite[Section 2]{bcs15}.
\end{rmk}
 \begin{defn} 
 Let $\mathbf{i}$ and $\mathbf{j}$ denote two loadings of size $n$ and let
  $r\in \ZZ/e\ZZ$.  
 We say that  $\mathbf{i}$ $r$-dominates $\mathbf{j}$ if 
 for every real number $a \in \RR$, we have that 
\[
|\{(x,r ) \in \mathbf{i} \mid  x<a\}| \geq  |\{(x,r ) \in \mathbf{j} \mid x<a\}|.
\]
We say that $\mathbf{i}$ dominates  $\mathbf{j}$ if $\mathbf{i}$ $r$-dominates  $\mathbf{j}$   for every $r \in \ZZ/e\ZZ$.
Given $\lambda,\mu \in \mptn ln$ and $\theta\in \RR^l$, we say that $\lambda$
$\theta$-dominates $\mu$ (and write $\mu \trianglelefteq_{\theta}
\lambda$) if $\mathbf{i}^\theta_\lambda$ dominates $\mathbf{i}^\theta_\mu$. 
\end{defn}

\begin{rmk}
We note that for $l>1$, the loading of the partitions (and therefore the resulting $\theta$-dominance order) is heavily dependent on the weighting $\theta\in \RR^l$.   
\end{rmk}

\begin{defn}
We refer to an unordered multiset $\mathcal{R}$ of $n$ elements from $ (\ZZ/e\ZZ)$ as a residue set of cardinality $n$.  We let $\mptn ln(\mathcal{R})$ denote the subset of $\mptn ln$ whose residue set is equal to $\mathcal{R}$.
\end{defn}

\begin{rmk}
We have that $\mptn ln = \cup_{\mathcal{R} } \mptn ln(\mathcal{R})$ is a disjoint decomposition of  the set $\mptn ln$; notice that all of the above combinatorics respects this decomposition.
\end{rmk}

\begin{rmk}
The above combinatorics can be generalised to multi-compositions in the obvious manner.
\end{rmk}

\subsection{The diagrammatic Cherednik algebra}
\label{relationspageofstuff} 

Recall that we have fixed   $l,n\in \ZZ_{>0}$,   $\g\in \RR_{>0}$
and $e\in\{3,4,\dots\}\cup\{\infty\}$. Given any weighting $\theta=(\theta_1,\dots, \theta_l)$ and $\kappa=(\kappa_1,\ldots ,\kappa_l)$ an $e$-multicharge, we will define what we refer to as the diagrammatic Cherednik algebra, $A(n,\theta,\kappa)$.

This is an example of one of many finite dimensional algebras (reduced steadied quotients of weighted KLR algebras in Webster's terminology) constructed in \cite{Webster}, whose module categories are equivalent, over the complex field, to    category $\mathcal{O}$ for a rational cyclotomic Cherednik algebra \cite[Theorem 2.3 and 3.9]{Webster}.
Over fields of arbitrary characteristic and $\theta$ a well-separated weighting, the algebra $A(n,\theta,\kappa)$ is Morita equivalent to the corresponding cyclotomic $q$-Schur algebra of \cite{DJM} with the same level, rank, quantum characteristic and $e$-multicharge \cite[Theorem 3.9]{Webster2}.
 
\begin{defn}
We define a  $\theta$-\emph{diagram} {of type} $G(l,1,n)$ to  be a  \emph{frame} $\mathbb{R}\times [0,1]$ with  distinguished  black points on the northern and southern boundaries given by the loadings $\mathbf{i}_\mu$ and $\mathbf{i}_\lambda$ for some $\lambda,\mu \in \mptn ln(\mathcal{R})$ and a collection of curves each of which starts at a northern point and ends at a southern point of the same residue, $i$ say (we refer to this as a \emph{black $i$-strand}).
We further require that each curve has a mapping diffeomorphically to $[0,1]$ via the projection to the $y$-axis.
Each curve is allowed to carry any number of dots. We draw
\begin{itemize}
\item  a dashed line $g$ units to the left of each strand, which we call a \emph{ghost $i$-strand} or $i$-\emph{ghost};
\item vertical red lines at $\theta_k \in \RR$ each of which carries a 
residue $\kappa_k$ for $1\leq k\leq l$ which we call a \emph{red $\kappa_k$-strand}. 
\end{itemize}
We now require that there are no triple points or tangencies involving any combination of strands, ghosts or red lines and no dots lie on crossings. We consider these diagrams equivalent if they are related by an isotopy that avoids these tangencies, double points and dots on crossings.
\end{defn}

\begin{rmk}Note  that our diagrams  do not distinguish between `over' and `under' crossings. 
\end{rmk}
\begin{defn}[\cite{Webster}]
The \emph{diagrammatic Cherednik algebra}, $A(n,\theta,\kappa)$, is the span of all $\theta$-diagrams modulo the following local relations (here a local relation means one that can be applied on a small region of the diagram).
\begin{enumerate}[label=(1.\arabic*)]
\item\label{rel1}  Any diagram may be deformed isotopically; that is, by a continuous deformation
of the diagram which at no point introduces or removes any crossings of strands (black,  ghost, or  red).
\item\label{rel2}
For $i\neq j$ we have that dots pass through crossings.
\[
\scalefont{0.8}\begin{tikzpicture}[scale=.6,baseline]
\draw[very thick](-4,0) +(-1,-1) -- +(1,1) node[below,at start]
{$i$}; \draw[very thick](-4,0) +(1,-1) -- +(-1,1) node[below,at
start] {$j$}; \fill (-4.5,.5) circle (5pt);
\node at (-2,0){=}; \draw[very thick](0,0) +(-1,-1) -- +(1,1)
node[below,at start] {$i$}; \draw[very thick](0,0) +(1,-1) --
+(-1,1) node[below,at start] {$j$}; \fill (.5,-.5) circle (5pt);
\node at (4,0){ };
\end{tikzpicture}\]
\item\label{rel3}  For two like-labelled strands we get an error term.
\[
\scalefont{0.8}\begin{tikzpicture}[scale=.5,baseline]
\draw[very thick](-4,0) +(-1,-1) -- +(1,1) node[below,at start]
{$i$}; \draw[very thick](-4,0) +(1,-1) -- +(-1,1) node[below,at
start] {$i$}; \fill (-4.5,.5) circle (5pt);
\node at (-2,0){=}; \draw[very thick](0,0) +(-1,-1) -- +(1,1)
node[below,at start] {$i$}; \draw[very thick](0,0) +(1,-1) --
+(-1,1) node[below,at start] {$i$}; \fill (.5,-.5) circle (5pt);
\node at (2,0){+}; \draw[very thick](4,0) +(-1,-1) -- +(-1,1)
node[below,at start] {$i$}; \draw[very thick](4,0) +(0,-1) --
+(0,1) node[below,at start] {$i$};
\end{tikzpicture}  \quad \quad \quad
\scalefont{0.8}\begin{tikzpicture}[scale=.5,baseline]
\draw[very thick](-4,0) +(-1,-1) -- +(1,1) node[below,at start]
{$i$}; \draw[very thick](-4,0) +(1,-1) -- +(-1,1) node[below,at
start] {$i$}; \fill (-4.5,-.5) circle (5pt);
\node at (-2,0){=}; \draw[very thick](0,0) +(-1,-1) -- +(1,1)
node[below,at start] {$i$}; \draw[very thick](0,0) +(1,-1) --
+(-1,1) node[below,at start] {$i$}; \fill (.5,.5) circle (5pt);
\node at (2,0){+}; \draw[very thick](4,0) +(-1,-1) -- +(-1,1)
node[below,at start] {$i$}; \draw[very thick](4,0) +(0,-1) --
+(0,1) node[below,at start] {$i$};
\end{tikzpicture}\]
\item\label{rel4} For double crossings of black strands, we have the following.
\[
\scalefont{0.8}\begin{tikzpicture}[very thick,scale=0.6,baseline]
\draw (-2.8,-1) .. controls (-1.2,0) ..  +(0,2)
node[below,at start]{$i$};
\draw (-1.2,-1) .. controls (-2.8,0) ..  +(0,2) node[below,at start]{$i$};
\node at (-.5,0) {=};
\node at (0.4,0) {$0$};
\end{tikzpicture}
\hspace{.7cm}
\scalefont{0.8}\begin{tikzpicture}[very thick,scale=0.6,baseline]
\draw (-2.8,-1) .. controls (-1.2,0) ..  +(0,2)
node[below,at start]{$i$};
\draw (-1.2,-1) .. controls (-2.8,0) ..  +(0,2)
node[below,at start]{$j$};
\node at (-.5,0) {=}; 

\draw (1.8,-1) -- +(0,2) node[below,at start]{$j$};
\draw (1,-1) -- +(0,2) node[below,at start]{$i$}; 
\end{tikzpicture}
\]
\end{enumerate}
\begin{enumerate}[resume, label=(1.\arabic*)]  
\item\label{rel5} If $j\neq i-1$,  then we can pass ghosts through black strands.
\[\begin{tikzpicture}[very thick,xscale=1,yscale=0.6,baseline]
 \draw (1,-1) to[in=-90,out=90]  node[below, at start]{$i$} (1.5,0) to[in=-90,out=90] (1,1)
;
\draw[dashed] (1.5,-1) to[in=-90,out=90] (1,0) to[in=-90,out=90] (1.5,1);
\draw (2.5,-1) to[in=-90,out=90]  node[below, at start]{$j$} (2,0) to[in=-90,out=90] (2.5,1);
\node at (3,0) {=};
\draw (3.7,-1) -- (3.7,1) node[below, at start]{$i$};
\draw[dashed] (4.2,-1) to (4.2,1);
\draw (5.2,-1) -- (5.2,1) node[below, at start]{$j$};
\end{tikzpicture} \quad\quad \quad \quad 
\scalefont{0.8}\begin{tikzpicture}[very thick,xscale=1,yscale=0.6,baseline]
\draw[dashed] (1,-1) to[in=-90,out=90] (1.5,0) to[in=-90,out=90] (1,1);
\draw (1.5,-1) to[in=-90,out=90] node[below, at start]{$i$} (1,0) to[in=-90,out=90] (1.5,1);
\draw (2,-1) to[in=-90,out=90]  node[below, at start]{$\tiny j$} (2.5,0) to[in=-90,out=90] (2,1);
\node at (3,0) {=};
\draw (3.7,-1) -- (3.7,1) node[below, at start]{$i$};
\draw[dashed] (4.2,-1) to (4.2,1);
\draw (5.2,-1) -- (5.2,1) node[below, at start]{$j$};
\end{tikzpicture}
\]
\end{enumerate} \begin{enumerate}[resume, label=(1.\arabic*)]  
\item\label{rel6} On the other hand, in the case where $j= i-1$, we have the following.
  \[\scalefont{0.8}\begin{tikzpicture}[very thick,xscale=1,yscale=0.6,baseline]
 \draw (1,-1) to[in=-90,out=90]  node[below, at start]{$i$} (1.5,0) to[in=-90,out=90] (1,1)
;
  \draw[dashed] (1.5,-1) to[in=-90,out=90] (1,0) to[in=-90,out=90] (1.5,1);
  \draw (2.5,-1) to[in=-90,out=90]  node[below, at start]{$\tiny i\!-\!1$} (2,0) to[in=-90,out=90] (2.5,1);
\node at (3,0) {=};
  \draw (3.7,-1) -- (3.7,1) node[below, at start]{$i$}
 ;
  \draw[dashed] (4.2,-1) to (4.2,1);
  \draw (5.2,-1) -- (5.2,1) node[below, at start]{$\tiny i\!-\!1$} node[midway,fill,inner sep=2.5pt,circle]{};
\node at (5.75,0) {$-$};

  \draw (6.2,-1) -- (6.2,1) node[below, at start]{$i$} node[midway,fill,inner sep=2.5pt,circle]{};
  \draw[dashed] (6.7,-1)-- (6.7,1);
  \draw (7.7,-1) -- (7.7,1) node[below, at start]{$\tiny i\!-\!1$};
\end{tikzpicture}\]
   
\item\label{rel7} We also have the relation below, obtained by symmetry.  \[
\scalefont{0.8}\begin{tikzpicture}[very thick,xscale=1,yscale=0.6,baseline]
 \draw[dashed] (1,-1) to[in=-90,out=90] (1.5,0) to[in=-90,out=90] (1,1)
;
  \draw (1.5,-1) to[in=-90,out=90] node[below, at start]{$i$} (1,0) to[in=-90,out=90] (1.5,1);
  \draw (2,-1) to[in=-90,out=90]  node[below, at start]{$\tiny i\!-\!1$} (2.5,0) to[in=-90,out=90] (2,1);
\node at (3,0) {=};
  \draw[dashed] (3.7,-1) -- (3.7,1);
  \draw (4.2,-1) -- (4.2,1) node[below, at start]{$i$};
  \draw (5.2,-1) -- (5.2,1) node[below, at start]{$\tiny i\!-\!1$} node[midway,fill,inner sep=2.5pt,circle]{};
\node at (5.75,0) {$-$};

  \draw[dashed] (6.2,-1) -- (6.2,1);
  \draw (6.7,-1)-- (6.7,1) node[midway,fill,inner sep=2.5pt,circle]{} node[below, at start]{$i$};
  \draw (7.7,-1) -- (7.7,1) node[below, at start]{$\tiny i\!-\!1$};
\end{tikzpicture}\]
\end{enumerate}
  \begin{enumerate}[resume, label=(1.\arabic*)] \item\label{rel8} Strands can move through crossings of black strands freely.
\[
\scalefont{0.8}\begin{tikzpicture}[very thick,scale=0.6 ,baseline]
\draw (-2,-1) -- +(-2,2) node[below,at start]{$k$};
\draw (-4,-1) -- +(2,2) node[below,at start]{$i$};
\draw (-3,-1) .. controls (-4,0) ..  +(0,2)
node[below,at start]{$j$};
\node at (-1,0) {=};
\draw (2,-1) -- +(-2,2)
node[below,at start]{$k$};
\draw (0,-1) -- +(2,2)
node[below,at start]{$i$};
\draw (1,-1) .. controls (2,0) ..  +(0,2)
node[below,at start]{$j$};
\end{tikzpicture}
\]
\end{enumerate}
\indent Similarly, this holds for triple points involving ghosts, 
except for the following relations when $j=i-1$.
\begin{enumerate}[resume, label=(1.\arabic*)]  \Item\label{rel9}
\[
\scalefont{0.8}\begin{tikzpicture}[very thick,xscale=1,yscale=0.6,baseline]
\draw[dashed] (-2.6,-1) -- +(-.8,2);
\draw[dashed] (-3.4,-1) -- +(.8,2); 
\draw (-1.1,-1) -- +(-.8,2)
node[below,at start]{$\tiny j$};
\draw (-1.9,-1) -- +(.8,2)
node[below,at start]{$\tiny j$}; 
\draw (-3,-1) .. controls (-3.5,0) ..  +(0,2)
node[below,at start]{$i$};

\node at (-.75,0) {=};

\draw[dashed] (.4,-1) -- +(-.8,2);
\draw[dashed] (-.4,-1) -- +(.8,2);
\draw (1.9,-1) -- +(-.8,2)
node[below,at start]{$\tiny j$};
\draw (1.1,-1) -- +(.8,2)
node[below,at start]{$\tiny j$};
\draw (0,-1) .. controls (.5,0) ..  +(0,2)
node[below,at start]{$i$};

\node at (2.25,0) {$-$};

\draw (4.9,-1) -- +(0,2)
node[below,at start]{$\tiny j$};
\draw (4.1,-1) -- +(0,2)
node[below,at start]{$\tiny j$};
\draw[dashed] (3.4,-1) -- +(0,2);
\draw[dashed] (2.6,-1) -- +(0,2);
\draw (3,-1) -- +(0,2) node[below,at start]{$i$};
\end{tikzpicture}
\]
\Item\label{rel10}
\[
\scalefont{0.8}\begin{tikzpicture}[very thick,xscale=1,yscale=0.6,baseline]
\draw[dashed] (-3,-1) .. controls (-3.5,0) ..  +(0,2);  
\draw (-2.6,-1) -- +(-.8,2)
node[below,at start]{$i$};
\draw (-3.4,-1) -- +(.8,2)
node[below,at start]{$i$};
\draw (-1.5,-1) .. controls (-2,0) ..  +(0,2)
node[below,at start]{$\tiny j$};

\node at (-.75,0) {=};

\draw (0.4,-1) -- +(-.8,2)
node[below,at start]{$i$};
\draw (-.4,-1) -- +(.8,2)
node[below,at start]{$i$};
\draw[dashed] (0,-1) .. controls (.5,0) ..  +(0,2);
\draw (1.5,-1) .. controls (2,0) ..  +(0,2)
node[below,at start]{$\tiny j$};

\node at (2.25,0) {$+$};

\draw (3.4,-1) -- +(0,2)
node[below,at start]{$i$};
\draw (2.6,-1) -- +(0,2)
node[below,at start]{$i$}; 
\draw[dashed] (3,-1) -- +(0,2);
\draw (4.5,-1) -- +(0,2)
node[below,at start]{$\tiny j$};
\end{tikzpicture}
\]
\end{enumerate}
In the diagrams with crossings in \ref{rel9} and \ref{rel10}, we say that the black (respectively ghost) strand bypasses the crossing of ghost strands (respectively black strands).
 The ghost strands may pass through red strands freely.
  For $i\neq j$, the black $i$-strands may pass through red $j$-strands freely. If the red and black strands have the same label, a dot is added to the black strand when straightening.
\begin{enumerate}[resume, label=(1.\arabic*)] \Item\label{rel11}
\[
\scalefont{0.8}\begin{tikzpicture}[very thick,baseline,scale=0.6]
\draw (-2.8,-1)  .. controls (-1.2,0) ..  +(0,2)
node[below,at start]{$i$};
\draw[wei] (-1.2,-1)  .. controls (-2.8,0) ..  +(0,2)
node[below,at start]{$i$};

\node at (-.3,0) {=};

\draw[wei] (2.8,-1) -- +(0,2)
node[below,at start]{$i$};
\draw (1.2,-1) -- +(0,2)
node[below,at start]{$i$};
\fill (1.2,0) circle (3pt);

\draw[wei] (6.8,-1)  .. controls (5.2,0) ..  +(0,2)
node[below,at start]{$j$};
\draw (5.2,-1)  .. controls (6.8,0) ..  +(0,2)
node[below,at start]{$i$};

\node at (7.7,0) {=};

\draw (9.2,-1) -- +(0,2)
node[below,at start]{$i$};
\draw[wei] (10.8,-1) -- +(0,2)
node[below,at start]{$j$};
\end{tikzpicture}
\]
\end{enumerate}
and their mirror images. All black crossings and dots can pass through  red strands, with a
correction term.
\begin{enumerate}[resume, label=(1.\arabic*)]
\Item\label{rel12}
\[
\scalefont{0.8}\begin{tikzpicture}[very thick,baseline,scale=0.6]
\draw (-2,-1)  -- +(-2,2)
node[at start,below]{$i$};
\draw (-4,-1) -- +(2,2)
node [at start,below]{$j$};
\draw[wei] (-3,-1) .. controls (-4,0) ..  +(0,2)
node[at start,below]{$k$};

\node at (-1,0) {=};

\draw (2,-1) -- +(-2,2)
node[at start,below]{$i$};
\draw (0,-1) -- +(2,2)
node [at start,below]{$j$};
\draw[wei] (1,-1) .. controls (2,0) .. +(0,2)
node[at start,below]{$k$};

\node at (2.8,0) {$+ $};

\draw (7.5,-1) -- +(0,2)
node[at start,below]{$i$};
\draw (5.5,-1) -- +(0,2)
node [at start,below]{$j$};
\draw[wei] (6.5,-1) -- +(0,2)
node[at start,below]{$k$};
\node at (3.8,-.2){$\delta_{i,j,k}$};
\end{tikzpicture}
\]
\Item\label{rel13}
\[
\scalefont{0.8}\begin{tikzpicture}[scale=0.6,very thick,baseline=2cm]
\draw[wei] (-2,2) -- +(-2,2);
\draw (-3,2) .. controls (-4,3) ..  +(0,2);
\draw (-4,2) -- +(2,2);

\node at (-1,3) {=};

\draw[wei] (2,2) -- +(-2,2);
\draw (1,2) .. controls (2,3) ..  +(0,2);
\draw (0,2) -- +(2,2);

\draw (6,2) -- +(-2,2);
\draw (5,2) .. controls (6,3) ..  +(0,2);
\draw[wei] (4,2) -- +(2,2);

\node at (7,3) {=};

\draw (10,2) -- +(-2,2);
\draw (9,2) .. controls (10,3) ..  +(0,2);
\draw[wei] (8,2) -- +(2,2);
\end{tikzpicture}
\]
\Item\label{rel14}
\[
\scalefont{0.8}\begin{tikzpicture}[very thick,baseline,scale=0.6]
  \draw(-3,0) +(-1,-1) -- +(1,1);
  \draw[wei](-3,0) +(1,-1) -- +(-1,1);
\fill (-3.5,-.5) circle (3pt);
\node at (-1,0) {=};
 \draw(1,0) +(-1,-1) -- +(1,1);
  \draw[wei](1,0) +(1,-1) -- +(-1,1);
\fill (1.5,.5) circle (3pt);

\draw[wei](1,0) +(3,-1) -- +(5,1);
  \draw(1,0) +(5,-1) -- +(3,1);
\fill (5.5,-.5) circle (3pt);
\node at (7,0) {=};
 \draw[wei](5,0) +(3,-1) -- +(5,1);
  \draw(5,0) +(5,-1) -- +(3,1);
\fill (8.5,.5) circle (3pt);
    \end{tikzpicture}
    \]
  \end{enumerate}
Finally, we have the following non-local idempotent relation.
\begin{enumerate}[resume, label=(1.\arabic*)]
\item\label{rel15}
Any idempotent  where the strands can be broken into two
groups separated by a blank space of size $>g$ (so no ghost from
the right-hand  group can be left of a strand in the left group and {\it vice versa}) with all red strands in the right-hand group is referred to as unsteady and set to be equal to zero.
\end{enumerate}
\end{defn}

\subsection{The grading on the diagrammatic Cherednik algebra}\label{grsubsec}
This algebra is graded as follows:
\begin{itemize}
\item
 dots have degree 2;
\item
 the crossing of two strands has degree 0, unless they have the
  same label, in which case it has degree $-2$;
\item the crossing of a black strand with label $i$ and a ghost has degree 1 if the ghost has label $i-1$ and 0 otherwise;
   \item
  the crossing of a black strand with a red strand has degree 0, unless they have the same label, in which case it has degree 1.
\end{itemize}
In other words,
\[
\deg\tikz[baseline,very thick,scale=1.5]{\draw
  (0,.3) -- (0,-.1) node[at end,below,scale=.8]{$i$}
  node[midway,circle,fill=black,inner
  sep=2pt]{};}=
  2 \qquad \deg\tikz[baseline,very thick,scale=1.5]
  {\draw (.2,.3) --
    (-.2,-.1) node[at end,below, scale=.8]{$i$}; \draw
    (.2,-.1) -- (-.2,.3) node[at start,below,scale=.8]{$j$};} =-2\delta_{i,j} \qquad  
  \deg\tikz[baseline,very thick,scale=1.5]{\draw[densely dashed] 
  (-.2,-.1)-- (.2,.3) node[at start,below, scale=.8]{$i$}; \draw
  (.2,-.1) -- (-.2,.3) node[at start,below,scale=.8]{$j$};} =\delta_{j,i+1} \qquad \deg\tikz[baseline,very thick,scale=1.5]{\draw (.2,.3) --
  (-.2,-.1) node[at end,below, scale=.8]{$i$}; \draw [densely dashed]
  (.2,-.1) -- (-.2,.3) node[at start,below,scale=.8]{$j$};} =\delta_{j,i-1}\]
\[
  \deg\tikz[baseline,very thick,scale=1.5]{ \draw[wei]
  (-.2,-.1)-- (.2,.3) node[at start,below, scale=.8]{$i$}; \draw
  (.2,-.1) -- (-.2,.3) node[at start,below,scale=.8]{$j$};} =\delta_{i,j} 
   \qquad \deg\tikz[baseline,very thick,scale=1.5] {\draw (.2,.3) --
  (-.2,-.1) node[at end,below, scale=.8]{$i$}; \draw[wei]   (.2,-.1) -- (-.2,.3) node[at start,below,scale=.8]{$j$};} =\delta_{j,i}.
\]

 \subsection{Representation theory of the diagrammatic Cherednik algebra}
Let $d$ be any $\theta$-diagram and  $y\in[0,1]$ be any fixed value such that 
there are no  crossings in $d$ at any point in $\mathbb{R}\times \{y\}$.
Then the positions of the various strands in this horizontal slice give a loading $\mathbf{i}_y$.
We say that the diagram $d\in A(n,\theta,\kappa)$ \emph{factors through} the loading $\mathbf{i}_y$.

The following lemma is a trivial consequence of the proof of \cite[Lemma 2.15]{Webster}.

\begin{lem} \label{215}
If a $\theta$-diagram, $d$,   factors through 
a loading $\mathbf{i}$ such that $\mathbf{i} \vartriangleright \mathbf{i}_\mu$ for some $\mu \in \mptn ln$, with $\mathbf{i}$ and $\mathbf{i}_\mu$ not isotopic, then $d$ factors through some $\mathbf{i}_\lambda$ such that $\mathbf{i}_\lambda \vartriangleright \mathbf{i}$, for some $\lambda \in \mptn ln$.
\end{lem}

Given $\SSTT \in \SStd(\lambda,\mu)$, we have a 
$\theta$-diagram $B_\SSTT$ consisting of a {\it frame} in which the $n$ black strands each connecting a northern and southern distinguished point are drawn so that they trace out the bijection determined by $\SSTT$ in such a way that we use  the minimal number of crossings without creating any bigons between pairs of strands or strands and ghosts. This diagram is not unique up to isotopy (since we have not specified how to resolve triple points), but we can choose one such diagram arbitrarily.

Given a pair of semistandard tableaux of the same shape
$(\SSTS,\SSTT)\in\SStd(\lambda,\mu)\times\SStd(\lambda,\nu)$,  we  have
a diagram  $C_{\SSTS, \SSTT}=B_\SSTS B^\ast_\SSTT$ where $B^\ast_\SSTT$
is the diagram obtained from $B_\SSTT $ by flipping it through the
horizontal axis.
 Notice that there is a unique element   $\SSTT^\lambda\in
\SStd(\lambda,\lambda)$ and the corresponding basis element
$C_{\SSTT^\lambda,\SSTT^\lambda}$ is the idempotent in which all black strands are
vertical. A degree function on tableaux is defined in \cite[Defintion 2.12]{Webster}; 
for our purposes it is enough to note that $\deg(\SSTT)=\deg(B_\SSTT)$ as we shall
always work with the $\theta$-diagrams directly.


\begin{thm}[{\cite[Section 2.6]{Webster}}]
\label{cellularitybreedscontempt}
The algebra $A(n,\theta,\kappa)$ is a graded cellular algebra with a theory of highest weights.
The cellular basis is given by
\[
 \mathcal{C}=\{C_{\SSTS, \SSTT} \mid \SSTS \in \SStd(\lambda,\mu), \SSTT\in \SStd(\lambda,\nu), 
 \lambda,\mu, \nu \in \mptn ln \}
\]
with respect to the $\theta$-dominance order on the set   $\mptn ln$
and the anti-isomorphism given by flipping a diagram through the horizontal axis.  
  \end{thm}

\begin{thm}[\cite{Webster}, Theorem 6.2]\label{step1}
Over $\mathbb C$, the (basic algebra of the) diagrammatic  Cherednik algebra  $A(n,\theta,\kappa)$ is Koszul.
 Over $\mathbb C$, we therefore have that the graded decomposition numbers $d_{\lambda\mu}(t)\in t\NN_0(t)$ for $\lambda\neq \mu \in \Lambda$.
 \end{thm}

\begin{rmk}Notice that  the basis of  $A(n,\theta,\kappa)$ also respects the decomposition of $\mptn ln$ 
by residue sets.  Given a residue set $\mathcal{R}$, we let  $A_\mathcal{R}(n,\theta,\kappa)$ denote the subalgebra of $A(n,\theta,\kappa)$ with basis given by all $\theta$-diagrams indexed by multipartitions $\lambda,\mu,\nu \in \mptn ln(\mathcal{R})$.  
\end{rmk}

 \section{Nested sign sequences}\label{nestsec}
In this section we recall the combinatorics of \cite{tanteo13} for calculating graded decomposition numbers.   We  include several illustrative examples.   
 We fix $i\in (\ZZ/e\ZZ)$ throughout.  
Given  $ \mu\in \mptn ln$, $\kappa \in (\ZZ/e\ZZ)^l$ and $\theta\in \RR^l$, 
   we read the loading $\mathbf{i}_{\mu }$ from left to right
   and record any addable and removable $i$-nodes in order and associate the following path  
\[
\begin{cases}
\diagup  &  \text{for each  removable $i$-node  of $[\mu]$}; \\
\diagdown  &\text{for each addable $i$-node  of $[\mu]$}.  
\end{cases}
\]
Connect these   line segments   in order, to obtain the path $\mathcal{P}(\mu)$, which we refer to as the \emph{terrain} of $\mu$.  
We place a vertex at each point where two line segments meet.  
If the  $j$th edge of $\mathcal{P}(\mu)$ is of the form $\diagup$
and the $(j+1)$th edge is of the form $\diagdown$ then 
we refer to the $j$th and $(j+1)$th edges  as a \emph{ridge in the terrain} of $\mu$.
 If the  $j$th edge of $\mathcal{P}(\mu)$ is of the form $\diagdown$
and the $(j+1)$th edge is of the form $\diagup$ then 
we refer to the $j$th and $(j+1)$th edges  as a \emph{valley in the terrain} of $\mu$.

\begin{eg}
Let $l=6$ and $n=3$ and let $\nu$ denote the $l$-multipartition 
$((1),(1),\varnothing,\varnothing,\varnothing,(1))$ for  $\kappa=(1,1,1,1,1,1)$ and some well-separated weighting $\theta\in\RR^6$.  The terrain of $\mu$ is given by the leftmost diagram in \cref{Diagram1}.
There is a ridge between the second and third edges and a valley between the fifth and sixth edges.

Let $l=10$ and $n=6$ and let $\mu$ denote the $l$-multipartition 
$((1),(1), \varnothing,\varnothing,(1),\varnothing,(1), \varnothing,(1), (1))$
 $\kappa=(1,1,1,1,1,1,1,1,1,1)$  and 
some well-separated weighting $\theta\in\RR^{10}$.    
 The terrain of $\mu$ is given by the  rightmost diagram in \cref{Diagram1}.

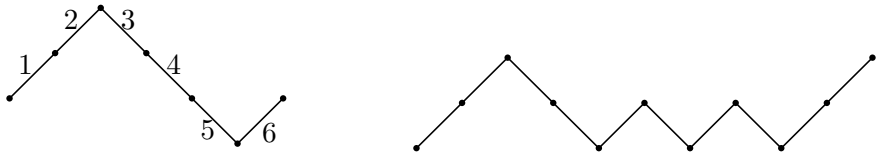
\begin{figure}[ht!]
\begin{center}
\begin{tikzpicture}[scale=0.6]
\draw [semithick] (0,0) -- (1,1) -- (2,2) -- (3,1) -- (4,0) -- (5,-1) -- (6,0);
\fill (0,0) circle (2pt);
\fill (1,1) circle (2pt);
\fill (2,2) circle (2pt);
\fill (3,1) circle (2pt);
\fill (4,0) circle (2pt);
\fill (5,-1) circle (2pt);
\fill (6,0) circle (2pt);
\draw (0.35,0.75) node {$1$};
\draw (1.35,1.75) node {$2$};
\draw (2.6,1.75) node {$3$};
\draw (3.6,0.75) node {$4$};
\draw (4.35,-0.7) node {$5$};
\draw (5.7,-0.75) node {$6$};
\end{tikzpicture}
\quad\quad\quad\quad
\begin{tikzpicture}[scale=0.6]
\draw [semithick] (0,0) -- (1,1) -- (2,2) -- (3,1) -- (4,0) -- (5,1) -- (6,0)-- (7,1)-- (8,0) -- (9,1)-- (10,2);
\fill (0,0) circle (2pt);
\fill  (1,1) circle (2pt);
\fill (2,2) circle (2pt);
\fill (3,1) circle (2pt);
\fill (4,0) circle (2pt);
\fill (5,1) circle (2pt);
\fill (6,0) circle (2pt);
\fill (7,1) circle (2pt);
\fill (8,0) circle (2pt);
\fill (9,1) circle (2pt);
\fill (10,2) circle (2pt);
 \end{tikzpicture}
\end{center}
\caption{Examples of the terrain of a multipartition.}
\label{Diagram1}
\end{figure}\end{eg}

Let  $ \mu,\lambda \in \mptn l n$ and suppose the $\mu  \lhd_\theta  \lambda$.  
We shall add a $\lambda$-decoration to the  terrain of $\mu$
  (denoted $\mathcal{P}(\mu,\lambda)$)
  as follows.  Let $A$ (respectively $R$) denote the  
set of nodes in $\lambda\setminus (\mu\cap \lambda)$ (respectively  $\mu\setminus (\mu\cap \lambda)$); in other words the set of nodes added to and removed from  $\mu$ to obtain the multipartition $\lambda$.  
 Associate to each  edge in $A$  
 an opening parenthesis and to each edge in $R$ a closing parenthesis.   This defines a natural pairing on the sets $A$ and $R$ according to the system of nested parentheses (that these parentheses form a nesting follows from the definition of the $\theta$-dominance order).  
 
 We identify a pair of parentheses with the edges   at which they open and close.  Given   $P=(j_1,j_2), Q=(k_1,k_2) \in \mathcal{P}( \mu,\lambda) $, we write $P\subset Q$ if $k_1<j_1$ and $ j_2<k_2$ and refer to this order as \emph{inclusion}.  We let $\mathcal{Q}(\mu,\lambda)$ denote the partially ordered set of 
pairs of parentheses  on $\mathcal{P}(\mu,\lambda)$ under inclusion.  

\begin{eg}
Let $l=10$ and $n=6$ and let 
$\mu=((1),(1), \varnothing,\varnothing,(1),\varnothing,(1), \varnothing,(1), (1))$
$\lambda=((1),(1),(1),(1),\varnothing,(1),(1),\varnothing,\varnothing,\varnothing)$ 
 The $\lambda$-decorated   $\mu$-terrain is given by the  diagram in \cref{Decoration}.

\begin{figure}[ht!]
\begin{center}
\scalefont{1.2}
\begin{tikzpicture}[scale=0.6]
\draw [semithick] (0,0) -- (1,1) -- (2,2) -- (3,1) -- (4,0) -- (5,1) -- (6,0)-- (7,1)-- (8,0) -- (9,1)-- (10,2);
\fill (0,0) circle (2pt);
\fill  (1,1) circle (2pt);
\fill (2,2) circle (2pt);
\fill (3,1) circle (2pt);
\fill (4,0) circle (2pt);
\fill (5,1) circle (2pt);
\fill (6,0) circle (2pt);
\fill (7,1) circle (2pt);
\fill (8,0) circle (2pt);
\fill (9,1) circle (2pt);
\fill (10,2) circle (2pt);
\draw (2.5,1.5) node {$\mathbf{(}$};
\draw (3.5,0.5) node {$\mathbf{(}$};
\draw (4.5,0.5) node {$\mathbf{)}$};
\draw (5.5,0.5) node {$\mathbf{(}$};
\draw (8.5,0.5) node {$\mathbf{)}$};
\draw (9.5,1.5) node {$\mathbf{)}$};
 \end{tikzpicture}
\end{center}
\caption{The terrain of $\mu$ decorated with multipartition $\lambda$.}
\label{Decoration}
\end{figure}
 \end{eg}
 
%

We shall now consider paths which may be obtained from $\mathcal{P}(\mu)$ by replacing up and down edges with horizontal line segments.  This requires us to slightly generalise the definition of a ridge, as follows. 
  If the  $j$th edge of $\mathcal{P}(\mu)$ is of the form $\diagup$  
 and  the $k$th edge of $\mathcal{P}(\mu)$ is of the form $\diagdown$  and 
  the  edges strictly between $j$ and $k$ are all horizontal line segments,
   then 
  we refer to this pair of edges  as a \emph{flattened ridge}.

\begin{defn}
Given $\mu,\lambda\in \mptn ln$, fix a pair of parentheses  $P \in \mathcal{Q}(\mu,\lambda)$;   
the set of {\em latticed paths of shape $\mu$  and decoration $\lambda$ with respect to the pair}    $P $ is the set of all  possible ways of replacing 
 some number of ridges formed 
 of edges \emph{strictly}
   between  the parentheses
 to obtain flattened ridges.  

We place an ordering on such paths by writing $\rho\leq \rho'$ if 
the $y$-coordinate of every  vertex in $\rho$ is less than or equal to  the 
$y$-coordinate of the corresponding vertex in $\rho'$.
 
Given $P\in\mathcal{Q}(\mu,\lambda)$  and $\rho$ a latticed path of shape $\mu$  and decoration $\lambda$, we say that   $\rho$ has norm, $\|\rho\|$,  given by the number of non-flattened steps strictly  between the fixed pair of brackets plus 1.  We refer to the unique path of maximal norm (in which no ridges are flattened) as the \emph{generic} latticed path.   
\end{defn}

\begin{eg}\label{yetanotherexample}
Suppose that $\mathcal{P}(\mu,\lambda)$ is as in \cref{Decoration}.  
There  are no  ridges strictly between the  pairs of parentheses $(4,5)$ and so the set of latticed paths consists only of the generic path.
There is a single ridge strictly between   $(6,9)$.  
Therefore  there are two distinct latticed paths $\rho$ with respect to  
$(6,9)$.  Namely,  the generic path $\mathcal{P}(\mu,\lambda)$ in Figure \ref{Decoration}
and the path in which we flatten the ridge between $(6,9)$; these are depicted in \cref{pathspostliron}.  These paths have  norms  3 and 1, respectively.

\begin{figure}[ht!]
\begin{center}
\scalefont{1.2}
\begin{tikzpicture}[scale=0.6]
\draw [semithick] (0,0) -- (1,1) -- (2,2) -- (3,1) -- (4,0) -- (5,1) -- (6,0)-- (7,1)-- (8,0) -- (9,1)-- (10,2);
\fill (0,0) circle (2pt);
\fill  (1,1) circle (2pt);
\fill (2,2) circle (2pt);
\fill (3,1) circle (2pt);
\fill (4,0) circle (2pt);
\fill (5,1) circle (2pt);
\fill (6,0) circle (2pt);
\fill (7,1) circle (2pt);
\fill (8,0) circle (2pt);
\fill (9,1) circle (2pt);
\fill (10,2) circle (2pt);
\draw (2.5,1.5) node {$\mathbf{(}$};
\draw (3.5,0.5) node {$\mathbf{(}$};
\draw (4.5,0.5) node {$\mathbf{)}$};
\draw (5.5,0.5) node {$\mathbf{(}$};
\draw (8.5,0.5) node {$\mathbf{)}$};
\draw (9.5,1.5) node {$\mathbf{)}$};
 \end{tikzpicture}
\quad
\begin{tikzpicture}[scale=0.6]
\draw [semithick] (0,0) -- (1,1) -- (2,2) -- (3,1) -- (4,0) -- (5,1) -- (6,0)-- (7,0) --(8,0) -- (9,1)-- (10,2);
\draw [dotted] (6,0) --(7,1)-- (8,0); 

\fill (0,0) circle (2pt);
\fill  (1,1) circle (2pt);
\fill (2,2) circle (2pt);
\fill (3,1) circle (2pt);
\fill (4,0) circle (2pt);
\fill (5,1) circle (2pt);
\fill (6,0) circle (2pt);
\fill (7,0) circle (2pt);
\fill (8,0) circle (2pt);
\fill (9,1) circle (2pt);
\fill (10,2) circle (2pt);
\draw (2.5,1.5) node {$\mathbf{(}$};
\draw (3.5,0.5) node {$\mathbf{(}$};
\draw (4.5,0.5) node {$\mathbf{)}$};
\draw (5.5,0.5) node {$\mathbf{(}$};
\draw (8.5,0.5) node {$\mathbf{)}$};
\draw (9.5,1.5) node {$\mathbf{)}$};
 \end{tikzpicture}
\end{center}
\caption{The latticed paths of shape $\mu$ and decoration  $\lambda$ with respect to the pair of parentheses
$(6,9)$.  These have norms 3 and 1 respectively.}
\label{pathspostliron}
\end{figure}
\end{eg}

\begin{defn} 
 A {\em well-nested latticed path for $\mathcal{P}(\mu,\lambda)$} is  a
 collection $\{\rho_P \mid P\in \mathcal{Q}(\mu,\lambda)\}$ of latticed paths  
  such that if  $P,Q\in\mathcal{Q}( \mu,\lambda)$  and $P\subset Q$, then $\rho_P \geq \rho_Q$.  We let $\Omega(\mu,\lambda)$ denote the set of all well-nested latticed paths.
 The norm of a well-nested latticed path is given by the sum of the norms of the constituent paths.
\end{defn}

\begin{eg}\label{newbyiui}Now consider $\mathcal{P}( \mu,\lambda)$  and the pair of parentheses given by $(3,10)$.  There are three  ridges between the  pair  of parentheses $(3,10)$.  
The set of all latticed paths with respect to the pair $(3,10)$
which are well-nested with respect to the rightmost 
latticed path in \cref{pathspostliron}
 is depicted in \cref{figure} below.
 The  paths  in \cref{figure} are also well-nested with respect to the leftmost 
latticed path in \cref{pathspostliron};  we also obtain a further 
 two additional paths with respect to the pair $(3,10)$, which are well-nested with respect to the leftmost diagram -- these are depicted in \cref{figurewhosurdaddy}.

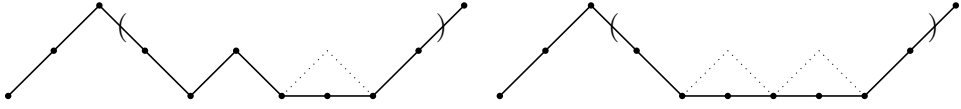
\begin{figure}[ht!]
\[
\begin{tikzpicture}[scale=0.6]
\draw [semithick] (0,0) -- (1,1) -- (2,2) -- (3,1) -- (4,0) -- (5,1) -- (6,0)-- (7,0) --(8,0) -- (9,1)-- (10,2);
\draw [dotted] (6,0) --(7,1)-- (8,0); 

\fill (0,0) circle (2pt);
\fill  (1,1) circle (2pt);
\fill (2,2) circle (2pt);
\fill (3,1) circle (2pt);
\fill (4,0) circle (2pt);
\fill (5,1) circle (2pt);
\fill (6,0) circle (2pt);
\fill (7,0) circle (2pt);
\fill (8,0) circle (2pt);
\fill (9,1) circle (2pt);
\fill (10,2) circle (2pt);
\draw (2.5,1.5) node {$\mathbf{(}$};
\draw (9.5,1.5) node {$\mathbf{)}$};
 \end{tikzpicture}
 \quad
 \begin{tikzpicture}[scale=0.6]
\draw [semithick] (0,0) -- (1,1) -- (2,2) -- (3,1) -- (4,0) -- (5,0) -- (6,0)-- (7,0) --(8,0) -- (9,1)-- (10,2);
\draw [dotted] (6,0) --(7,1)-- (8,0); 
\draw [dotted] (4,0) --(5,1)-- (6,0); 

\fill (0,0) circle (2pt);
\fill  (1,1) circle (2pt);
\fill (2,2) circle (2pt);
\fill (3,1) circle (2pt);
\fill (4,0) circle (2pt);
\fill (5,0) circle (2pt);
\fill (6,0) circle (2pt);
\fill (7,0) circle (2pt);
\fill (8,0) circle (2pt);
\fill (9,1) circle (2pt);
\fill (10,2) circle (2pt);
\draw (2.5,1.5) node {$\mathbf{(}$};
 \draw (9.5,1.5) node {$\mathbf{)}$};
 \end{tikzpicture}
\]
\caption{The set of all latticed paths with respect to the pair $(3,10)$ 
which are well-nested with respect to the rightmost  path in \cref{pathspostliron}.
These paths have norms 5 and 3 respectively.  
}
 \label{figure}
\end{figure}

\begin{figure}[ht!]
\[
\begin{tikzpicture}[scale=0.6]
\draw [semithick] (0,0) -- (1,1) -- (2,2) -- (3,1) -- (4,0) -- (5,1) -- (6,0)-- (7,1) --(8,0) -- (9,1)-- (10,2);
 
\fill (0,0) circle (2pt);
\fill  (1,1) circle (2pt);
\fill (2,2) circle (2pt);
\fill (3,1) circle (2pt);
\fill (4,0) circle (2pt);
\fill (5,1) circle (2pt);
\fill (6,0) circle (2pt);
\fill (7,1) circle (2pt);
\fill (8,0) circle (2pt);
\fill (9,1) circle (2pt);
\fill (10,2) circle (2pt);
\draw (2.5,1.5) node {$\mathbf{(}$};
\draw (9.5,1.5) node {$\mathbf{)}$};
 \end{tikzpicture}
 \quad
 \begin{tikzpicture}[scale=0.6]
\draw [semithick] (0,0) -- (1,1) -- (2,2) -- (3,1) -- (4,0) -- (5,0) -- (6,0)-- (7,1) --(8,0) -- (9,1)-- (10,2);
 \draw [dotted] (4,0) --(5,1)-- (6,0); 

\fill (0,0) circle (2pt);
\fill  (1,1) circle (2pt);
\fill (2,2) circle (2pt);
\fill (3,1) circle (2pt);
\fill (4,0) circle (2pt);
\fill (5,0) circle (2pt);
\fill (6,0) circle (2pt);
\fill (7,1) circle (2pt);
\fill (8,0) circle (2pt);
\fill (9,1) circle (2pt);
\fill (10,2) circle (2pt);
\draw (2.5,1.5) node {$\mathbf{(}$};
 \draw (9.5,1.5) node {$\mathbf{)}$};
 \end{tikzpicture}
\] 
\caption{The additional latticed paths with respect to the pair $(3,10)$ 
which are well-nested with respect to  the leftmost  path in \cref{pathspostliron}.
These paths have norms 7 and 5, respectively.
}
 \label{figurewhosurdaddy}
\end{figure}
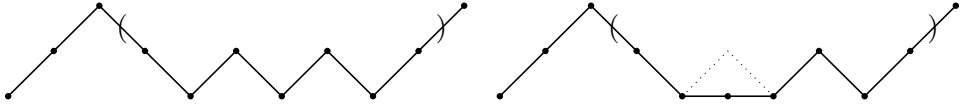
%
   We have that there are a total of  eight triples  of latticed paths (corresponding to the three distinct pairs of parentheses), six of which are well-nested.  We have that
\[
\sum_{\mathclap{\omega \in \Omega(\mu,\lambda)}} \;
t^{\|w\|} = t^{11}+2t^{9}+2t^7+ t^5.
\]
\end{eg}

\section{Cherednik algebras and their subquotients}\label{subqsec}

 In this  section,  
 we shall define the  subquotients  of the diagrammatic  Cherednik algebras in which we shall be interested for the remainder of the paper.  
\begin{defn}
A set of residues $S\subset \ZZ/e\ZZ$, is said to be \emph{adjacency-free} if $i \in S$ implies  $i\pm1\not\in S$.
\end{defn}

\begin{defn}
We say that  $\gamma\in\mptn ln (\mathcal{R})$ is \emph{$S$-admissible} if $\Remb_S(\gamma)=\emptyset$.
For an $S$-admissible $\gamma\in\mptn ln (\mathcal{R})$ and $\mathcal{M}$ a multiset of $S$-residues of  size $m$, we let $\Gamma = \Gamma(\mathcal{M})$ denote the set of all multipartitions which may be obtained from $\gamma$ by adding a set of nodes whose residue multiset is $\mathcal{M}$.
\end{defn}

\begin{eg}\label{admiseg}
Let $e=4$, $g=0.99$, $\kappa = (0,3)$, and $\theta=(0,7)$.
Let $\gamma=((3,2,1^3),(4,2^2,1))$; this bipartition has residue set $\mathcal{R}=\{0^5,1^4,2^5,3^3\}$.

Given  $S=\{1,3\}$ and $\gamma$ as above, we have that $\Remb_S(\gamma)=\emptyset $ and  therefore $\gamma$ is $S$-admissible. Given $\mathcal{M}=\{1,3^3\}$, 
 the set $\Gamma=\Gamma(\mathcal{M})$ consists of the 20 bipartitions which may be obtained by adding a single $1$-node and three $3$-nodes to the bipartition $\gamma$. For example,
we have that  $\alpha=((4,3,1^4),(5,2^2,1))$
and $\beta=((3,2,1^4),(5,2^3,1))$ both belong to $\Gamma$.  
 
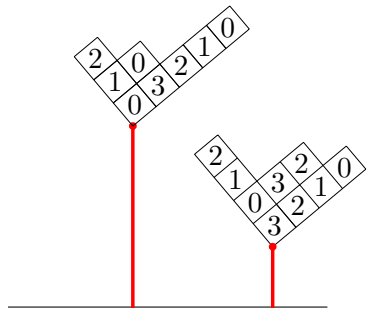
\begin{figure}[ht!]
\[
\begin{tikzpicture}[scale=0.4]
 {\clip (-2.5,-2) rectangle ++(14,11.2);
  \path [draw,name path=upward line] (-4,-2) -- (8,-2); 
         \draw[black,thin] 
      (6.2, 0)              coordinate (X1)
            -- +(40:1) coordinate (X2)  
      (X2)   -- +(130:1) coordinate (X3)
      (X3)   -- +(220:1) coordinate (X4)       
      (X4)   -- +(310:1) coordinate (X5)   ;  
      \draw  (X1)++(40:0.5)++(130:0.5) node {${3}$}   ;  
             \draw[black,thin] 
      (X2)              coordinate (Y1)
            -- +(40:1) coordinate (Y2)  
      (Y2)   -- +(130:1) coordinate (Y3)
      (Y3)   -- +(220:1) coordinate (Y4)       
      (Y4)   -- +(310:1) coordinate (Y5)   ;    
            \draw  (X2)++(40:0.5)++(130:0.5) node {${2}$}   ;  
   \draw[black,thin] 
      (Y2)              
            -- +(40:1) coordinate (YY2)  
      (YY2)   -- +(130:1) coordinate (YY3)
       (YY3)   -- +(220:1) coordinate (YY4)       
 ;         \draw  (Y2)++(40:0.5)++(130:0.5) node {${1}$}   ;  
   \draw[black,thin] 
      (YY2)              
            -- +(40:1) coordinate (YYY2)  
      (YYY2)   -- +(130:1) coordinate (YYY3)
       (YYY3)   -- +(220:1) coordinate (YYY4)       
 ;         \draw  (YY2)++(40:0.5)++(130:0.5) node {${0}$}   ;  
    \draw[black,thin] 
      (X3)              
            -- +(130:1) coordinate (XX3)  
      (XX3)   -- +(220:1) coordinate (XX4)
       (XX4)   -- +(310:1) coordinate (XX5)       
 ;  
       \draw  (X3)++(40:0.5)++(130:0.5)++(220:1) node {${0}$}   ;  
    \draw[black,thin] 
      (XX3)              
            -- +(40:1) coordinate (XY3)  
      (XY3)   -- +(310:1) coordinate (XY4)
 ;  
     \draw  (X3)++(40:0.5)++(130:0.5) node {${3}$}   ;  
      \draw[black,thin] 
      (XY3)              
            -- +(40:1) coordinate (XYY3)  
      (XYY3)   -- +(310:1) coordinate (XYY4)
 ;  
     \draw  (YY4)++(40:0.5)++(130:0.5) node {${2}$}   ;  
      \draw[black,thin] 
      (XX3)              
            -- +(130:1) coordinate (XXX3)  
      (XXX3)   -- +(220:1) coordinate (XXX4)
            (XXX4)   -- +(310:1) coordinate (XXX5)
 ;  
 
     \draw  (XXX5)++(40:0.5)++(130:0.5) node {${1}$}   ;  
     \draw[black,thin] 
      (XXX3)              
            -- +(130:1) coordinate (XXXX3)  
      (XXXX3)   -- +(220:1) coordinate (XXXX4)
            (XXXX4)   -- +(310:1) coordinate (XXXX5)
 ;  
 
     \draw  (XXXX5)++(40:0.5)++(130:0.5) node {${2}$}   ;  
         \draw[wei2](X1) -- +(270:10);
        \draw[wei2](X1)  circle (2pt);
               \draw[wei2](1.6,4) -- +(270:10);       \draw[wei2](1.6,4)  circle (2pt);

     \draw[black,thin] 
      (1.6,4)              coordinate (X1)
            -- +(40:1) coordinate (X2)  
      (X2)   -- +(130:1) coordinate (X3)
      (X3)   -- +(220:1) coordinate (X4)       
      (X4)   -- +(310:1) coordinate (X5)   ;    
 \draw  (X1)++(40:0.5)++(130:0.5) node {${0}$}   ;  
     \draw[black,thin] 
       (X3)
            -- +(40:1) coordinate (VVX2)  
      (VVX2)   -- +(130:1) coordinate (VVX3)
      (VVX3)   -- +(220:1) coordinate (VVX4)       
      (VVX4)   -- +(310:1) coordinate (X5)   ;    
 \draw  (X3)++(40:0.5)++(130:0.5) node {${0}$}   ;  
     \draw[black,thin] 
      (X4) -- +(40:1) coordinate (XX2)  
      (XX2)   -- +(130:1) coordinate (XX3)
      (XX3)   -- +(220:1) coordinate (XX4)       
      (XX4)   -- +(310:1) coordinate (XX5)   ;    
 \draw  (X4)++(40:0.5)++(130:0.5) node {${1}$}   ;  
 \draw[black,thin] 
      (XX4) -- +(40:1) coordinate (XXX2)  
      (XXX2)   -- +(130:1) coordinate (XXX3)
      (XXX3)   -- +(220:1) coordinate (XXX4)       
      (XXX4)   -- +(310:1) coordinate (XXX5)   ;    
 \draw  (XX4)++(40:0.5)++(130:0.5) node {${2}$}   ;  

               \draw[black,thin] 
      (X2)              coordinate (Y1)
            -- +(40:1) coordinate (Y2)  
      (Y2)   -- +(130:1) coordinate (Y3)
      (Y3)   -- +(220:1) coordinate (Y4)       
      (Y4)   -- +(310:1) coordinate (Y5)   ;
       \draw  (Y1)++(40:0.5)++(130:0.5) node {${3}$} ; 
                \draw[black,thin] 
      (Y2)              coordinate (Y1)
            -- +(40:1) coordinate (YY2)  
      (YY2)   -- +(130:1) coordinate (YY3)
      (YY3)   -- +(220:1) coordinate (YY4)       
      (YY4)   -- +(310:1) coordinate (YY5)   ;
       \draw  (Y1)++(40:0.5)++(130:0.5) node {${2}$} ; 
               \draw[black,thin] 
      (YY2)              coordinate (YY1)
            -- +(40:1) coordinate (YYY2)  
      (YYY2)   -- +(130:1) coordinate (YYY3)
      (YYY3)   -- +(220:1) coordinate (YYY4)       
      (YYY4)   -- +(310:1) coordinate (YYY5)   ;
       \draw  (YY1)++(40:0.5)++(130:0.5) node {${1}$} ; 
   \draw[black,thin] 
      (YYY2)              coordinate (YYY1)
            -- +(40:1) coordinate (YYYY2)  
      (YYYY2)   -- +(130:1) coordinate (YYYY3)
      (YYYY3)   -- +(220:1) coordinate (YYYY4)       
      (YYYY4)   -- +(310:1) coordinate (YYYY5)   ;
       \draw  (YYY1)++(40:0.5)++(130:0.5) node {${0}$} ;

         }
    \end{tikzpicture}
\]
\caption{The bipartition $\gamma=((3,2,1^3),(4,2^2,1))$ for 
 $e=4$, $\kappa=(0,3)$, $g=0.99$, and   $\theta=(0,7)$.}
\label{admissible}
\end{figure}
\end{eg} 

We wish to consider the subalgebra $A_{\mathcal{R}\cup\mathcal{M}}(n+m,\theta,\kappa)$.
In particular, we wish to consider the subquotient of $A_{\mathcal{R}\cup\mathcal{M}}(n+m,\theta,\kappa)$ whose representations are indexed by the subset of multipartitions $\Gamma \subset \mptn l{n+m}(\mathcal{R}\cup\mathcal{M})$.

The set $\Gamma$ has unique maximal and minimal elements under the $\theta$-dominance order which we shall now describe.
We let $\gamma^+$ denote the $\dom_\theta$-maximal multipartition in $\Gamma$; that is, the multipartition obtained from $\gamma$ by adding all nodes as far left as possible. Similarly, denote by $\gamma^-$ the $\dom_\theta$-minimal multipartition in $\Gamma$, obtained from $\gamma$ by adding nodes as far to the right as possible.
It is clear that we can characterise $\Gamma$ as follows:
\[
\Gamma=\{\lambda \in  \mptn l{n+m}(\mathcal{R}\cup\mathcal{M}) \mid \gamma^+ \dom_\theta \lambda \dom_\theta \gamma^-\}
\]

\begin{eg}
In the example above, $\gamma^+=((4,3,2,1^2), (5,2^2,1))$
and $\gamma^-=((3,2,1^4), (5,2^3,1))$.  
\end{eg}

\begin{defn}
Given  $\gamma\in \mptn ln$, we define   idempotents
\[
e = \sum_{\mu   \dom \gamma^-} 1_\mu \quad \text{and}\quad   
f = \sum_{\mu   \ntrianglelefteqslant \gamma^+} 1_\mu   
\]
in $A_{\mathcal{R}\cup\mathcal{M}}(n+m,\theta,\kappa)$.  
We let $ A_{\Gamma}= A_\Gamma(\calm, \theta  )$ denote the subquotient of $A_{\mathcal{R}\cup\calm} = A_{\mathcal{R}\cup\calm}(n+m,\theta,\kappa)$ given by
\[
f(A_{\mathcal{R}\cup\calm}/(A_{\mathcal{R}\cup\calm} f A_{\mathcal{R}\cup\calm}))e.
\]
\end{defn}


\begin{prop}\label{3.6rmk}
The algebra $A_{\Gamma}$ is a   graded cellular algebra with a theory of highest weights.  
The cellular basis is given by
 \[
 \{C_{\SSTS,\SSTT}\mid \SSTS \in \SStd(\lambda,\mu), \, \SSTT\in \SStd(\lambda,\nu), \, \lambda,\mu, \nu \in \Gamma\},
 \]
 with respect to the  $\theta$-dominance order on  $\Gamma$. In particular, for $\la, \mu \in \Gamma$, we have that the graded decomposition number $d_{\la\mu}(t)$ for the algebras $A(n+m,\theta,\kappa)$ and $A_{\Gamma}$ are identical, and moreover, if $\la\neq \mu$, then $d_{\la\mu}(t)\in t \mathbb{N}_0[t]$.
\end{prop}

\begin{proof}
By definition, the sets   $E = \{\mu \mid \mu \dom \gamma^-\}$ and $F = \{\mu \mid \mu   \ntrianglelefteqslant \gamma^+\}$ are both cosaturated (in the sense of \cite[Appendix]{Donkin}) in the $\theta$-dominance ordering.
We claim that 
\[
A_{\mathcal{R}\cup\calm} f A_{\mathcal{R}\cup\calm} =
\langle C_{\SSTS\SSTT} \mid \SSTS \in \SStd(\lambda,\mu), \, \SSTT\in \SStd(\lambda,\nu), \, \lambda \in F, \mu,\nu\in\mptn ln \rangle_{\CC}.
\]
To see that the right-hand side is contained in the left-hand side, we note that if $\SSTS$ and $\SSTT$ are semistandard tableaux of shape $\lambda\in F$, then $C_{\SSTS\SSTT}=B_\SSTS  1_\lambda B_\SSTT^\ast$ by definition.
The reverse containment follows from axiom (3) of \cref{defn1} because each element $1_\lambda= C_{\SSTT^\lambda \SSTT^\lambda}$ is itself an element of the cellular basis.

The resulting quotient algebra has   basis indexed by $\SSTS \in \SStd(\lambda,\mu), \, \SSTT\in \SStd(\lambda,\nu), \, \lambda,\mu, \nu \not\in F$ (by conditions (2) and (3) of Definition \ref{defn1} and Theorem  \ref{cellularitybreedscontempt}).
Applying the idempotent truncation to this basis (and using (6) of Definition \ref{defn1}) we obtain the required basis of $A_{\Gamma}$.
The graded decomposition numbers (as well as dimensions of higher extension groups) are preserved under both the quotient and truncation maps, see for example \cite[Appendix]{Donkin} for the ungraded case. Applying Theorem \ref{step1} will thus prove the claim about graded decomposition numbers.
\end{proof}

\section{An isomorphism theorem}\label{isosec}
Let $m,e,\overline e \in \ZZ_{> 0}$ and let $i  \in \ZZ/e\ZZ$, $\overline i  \in \ZZ/\bar e\ZZ$. Suppose $\gamma$ and $\overline\gamma$ are $i$-admissible and $\overline  i$-admissible multipartitions, respectively.
We let $\mathcal M$ (respectively $\overline{\mathcal M})$ be a set of $m$ $i$-nodes (respectively $\overline i$-nodes), and let $\Gamma=\Gamma(\mathcal M)$ and $\overline\Gamma=\Gamma(\overline{\mathcal M})$, i.e.~$\Gamma$ is the set of multipartitions obtained from $\gamma$ by adding $m$ $i$-nodes, and similarly for $\overline\Gamma$.
 We shall associate a  sequence $\chi(\gamma)$  to $\gamma$
  (respectively $\chi(\overline\gamma)$   to $\overline\gamma$)
   which records the 
 series of $i$-\emph{diagonals} in $\gamma$ (respectively $\overline i$-\emph{diagonals} in $\gamma$).

  If  
  $\chi(\gamma)=\chi(\overline\gamma)$, then we shall show that  $A_{\Gamma} $ and 
  $A_{\overline\Gamma }$ are isomorphic as graded $\Bbbk$-algebras.  
 These isomorphisms 
   are independent of the quantum characteristic,  $e$-multicharge, weighting,
 level, and     degree  of the  corresponding diagrammatic  Cherednik algebras.  This provides new structural information 
 even for the classical Schur algebras in level 1.

  \subsection{Building $i$-diagonals from bricks}  
The combinatorics needed to state and prove our isomorphism theorem is that 
of  diagonals in the Young diagram of $\gamma$, which we now describe.  
\begin{defn}
Let $1\leq  k \leq l$  and let $(r,c) \in [\gamma^{(k)} ]\cup  
\Add(\gamma^{(k)})$ 
be an $i$-node.  
We refer to the set of nodes 
\[
\mathbf{D}= \{(a,b)
 \in [\gamma^{(k)}]  \mid a-b \in \{r-c-1,r-c,r-c+1\} \}
 \]
  as the associated \emph{$i$-diagonal}.  
  If $a-b$ is greater than, less than, or equal to zero,
 we say that the $i$-diagonal is to  the left of, right of, or centred on $\theta_k$, respectively.

We say that an $i$-diagonal in $[\gamma]$ 
 is \emph{visible} (respectively \emph{invisible})
 if the diagonal has (respectively does not have) 
 an addable $i$-node.
\end{defn}

\begin{eg}\label{notchiexample}
Let $e=5$, $\gamma=(10,9^2,6,4^2,3,2,1^2)$, and suppose $\kappa=(0)$, $\theta=(0)$ and $g=1$.  
This partition contains five $0$-diagonals 
(see Figure \ref{level1}).   
\end{eg}

\begin{figure}[h]
\begin{center}\scalefont{0.6}
\begin{tikzpicture}[scale=1]
   \path (0,0) coordinate (origin);     \draw[wei2] (0,0)   circle (2pt);
  \draw[thick] (origin) 
   --++(130:10*0.4)
   --++(40:1*0.4)
  --++(-50:1*0.4)	
  --++(40:2*0.4)
  --++(-50:3*0.4)
  --++(40:1*0.4)
  --++(-50:1*0.4)
  --++(-50:1*0.4)--++(40:1*0.4)
    --++(40:1*0.4)
  --++(-50:1*0.4)
  --++(40:1*0.4)
  --++(-50:1*0.4)
  --++(40:1*0.4)
  --++(-50:1*0.4)
  --++(40:2*0.4)
  --++(-50:1*0.4)
    --++(220:10*0.4);
     \clip (origin) 
     --++(130:10*0.4)
   --++(40:1*0.4)
  --++(-50:1*0.4)	
  --++(40:2*0.4)
  --++(-50:3*0.4)
  --++(40:1*0.4)
  --++(-50:1*0.4)
  --++(-50:1*0.4)--++(40:1*0.4)
    --++(40:1*0.4)
  --++(-50:1*0.4)
  --++(40:1*0.4)
  --++(-50:1*0.4)
  --++(40:1*0.4)
  --++(-50:1*0.4)
  --++(40:2*0.4)
  --++(-50:1*0.4)
    --++(220:10*0.4);
   \path (40:1cm) coordinate (A1);
  \path (40:2cm) coordinate (A2);
  \path (40:3cm) coordinate (A3);
  \path (40:4cm) coordinate (A4);
  \path (130:1cm) coordinate (B1);
  \path (130:2cm) coordinate (B2);
  \path (130:3cm) coordinate (B3);
  \path (130:4cm) coordinate (B4);
  \path (A1) ++(130:3cm) coordinate (C1);
  \path (A2) ++(130:2cm) coordinate (C2);
  \path (A3) ++(130:1cm) coordinate (C3);
  \foreach \i in {1,...,19}
  {
    \path (origin)++(40:0.4*\i cm)  coordinate (a\i);
    \path (origin)++(130:0.4*\i cm)  coordinate (b\i);
    \path (a\i)++(130:4cm) coordinate (ca\i);
    \path (b\i)++(40:4cm) coordinate (cb\i);
    \draw[thin,gray] (a\i) -- (ca\i)  (b\i) -- (cb\i); 
   \draw  (origin)++(40:0.2)++(130:0.2) node {${0}$} ; 
   \draw  (origin)
   ++(40:0.4)   ++(130:0.4)
   ++(40:0.2)++(130:0.2) node {${0}$} ; 
    \draw  (origin)
   ++(40:0.4)   ++(130:0.4)   ++(40:0.4)   ++(130:0.4)
   ++(40:0.2)++(130:0.2) node {${0}$} ; 
    \draw  (origin)
   ++(40:0.4)   ++(130:0.4)    ++(40:0.4)   ++(130:0.4)   ++(40:0.4)   ++(130:0.4)
   ++(40:0.2)++(130:0.2) node {${0}$} ; 
     \draw  (origin)
   ++(40:0.4)   ++(130:0.4)      ++(40:0.4)   ++(130:0.4)    ++(40:0.4)   ++(130:0.4)   ++(40:0.4)   ++(130:0.4)
   ++(40:0.2)++(130:0.2) node {${0}$} ; 
      \draw  (origin)
   ++(40:0.4) 
   ++(40:0.2)++(130:0.2) node {${4}$} ; 
      \draw  (origin)
   ++(40:0.4)++(40:0.4)   ++(130:0.4)
   ++(40:0.2)++(130:0.2) node {${4}$} ; 
    \draw  (origin)
  ++(40:0.4) ++(40:0.4)   ++(130:0.4)   ++(40:0.4)   ++(130:0.4)
   ++(40:0.2)++(130:0.2) node {${4}$} ; 
    \draw  (origin)
  ++(40:0.4) ++(40:0.4)   ++(130:0.4)    ++(40:0.4)   ++(130:0.4)   ++(40:0.4)   ++(130:0.4)
   ++(40:0.2)++(130:0.2) node {${4}$} ; 
     \draw  (origin)
 ++(40:0.4)  ++(40:0.4)   ++(130:0.4)      ++(40:0.4)   ++(130:0.4)    ++(40:0.4)   ++(130:0.4)   ++(40:0.4)   ++(130:0.4)
   ++(40:0.2)++(130:0.2) node {${4}$} ; 
      \draw  (origin)
   ++(130:0.4) 
   ++(40:0.2)++(130:0.2) node {${1}$} ; 
      \draw  (origin)
   ++(130:0.4) ++(40:0.4)   ++(130:0.4)
   ++(40:0.2)++(130:0.2) node {${1}$} ; 
    \draw  (origin)
   ++(130:0.4)  ++(40:0.4)   ++(130:0.4)   ++(40:0.4)   ++(130:0.4)
   ++(40:0.2)++(130:0.2) node {${1}$} ; 
    \draw  (origin)
   ++(130:0.4)  ++(40:0.4)   ++(130:0.4)    ++(40:0.4)   ++(130:0.4)   ++(40:0.4)   ++(130:0.4)
   ++(40:0.2)++(130:0.2) node {${1}$} ; 
     \draw  (origin)
   ++(130:0.4)   ++(40:0.4)   ++(130:0.4)      ++(40:0.4)   ++(130:0.4)    ++(40:0.4)   ++(130:0.4)   ++(40:0.4)   ++(130:0.4)
   ++(40:0.2)++(130:0.2) node {${1}$} ; 

    \path (origin) ++(130:2cm) coordinate (XXXXXX);

   \draw  (XXXXXX)++(40:0.2)++(130:0.2) node {${0}$} ; 
   \draw  (XXXXXX)
   ++(40:0.4)   ++(130:0.4)
   ++(40:0.2)++(130:0.2) node {${0}$} ; 
    \draw  (XXXXXX)
   ++(40:0.4)   ++(130:0.4)   ++(40:0.4)   ++(130:0.4)
   ++(40:0.2)++(130:0.2) node {${0}$} ; 
    \draw  (XXXXXX)
   ++(40:0.4)   ++(130:0.4)    ++(40:0.4)   ++(130:0.4)   ++(40:0.4)   ++(130:0.4)
   ++(40:0.2)++(130:0.2) node {${0}$} ; 
     \draw  (XXXXXX)
   ++(40:0.4)   ++(130:0.4)      ++(40:0.4)   ++(130:0.4)    ++(40:0.4)   ++(130:0.4)   ++(40:0.4)   ++(130:0.4)
   ++(40:0.2)++(130:0.2) node {${0}$} ; 
      \draw  (XXXXXX)
   ++(40:0.4) 
   ++(40:0.2)++(130:0.2) node {${4}$} ; 
      \draw  (XXXXXX)
   ++(40:0.4)++(40:0.4)   ++(130:0.4)
   ++(40:0.2)++(130:0.2) node {${4}$} ; 
    \draw  (XXXXXX)
  ++(40:0.4) ++(40:0.4)   ++(130:0.4)   ++(40:0.4)   ++(130:0.4)
   ++(40:0.2)++(130:0.2) node {${4}$} ; 
    \draw  (XXXXXX)
  ++(40:0.4) ++(40:0.4)   ++(130:0.4)    ++(40:0.4)   ++(130:0.4)   ++(40:0.4)   ++(130:0.4)
   ++(40:0.2)++(130:0.2) node {${4}$} ; 
     \draw  (XXXXXX)
 ++(40:0.4)  ++(40:0.4)   ++(130:0.4)      ++(40:0.4)   ++(130:0.4)    ++(40:0.4)   ++(130:0.4)   ++(40:0.4)   ++(130:0.4)
   ++(40:0.2)++(130:0.2) node {${4}$} ; 
      \draw  (XXXXXX)
   ++(130:0.4) 
   ++(40:0.2)++(130:0.2) node {${1}$} ; 
      \draw  (XXXXXX)
   ++(130:0.4) ++(40:0.4)   ++(130:0.4)
   ++(40:0.2)++(130:0.2) node {${1}$} ; 
    \draw  (XXXXXX)
   ++(130:0.4)  ++(40:0.4)   ++(130:0.4)   ++(40:0.4)   ++(130:0.4)
   ++(40:0.2)++(130:0.2) node {${1}$} ; 
    \draw  (XXXXXX)
   ++(130:0.4)  ++(40:0.4)   ++(130:0.4)    ++(40:0.4)   ++(130:0.4)   ++(40:0.4)   ++(130:0.4)
   ++(40:0.2)++(130:0.2) node {${1}$} ; 
     \draw  (XXXXXX)
   ++(130:0.4)   ++(40:0.4)   ++(130:0.4)      ++(40:0.4)   ++(130:0.4)    ++(40:0.4)   ++(130:0.4)   ++(40:0.4)   ++(130:0.4)
   ++(40:0.2)++(130:0.2) node {${1}$} ; 
  
     \draw  (XXXXXX)++(40:0.2)++(130:0.2)++(-50:0.4) node {${4}$} ; 
     \draw  (XXXXXX)++(40:0.2)++(130:0.2)++(130:0.4*4) node {${4}$} ; 
  
    \path (origin) ++(40:2cm) coordinate (XXXXXX);

   \draw  (XXXXXX)++(40:0.2)++(130:0.2)++(220:0.4) node {${1}$} ; 
   \draw  (XXXXXX)++(40:0.2)++(130:0.2) node {${0}$} ; 
   \draw  (XXXXXX)
   ++(40:0.4)   ++(130:0.4)
   ++(40:0.2)++(130:0.2) node {${0}$} ; 
    \draw  (XXXXXX)
   ++(40:0.4)   ++(130:0.4)   ++(40:0.4)   ++(130:0.4)
   ++(40:0.2)++(130:0.2) node {${0}$} ; 
    \draw  (XXXXXX)
   ++(40:0.4)   ++(130:0.4)    ++(40:0.4)   ++(130:0.4)   ++(40:0.4)   ++(130:0.4)
   ++(40:0.2)++(130:0.2) node {${0}$} ; 
     \draw  (XXXXXX)
   ++(40:0.4)   ++(130:0.4)      ++(40:0.4)   ++(130:0.4)    ++(40:0.4)   ++(130:0.4)   ++(40:0.4)   ++(130:0.4)
   ++(40:0.2)++(130:0.2) node {${0}$} ; 
      \draw  (XXXXXX)
   ++(40:0.4) 
   ++(40:0.2)++(130:0.2) node {${4}$} ; 
      \draw  (XXXXXX)
   ++(40:0.4)++(40:0.4)   ++(130:0.4)
   ++(40:0.2)++(130:0.2) node {${4}$} ; 
    \draw  (XXXXXX)
  ++(40:0.4) ++(40:0.4)   ++(130:0.4)   ++(40:0.4)   ++(130:0.4)
   ++(40:0.2)++(130:0.2) node {${4}$} ; 
    \draw  (XXXXXX)
  ++(40:0.4) ++(40:0.4)   ++(130:0.4)    ++(40:0.4)   ++(130:0.4)   ++(40:0.4)   ++(130:0.4)
   ++(40:0.2)++(130:0.2) node {${4}$} ; 
     \draw  (XXXXXX)
 ++(40:0.4)  ++(40:0.4)   ++(130:0.4)      ++(40:0.4)   ++(130:0.4)    ++(40:0.4)   ++(130:0.4)   ++(40:0.4)   ++(130:0.4)
   ++(40:0.2)++(130:0.2) node {${4}$} ; 
      \draw  (XXXXXX)
   ++(130:0.4) 
   ++(40:0.2)++(130:0.2) node {${1}$} ; 
      \draw  (XXXXXX)
   ++(130:0.4) ++(40:0.4)   ++(130:0.4)
   ++(40:0.2)++(130:0.2) node {${1}$} ; 
    \draw  (XXXXXX)
   ++(130:0.4)  ++(40:0.4)   ++(130:0.4)   ++(40:0.4)   ++(130:0.4)
   ++(40:0.2)++(130:0.2) node {${1}$} ; 
    \draw  (XXXXXX)
   ++(130:0.4)  ++(40:0.4)   ++(130:0.4)    ++(40:0.4)   ++(130:0.4)   ++(40:0.4)   ++(130:0.4)
   ++(40:0.2)++(130:0.2) node {${1}$} ; 
     \draw  (XXXXXX)
   ++(130:0.4)   ++(40:0.4)   ++(130:0.4)      ++(40:0.4)   ++(130:0.4)    ++(40:0.4)   ++(130:0.4)   ++(40:0.4)   ++(130:0.4)
   ++(40:0.2)++(130:0.2) node {${1}$} ; 
        \draw  (origin)
    ++(40:9*0.4)  
   ++(40:0.2)++(130:0.2) node {${1}$} ; 
  }
\end{tikzpicture}
\end{center}
\caption{The partition $\gamma=(10,9^2,6,4^2,3,2,1^2)$ with $\kappa=(0)$ and $e=5$ with the diagonals of interest for $i=0$.   The diagram features two $i$-diagonals to the left of $\theta_1$, one   centred on $\theta_1$ and two  to the right of $\theta_1$.  }
\label{level1}
\end{figure}
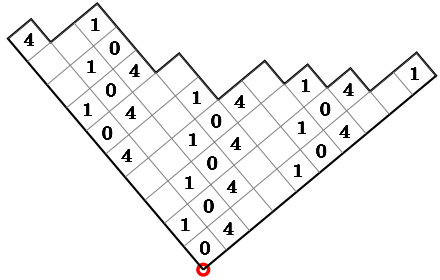


 Since the multipartitions 
 $\lambda,\mu \in \Gamma$ differ only by moving a set of $i$-nodes, we are only interested in neighbourhoods  (of a diagram)  
 in which an $i$-strand, $A$, crosses a strand  labelled by an $i$-, $(i+1)$-, or $(i-1)$- node in $\gamma$.  
 We shall now describe all ways in which this can happen.  
We shall build these $i$-diagonals from the set of bricks $\mathbf{B}_k$ for $k=1,\ldots ,5$ depicted in \cref{bricksleft} and the empty brick, $\mathbf{B}_6$.

\begin{figure}[h]
\begin{center}\scalefont{0.8}
\begin{tikzpicture}[scale=1]
   \path (0,0) coordinate (origin);
  \draw[thick] (origin) 
   --++(130:2*0.8)
   --++(40:1*0.8)
  --++(-50:1*0.8)	
   --++(40:1*0.8)
  --++(-50:1*0.8)	     --++(220:2*0.8)
;     \clip (origin) 
  (origin) 
    --++(130:2*0.8)
   --++(40:1*0.8)
  --++(-50:1*0.8)	
   --++(40:1*0.8)
  --++(-50:1*0.8)	     --++(220:2*0.8);     
   \path (40:1cm) coordinate (A1);
  \path (40:2cm) coordinate (A2);
  \path (40:3cm) coordinate (A3);
  \path (40:4cm) coordinate (A4);
  \path (130:1cm) coordinate (B1);
  \path (130:2cm) coordinate (B2);
  \path (130:3cm) coordinate (B3);
  \path (130:4cm) coordinate (B4);
  \path (A1) ++(130:3cm) coordinate (C1);
  \path (A2) ++(130:2cm) coordinate (C2);
  \path (A3) ++(130:1cm) coordinate (C3);
   \foreach \i in {1,...,19}
  {
    \path (origin)++(40:0.8*\i cm)  coordinate (a\i);
    \path (origin)++(130:0.8*\i cm)  coordinate (b\i);
    \path (a\i)++(130:4cm) coordinate (ca\i);
    \path (b\i)++(40:4cm) coordinate (cb\i);
    \draw[thin,gray] (a\i) -- (ca\i)  (b\i) -- (cb\i); 
 \draw  (origin)
    ++(40:0.4)++(130:0.4) node {${i}$} ;
  \draw  (origin)
    ++(40:0.4)++(130:1.2) node {${i+1}$} ;
      \draw  (origin)
      ++(40:1.2)++(130:0.4)  node {${i-1}$} ;
  }  
\end{tikzpicture}
 \quad
 \begin{tikzpicture}[scale=1]
   \path (0,0) coordinate (origin);
  \draw[thick] (origin) 
   --++(130:1*0.8)
   --++(40:2*0.8)
  --++(-50:1*0.8)	
      --++(220:2*0.8)
;     \clip (origin) 
   --++(130:1*0.8)
   --++(40:2*0.8)
  --++(-50:1*0.8)	
      --++(220:2*0.8)
;     
   \path (40:1cm) coordinate (A1);
  \path (40:2cm) coordinate (A2);
  \path (40:3cm) coordinate (A3);
  \path (40:4cm) coordinate (A4);
  \path (130:1cm) coordinate (B1);
  \path (130:2cm) coordinate (B2);
  \path (130:3cm) coordinate (B3);
  \path (130:4cm) coordinate (B4);
  \path (A1) ++(130:3cm) coordinate (C1);
  \path (A2) ++(130:2cm) coordinate (C2);
  \path (A3) ++(130:1cm) coordinate (C3);
   \foreach \i in {1,...,19}
  {
    \path (origin)++(40:0.8*\i cm)  coordinate (a\i);
    \path (origin)++(130:0.8*\i cm)  coordinate (b\i);
    \path (a\i)++(130:4cm) coordinate (ca\i);
    \path (b\i)++(40:4cm) coordinate (cb\i);
    \draw[thin,gray] (a\i) -- (ca\i)  (b\i) -- (cb\i); 
 \draw  (origin)
    ++(40:0.4)++(130:0.4) node {${i}$} ;
  \draw  (origin)
    ++(40:0.4)++(130:1.2) node {${i+1}$} ;
      \draw  (origin)
      ++(40:1.2)++(130:0.4)  node {${i-1}$} ;
  }  
\end{tikzpicture}
\quad
 \begin{tikzpicture}[scale=1]
   \path (0,0) coordinate (origin);
  \draw[thick] (origin) 
   --++(130:2*0.8)
   --++(40:1*0.8)
  --++(-50:2*0.8)	
      --++(220:1*0.8)
;     \clip (origin) 
  --++(130:2*0.8)
   --++(40:1*0.8)
  --++(-50:2*0.8)	
      --++(220:1*0.8)
;     
   \path (40:1cm) coordinate (A1);
  \path (40:2cm) coordinate (A2);
  \path (40:3cm) coordinate (A3);
  \path (40:4cm) coordinate (A4);
  \path (130:1cm) coordinate (B1);
  \path (130:2cm) coordinate (B2);
  \path (130:3cm) coordinate (B3);
  \path (130:4cm) coordinate (B4);
  \path (A1) ++(130:3cm) coordinate (C1);
  \path (A2) ++(130:2cm) coordinate (C2);
  \path (A3) ++(130:1cm) coordinate (C3);
   \foreach \i in {1,...,19}
  {
    \path (origin)++(40:0.8*\i cm)  coordinate (a\i);
    \path (origin)++(130:0.8*\i cm)  coordinate (b\i);
    \path (a\i)++(130:4cm) coordinate (ca\i);
    \path (b\i)++(40:4cm) coordinate (cb\i);
    \draw[thin,gray] (a\i) -- (ca\i)  (b\i) -- (cb\i); 
 \draw  (origin)
    ++(40:0.4)++(130:0.4) node {${i}$} ;
  \draw  (origin)
    ++(40:0.4)++(130:1.2) node {${i+1}$} ;
      \draw  (origin)
      ++(40:1.2)++(130:0.4)  node {${i-1}$} ;
  }  
\end{tikzpicture} 
\quad
 \begin{tikzpicture}[scale=1]
   \path (0,0) coordinate (origin);
  \draw[thick] (origin) 
   --++(130:1*0.8)
   --++(40:1*0.8)
  --++(-50:1*0.8)	
   --++(220:1*0.8)
;     \clip (origin) 
  (origin) 
   --++(130:1*0.8)
   --++(40:1*0.8)
  --++(-50:1*0.8)	
   --++(220:1*0.8)
;     
   \path (40:1cm) coordinate (A1);
  \path (40:2cm) coordinate (A2);
  \path (40:3cm) coordinate (A3);
  \path (40:4cm) coordinate (A4);
  \path (130:1cm) coordinate (B1);
  \path (130:2cm) coordinate (B2);
  \path (130:3cm) coordinate (B3);
  \path (130:4cm) coordinate (B4);
  \path (A1) ++(130:3cm) coordinate (C1);
  \path (A2) ++(130:2cm) coordinate (C2);
  \path (A3) ++(130:1cm) coordinate (C3);
  \foreach \i in {1,...,19}
  {
    \path (origin)++(40:0.8*\i cm)  coordinate (a\i);
    \path (origin)++(130:0.8*\i cm)  coordinate (b\i);
    \path (a\i)++(130:4cm) coordinate (ca\i);
    \path (b\i)++(40:4cm) coordinate (cb\i);
    \draw[thin,gray] (a\i) -- (ca\i)  (b\i) -- (cb\i); 
\draw  (origin)
    ++(40:0.4)++(130:0.4) node {${i-1}$} ;
 \draw  (origin)
    ++(40:0.4)++(130:1.2) node {${i}$} ; 
     \draw  (origin)
    ++(40:1.2)++(130:1.2) node {${i-1}$} ; 
  }  
\end{tikzpicture}
\quad
\begin{tikzpicture}[scale=1]
   \path (0,0) coordinate (origin);
  \draw[thick] (origin) 
   --++(130:1*0.8)
   --++(40:1*0.8)
  --++(-50:1*0.8)	
   --++(220:1*0.8)
;     \clip (origin) 
  (origin) 
   --++(130:1*0.8)
   --++(40:1*0.8)
  --++(-50:1*0.8)	
   --++(220:1*0.8)
;     
   \path (40:1cm) coordinate (A1);
  \path (40:2cm) coordinate (A2);
  \path (40:3cm) coordinate (A3);
  \path (40:4cm) coordinate (A4);
  \path (130:1cm) coordinate (B1);
  \path (130:2cm) coordinate (B2);
  \path (130:3cm) coordinate (B3);
  \path (130:4cm) coordinate (B4);
  \path (A1) ++(130:3cm) coordinate (C1);
  \path (A2) ++(130:2cm) coordinate (C2);
  \path (A3) ++(130:1cm) coordinate (C3);
   \foreach \i in {1,...,19}
  {
    \path (origin)++(40:0.8*\i cm)  coordinate (a\i);
    \path (origin)++(130:0.8*\i cm)  coordinate (b\i);
    \path (a\i)++(130:4cm) coordinate (ca\i);
    \path (b\i)++(40:4cm) coordinate (cb\i);
    \draw[thin,gray] (a\i) -- (ca\i)  (b\i) -- (cb\i); 
\draw  (origin)
    ++(40:0.4)++(130:0.4) node {${i+1}$} ;
   }

\end{tikzpicture}\quad
\end{center}
\caption{The bricks $\mathbf{B}_1$, $\mathbf{B}_2$, $\mathbf{B}_3$, $\mathbf{B}_4$, $\mathbf{B}_5$ respectively. 
The $\mathbf{B}_6$ brick is a single red $i$-strand (in other words it corresponds to an empty partition with charge  $i$).}
\label{bricksleft}
\end{figure}
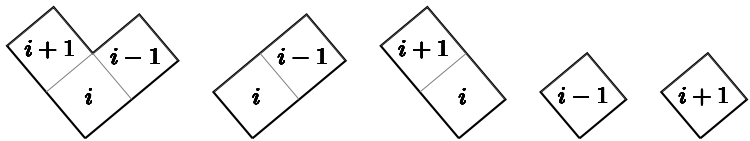

\noindent\textbf{Case 1: Visible $i$-diagonals.}
 Fix some component $1\leq k \leq l$.  
 There are three types of $i$-diagonal which can occur (in this component)
  which have an addable $i$-node at the top.  Namely, 
  those which occur to the left or right of the node $(1,1,k)$ and those which
  occur on the node $(1,1,k)$. 
  It's not difficult to see that all three of these cases can be built out of the bricks $\mathbf{B}_1 $
 and a single    $\mathbf{B}_4$, $\mathbf{B}_5$, 
 or $\mathbf{B}_6$ brick respectively.    Namely, we place a $\mathbf{B}_4$, $\mathbf{B}_5$, 
 or $\mathbf{B}_6$ at the base (for $i$-diagonals to the left, right, or centred on $\theta_k$, respectively) and then 
 put some number (possibly zero) of   $\mathbf{B}_1$ bricks on top.  
 Examples  of how to construct such an $i$-diagonal are 
  depicted in   \cref{case1sum}.

\begin{figure}[h]
\begin{center}\scalefont{0.6}
\begin{tikzpicture}[scale=1]
   \path (0,0) coordinate (origin);     \draw[wei2] (0,0)   circle (2pt);
  \draw[very thick] (origin) 
   --++(130:2*0.4)
   --++(40:1*0.4)   --++(130:1*0.4)
   ;
  \draw[very thick] (origin) 
   --++(130:3*0.4)
   --++(40:2*0.4)
  --++(130:1*0.4)	
   --++(220:1*0.4) coordinate (VVVVVVV)
    --++(130:1*0.4)
--++(220:1*0.4)--++(-50:2*0.4) ;
      \draw[very thick] (VVVVVVV)
   --++(40:2*0.4)
  --++(130:1*0.4)	
   --++(220:1*0.4) coordinate (VVVVVVV1)
    --++(130:1*0.4)
--++(220:1*0.4)--++(-50:2*0.4) ;
     \draw[very thick] (VVVVVVV1)
   --++(40:2*0.4)
  --++(130:1*0.4)	
   --++(220:1*0.4) coordinate (VVVVVVV2)
    --++(130:1*0.4)
--++(220:1*0.4)--++(-50:2*0.4) ;
     \draw[very thick] (VVVVVVV2)
   --++(40:2*0.4)
  --++(130:1*0.4)	
   --++(220:1*0.4) coordinate (VVVVVVV3)
    --++(130:1*0.4)
--++(220:1*0.4)--++(-50:2*0.4) ;

  \draw[thick] (origin) 
   --++(130:10*0.4)
   --++(40:1*0.4)
  --++(-50:1*0.4)	
  --++(40:2*0.4)
  --++(-50:1*0.4)
  --++(40:1*0.4)
    --++(-50:1*0.4)
  --++(40:2*0.4)
  --++(-50:1*0.4)
  --++(-50:2*0.4)--++(40:1*0.4)
  --++(-50:1*0.4)
  --++(40:1*0.4)
  --++(-50:1*0.4)
  --++(-50:2*0.4)
    --++(220:8*0.4);
     \clip (origin) 
   --++(130:10*0.4)
   --++(40:1*0.4)
  --++(-50:1*0.4)	
  --++(40:2*0.4)
  --++(-50:1*0.4)
  --++(40:1*0.4)
    --++(-50:1*0.4)
  --++(40:2*0.4)
  --++(-50:1*0.4)
  --++(-50:2*0.4)--++(40:1*0.4)
  --++(-50:1*0.4)
  --++(40:1*0.4)
  --++(-50:1*0.4)
  --++(-50:2*0.4)
    --++(220:8*0.4);
   \path (40:1cm) coordinate (A1);
  \path (40:2cm) coordinate (A2);
  \path (40:3cm) coordinate (A3);
  \path (40:4cm) coordinate (A4);
  \path (130:1cm) coordinate (B1);
  \path (130:2cm) coordinate (B2);
  \path (130:3cm) coordinate (B3);
  \path (130:4cm) coordinate (B4);
  \path (A1) ++(130:3cm) coordinate (C1);
  \path (A2) ++(130:2cm) coordinate (C2);
  \path (A3) ++(130:1cm) coordinate (C3);
\end{tikzpicture}
\quad
\begin{tikzpicture}[scale=1]
   \path (0,0) coordinate (origin);     \draw[wei2] (0,0)   circle (2pt);
  \draw[very thick] (origin) 
   --++(40:1*0.4)
   --++(130:1*0.4)   --++(40:1*0.4)
   ;
  \draw[very thick] (origin) 
   --++(40:2*0.4)
   --++(40:2*0.4)
  --++(130:1*0.4)	
   --++(220:1*0.4) coordinate (VVVVVVV)
    --++(130:1*0.4)
--++(220:1*0.4)--++(-50:2*0.4) ;
      \draw[very thick] (VVVVVVV)
   --++(40:2*0.4)
  --++(130:1*0.4)	
   --++(220:1*0.4) coordinate (VVVVVVV1)
    --++(130:1*0.4)
--++(220:1*0.4)--++(-50:2*0.4) ;
     \draw[very thick] (VVVVVVV1)
   --++(40:2*0.4)
  --++(130:1*0.4)	
   --++(220:1*0.4) coordinate (VVVVVVV2)
    --++(130:1*0.4)
--++(220:1*0.4)--++(-50:2*0.4) ;
     \draw[very thick] (VVVVVVV2)
   --++(40:2*0.4)
  --++(130:1*0.4)	
   --++(220:1*0.4) coordinate (VVVVVVV3)
    --++(130:1*0.4)
--++(220:1*0.4)--++(-50:2*0.4) ;

  \draw[thick] (origin) 
   --++(130:8*0.4)
   --++(40:1*0.4)
     --++(-50:1*0.4)
  --++(40:1*0.4)
   --++(40:1*0.4)
  --++(-50:1*0.4)
  --++(40:1*0.4)
    --++(-50:1*0.4)
  --++(40:2*0.4)
  --++(-50:1*0.4)--++(40:1*0.4)
  --++(-50:1*0.4)
  --++(40:1*0.4)
  --++(-50:1*0.4)
  --++(40:1*0.4)
  --++(-50:1*0.4)
  --++(40:1*0.4)
  --++(-50:1*0.4)
    --++(220:10*0.4);
     \clip (origin) 
     --++(130:8*0.4)
   --++(40:1*0.4)
     --++(-50:1*0.4)
  --++(40:1*0.4)
   --++(40:1*0.4)
  --++(-50:1*0.4)
  --++(40:1*0.4)
    --++(-50:1*0.4)
  --++(40:2*0.4)
  --++(-50:1*0.4)--++(40:1*0.4)
  --++(-50:1*0.4)
  --++(40:1*0.4)
  --++(-50:1*0.4)
  --++(40:1*0.4)
  --++(-50:1*0.4)
  --++(40:1*0.4)
  --++(-50:1*0.4)
    --++(220:10*0.4);
   \path (40:1cm) coordinate (A1);
  \path (40:2cm) coordinate (A2);
  \path (40:3cm) coordinate (A3);
  \path (40:4cm) coordinate (A4);
  \path (130:1cm) coordinate (B1);
  \path (130:2cm) coordinate (B2);
  \path (130:3cm) coordinate (B3);
  \path (130:4cm) coordinate (B4);
  \path (A1) ++(130:3cm) coordinate (C1);
  \path (A2) ++(130:2cm) coordinate (C2);
  \path (A3) ++(130:1cm) coordinate (C3);
 define a clip region to crop the resulting figure
\end{tikzpicture}
\quad
\begin{tikzpicture}[scale=1]
   \path (0,0) coordinate (origin);     \draw[wei2] (0,0)   circle (2pt);
  \draw[very thick] (origin) 
   --++(130:0*0.4)
   --++(40:2*0.4)
  --++(130:1*0.4)	
   --++(220:1*0.4) coordinate (VVVVVVV)
    --++(130:1*0.4)
--++(220:1*0.4)--++(-50:2*0.4) ;
      \draw[very thick] (VVVVVVV)
   --++(40:2*0.4)
  --++(130:1*0.4)	
   --++(220:1*0.4) coordinate (VVVVVVV1)
    --++(130:1*0.4)
--++(220:1*0.4)--++(-50:2*0.4) ;
     \draw[very thick] (VVVVVVV1)
   --++(40:2*0.4)
  --++(130:1*0.4)	
   --++(220:1*0.4) coordinate (VVVVVVV2)
    --++(130:1*0.4)
--++(220:1*0.4)--++(-50:2*0.4) ;
     \draw[very thick] (VVVVVVV2)
   --++(40:2*0.4)
  --++(130:1*0.4)	
   --++(220:1*0.4) coordinate (VVVVVVV3)
    --++(130:1*0.4)
--++(220:1*0.4)--++(-50:2*0.4) ;

  \draw[thick] (origin) 
   --++(130:9*0.4)
   --++(40:1*0.4)
  --++(-50:1*0.4)
     --++(40:1*0.4)
  --++(-50:1*0.4)	
  --++(40:1*0.4)
  --++(-50:1*0.4)	
  --++(40:1*0.4)
  --++(-50:2*0.4)
  --++(40:1*0.4)
    --++(-50:1*0.4)
  --++(40:1*0.4)
  --++(-50:1*0.4)
--++(40:1*0.4)
  --++(-50:1*0.4)
  --++(40:1*0.4)
  --++(-50:1*0.4)
    --++(220:8*0.4);
     \clip (origin) 
   --++(130:9*0.4)
   --++(40:1*0.4)
  --++(-50:1*0.4)
     --++(40:1*0.4)
  --++(-50:1*0.4)	
  --++(40:1*0.4)
  --++(-50:1*0.4)	
  --++(40:1*0.4)
  --++(-50:2*0.4)
  --++(40:1*0.4)
    --++(-50:1*0.4)
  --++(40:1*0.4)
  --++(-50:1*0.4)
--++(40:1*0.4)
  --++(-50:1*0.4)
  --++(40:1*0.4)
  --++(-50:1*0.4)
    --++(220:8*0.4);
   \path (40:1cm) coordinate (A1);
  \path (40:2cm) coordinate (A2);
  \path (40:3cm) coordinate (A3);
  \path (40:4cm) coordinate (A4);
  \path (130:1cm) coordinate (B1);
  \path (130:2cm) coordinate (B2);
  \path (130:3cm) coordinate (B3);
  \path (130:4cm) coordinate (B4);
  \path (A1) ++(130:3cm) coordinate (C1);
  \path (A2) ++(130:2cm) coordinate (C2);
  \path (A3) ++(130:1cm) coordinate (C3);
\end{tikzpicture}
\end{center}
\caption{Examples of visible $i$-diagonals to the left, right, and centred on $(1,1,k)$.   }
\label{case1sum}
\end{figure}
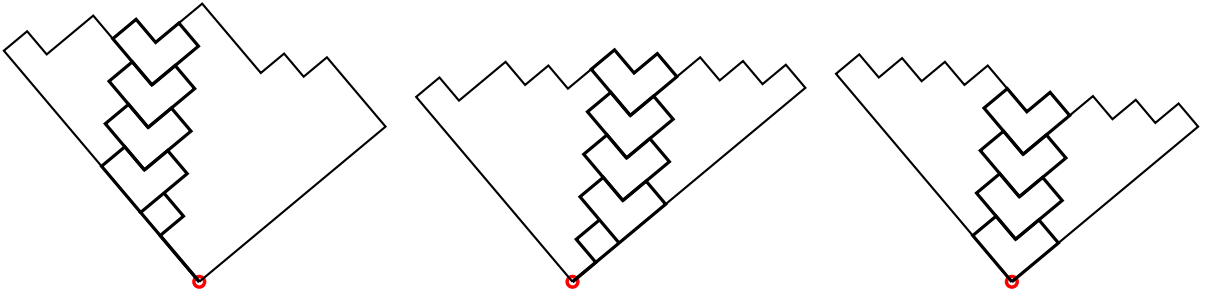

\noindent\textbf{Case 2: Invisible $i$-diagonals.}    
Recall that we say an $i$-diagonal is invisible if it does not have an addable $i$-node at the top. Since $\gamma$ is $i$-admissible, it has no removable $i$-nodes, and there are thus six possible invisible $i$-diagonals; these are obtained by adding either a $\mathbf{B}_2$ or $\mathbf{B}_3$ brick to the top of one of the three types of visible $i$-diagonal.
Examples of how to construct such $i$-diagonals are depicted in \cref{case2sum}.

\begin{figure}[ht]
\begin{center}\scalefont{0.6}
\begin{tikzpicture}[scale=1]
   \path (0,0) coordinate (origin);     \draw[wei2] (0,0)   circle (2pt);
  \draw[very thick] (origin) 
   --++(130:2*0.4)
   --++(40:1*0.4)   --++(130:1*0.4)
   ;
  \draw[very thick] (origin) 
   --++(130:3*0.4)
   --++(40:2*0.4)
  --++(130:1*0.4)	
   --++(220:1*0.4) coordinate (VVVVVVV)
    --++(130:1*0.4)
--++(220:1*0.4)--++(-50:2*0.4) ;
      \draw[very thick] (VVVVVVV)
   --++(40:2*0.4)
  --++(130:1*0.4)	
   --++(220:1*0.4) coordinate (VVVVVVV1)
    --++(130:1*0.4)
--++(220:1*0.4)--++(-50:2*0.4) ;
     \draw[very thick] (VVVVVVV1)
   --++(40:2*0.4)
  --++(130:1*0.4)	
   --++(220:1*0.4) coordinate (VVVVVVV2)
    --++(130:1*0.4)
--++(220:1*0.4)--++(-50:2*0.4) ;
     \draw[very thick] (VVVVVVV2)
   --++(40:2*0.4)
  --++(130:1*0.4)	
   --++(220:1*0.4) coordinate (VVVVVVV3)
    --++(130:1*0.4)
--++(220:1*0.4)--++(-50:2*0.4) ;
  \draw[very thick] (VVVVVVV3)
   --++(40:2*0.4)
  --++(130:1*0.4)	
   --++(220:2*0.4)  
    --++(-50:1*0.4)
;
  \draw[thick] (origin) 
   --++(130:10*0.4)
   --++(40:1*0.4)
  --++(-50:1*0.4)	
  --++(40:2*0.4)
  --++(-50:1*0.4)
  --++(40:1*0.4)
    --++(-50:1*0.4)
  --++(40:2*0.4)
  --++(-50:1*0.4)
  --++(-50:2*0.4)--++(40:1*0.4)
  --++(-50:1*0.4)
  --++(40:1*0.4)
  --++(-50:1*0.4)
  --++(-50:2*0.4)
    --++(220:8*0.4);
     \clip (origin) 
   --++(130:10*0.4)
   --++(40:1*0.4)
  --++(-50:1*0.4)	
  --++(40:2*0.4)
  --++(-50:1*0.4)
  --++(40:1*0.4)
    --++(-50:1*0.4)
  --++(40:2*0.4)
  --++(-50:1*0.4)
  --++(-50:2*0.4)--++(40:1*0.4)
  --++(-50:1*0.4)
  --++(40:1*0.4)
  --++(-50:1*0.4)
  --++(-50:2*0.4)
    --++(220:8*0.4);
   \path (40:1cm) coordinate (A1);
  \path (40:2cm) coordinate (A2);
  \path (40:3cm) coordinate (A3);
  \path (40:4cm) coordinate (A4);
  \path (130:1cm) coordinate (B1);
  \path (130:2cm) coordinate (B2);
  \path (130:3cm) coordinate (B3);
  \path (130:4cm) coordinate (B4);
  \path (A1) ++(130:3cm) coordinate (C1);
  \path (A2) ++(130:2cm) coordinate (C2);
  \path (A3) ++(130:1cm) coordinate (C3);
\end{tikzpicture}
 \quad
\begin{tikzpicture}[scale=1]
   \path (0,0) coordinate (origin);     \draw[wei2] (0,0)   circle (2pt);
  \draw[very thick] (origin) 
   --++(130:0*0.4)
   --++(40:2*0.4)
  --++(130:1*0.4)	
   --++(220:1*0.4) coordinate (VVVVVVV)
    --++(130:1*0.4)
--++(220:1*0.4)--++(-50:2*0.4) ;
      \draw[very thick] (VVVVVVV)
   --++(40:2*0.4)
  --++(130:1*0.4)	
   --++(220:1*0.4) coordinate (VVVVVVV1)
    --++(130:1*0.4)
--++(220:1*0.4)--++(-50:2*0.4) ;
     \draw[very thick] (VVVVVVV1)
   --++(40:2*0.4)
  --++(130:1*0.4)	
   --++(220:1*0.4) coordinate (VVVVVVV2)
    --++(130:1*0.4)
--++(220:1*0.4)--++(-50:2*0.4) ;
     \draw[very thick] (VVVVVVV2)
   --++(40:2*0.4)
  --++(130:1*0.4)	
   --++(220:1*0.4) coordinate (VVVVVVV3)
    --++(130:1*0.4)
--++(220:1*0.4)--++(-50:2*0.4) ;
  \draw[very thick] (VVVVVVV3)
   --++(40:1*0.4)
  --++(130:2*0.4)	
   --++(220:1*0.4)  
    --++(-50:1*0.4)
;
  \draw[thick] (origin) 
   --++(130:9*0.4)
   --++(40:1*0.4)
  --++(-50:1*0.4)
     --++(40:1*0.4)
  --++(-50:1*0.4)	
  --++(40:1*0.4)
  --++(-50:1*0.4)	
  --++(40:1*0.4)
  --++(-50:2*0.4)
  --++(40:1*0.4)
    --++(-50:1*0.4)
  --++(40:1*0.4)
  --++(-50:1*0.4)
--++(40:1*0.4)
  --++(-50:1*0.4)
  --++(40:1*0.4)
  --++(-50:1*0.4)
    --++(220:8*0.4);
     \clip (origin) 
   --++(130:9*0.4)
   --++(40:1*0.4)
  --++(-50:1*0.4)
     --++(40:1*0.4)
  --++(-50:1*0.4)	
  --++(40:1*0.4)
  --++(-50:1*0.4)	
  --++(40:1*0.4)
  --++(-50:2*0.4)
  --++(40:1*0.4)
    --++(-50:1*0.4)
  --++(40:1*0.4)
  --++(-50:1*0.4)
--++(40:1*0.4)
  --++(-50:1*0.4)
  --++(40:1*0.4)
  --++(-50:1*0.4)
    --++(220:8*0.4);
   \path (40:1cm) coordinate (A1);
  \path (40:2cm) coordinate (A2);
  \path (40:3cm) coordinate (A3);
  \path (40:4cm) coordinate (A4);
  \path (130:1cm) coordinate (B1);
  \path (130:2cm) coordinate (B2);
  \path (130:3cm) coordinate (B3);
  \path (130:4cm) coordinate (B4);
  \path (A1) ++(130:3cm) coordinate (C1);
  \path (A2) ++(130:2cm) coordinate (C2);
  \path (A3) ++(130:1cm) coordinate (C3);
\end{tikzpicture}
\end{center} 
\caption{Examples of invisible $i$-diagonals.  The former (respectively latter) is obtained by adding a $\mathbf{B}_2$ (respectively $\mathbf{B}_3$) brick to the leftmost (respectively rightmost) diagram in \cref{case1sum}.
}
\label{case2sum}
\end{figure}
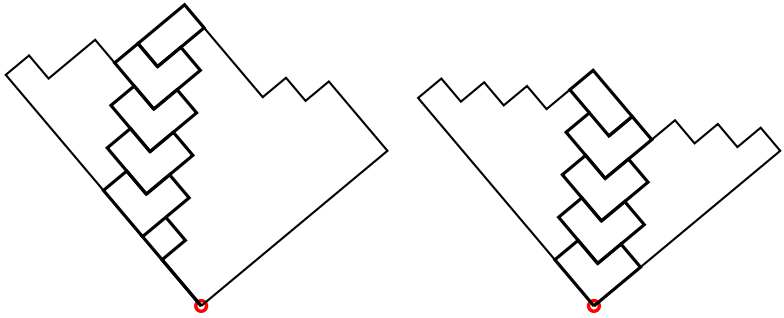

Let   $\mathbf{D}$ be any $i$-diagonal.  
We define $x(\mathbf{D})$ to be the $x$-coordinate of the top vertex of the top $i$-node in $\mathbf{D}$ or the left vertex of the top $(i-1)$-node in $\mathbf{D}$ or the right vertex of the top $(i+1)$-node in $\mathbf{D}$ if such a node exists  (if they all exist, then the definitions clearly agree).

\begin{eg} Continuing from  \cref{notchiexample}; the ordered set of $x$-coordinates of the diagonals  is equal to $\{-10+10\epsilon, -5+11\epsilon, 8\epsilon, 5+9\epsilon, 10+11\epsilon\}$.
\end{eg}

\subsection{Strands passing through $i$-diagonals}
We shall  let $A$ denote an $i$-node of $\SSTT \in \SStd(\lambda,\mu)$ and identify the 
 node with the strand it labels in the diagram  $B_\SSTT$.    

\begin{defn}\label{passdef}
We say that an $i$-strand, $A$, \emph{passes through} an $i$-diagonal, $\mathbf{D}$, if there is a neighbourhood of the diagram in which  $A$  is at least $2\epsilon$ to the left of all
ghost $(i-1)$-strands in $\mathbf{D}$ and a neighbourhood of the diagram in which 
the ghost of $A$  is at least $2\epsilon$ to the right of all the black  $(i+1)$-strands in $\mathbf{D}$.
\end{defn}
 
\begin{rmk}\label{remarkonmovingfromrighttoleft}
If  $A$  satisfies \cref{passdef} then it also has the property that the ghost of $A$ is to the  left of all black $(i+1)$-strands in some neighbourhood (the first in the definition) and that $A$ is strictly to the right of all black $i$-strands and ghost $(i-1)$-strands in some neighbourhood (the second in the definition). In this way, $A$ passing through an $i$-diagonal means that $A$ and its ghost cross all strands corresponding to the $i$-diagonal which may contribute to the degree or give rise to relations.

In fact,  suppose that   $A$ passes through an $i$-diagonal $\mathbf{D}$ and that $\mathbf{B}$ and $\mathbf{B}'$ are two bricks in $\mathbf{D}$ such that $\mathbf{B}$ is above $\mathbf{B}'$ in the $[\gamma]$.  
 Then the  
  $(i-1)$-ghost,    $i$-strand, and   $(i+1)$-strand 
in $\mathbf{B}$ each occur strictly to the right of the corresponding strand in $\mathbf{B}'$.
Therefore all non-trivial 
interactions between $A$ and $\mathbf{B}$ happen 
before those between $A$ and $\mathbf{B}'$ (reading from right to left).  
\end{rmk}

 \begin{eg}
Consider the diagrams   $B_\SSTT$  and  $B_\SSTU$ for 
$\SSTT$  and $\SSTU$ as in \cref{semistandard}.  These are   depicted in Figures  \ref{passing} and \ref{passing2}.  


Consider the diagram $B_\SSTT$.  This diagram has a total of three 
0-diagonals;   $A$  passes through the two rightmost $0$-diagonals.  
 The first of these two crossings is with a  $0$-diagonal consisting of a 1-strand, a 0-strand, and a 4-strand, centred at $\theta_1$.  
The other 0-diagonal consists only of a single $4$-strand (a $\mathbf{B}_5$ brick)
 and is at the far right of the diagram.  
   The total degree of the diagram is 2; the crossing of  $A$  with the centred diagonal has degree 1 (as before);
   the crossing of  $A$  with the  rightmost diagonal also has degree 1.

Now consider the diagram $B_\SSTU$.  This diagram has a total of three 
0-diagonals;   $A$  passes through the two rightmost $0$-diagonals.  
 The first of these two crossings is with a  $0$-diagonal  built from
  a $\mathbf{B}_7$ brick a $\mathbf{B}_1$ brick and a $\mathbf{B}_2$ brick.
 The second diagonal consists of a $\mathbf{B}_5$ brick, as before.   
   
    \begin{figure}[ht!]
  \[  
    \scalefont{0.7}  
  \begin{tikzpicture}[scale=0.8] 
     \clip(-5.6,-0.5)rectangle (6,2.2);
  \draw (-5.5,0) rectangle (6,2);
    \node [white,above] at (-0.8,2)  {\tiny $1$};
   \foreach \x in {-3.5,-2.6,-1.7,-0.8,0.1,1.2,2.3,3.4,4.5}  
     {\draw  (\x,0)--(\x,2); \draw[dashed]  (\x-1,0)--(\x-1,2); }
 \draw[wei2] (-0.2,0)  to [out=90,in=-90] (-0.2,2) ;      
 \draw[dashed]  (-4.4-1,0)  to [out=35,in=-150]
 (5.6-0.9,2) ;      
 \draw  (-4.4,0)   to [out=30,in=-150]
  (5.6,2) ;      
           \node [below] at (-4.4,0)  {\tiny $0$}; 

           \node [below] at (-3.5,0)  {\tiny $4$}; 
           \node [below] at (-2.6,0)  {\tiny $3$};           
           \node [below] at (-1.7,0)  {\tiny $2$};
           \node [below] at (-0.8,0)  {\tiny $1$};
                      \node [wei2,below] at (-0.2,0)  {\tiny $0$};
           \node [below] at (0.1,0)  {\tiny $0$};
                      \node [below] at (1.2,0)  {\tiny $4$}; 
\node [below] at (2.3,0)  {\tiny $3$}; 
\node [below] at (3.4,0)  {\tiny $2$}; 
\node [below] at (4.5,0)  {\tiny $1$};  
  \end{tikzpicture}
\]
\caption{
  The  diagram  $B_\SSTT$  for $\SSTT$ as in Figure \ref{semistandard}.    
 }
\label{passing}
\end{figure}

 \begin{figure}[ht!]
  \[    \scalefont{0.7}  
     \begin{tikzpicture}[scale=0.8] 
     \clip(-5.6,-0.5)rectangle (6,2.2);
  \draw (-5.5,0) rectangle (6,2);
    \node [white,above] at (-0.8,2)  {\tiny $1$};
   \foreach \x in {-3.5,-2.6,-1.7,-0.8,0.1,0.3,1.2,1.4,2.3,3.4,4.5}  
     {\draw  (\x,0)--(\x,2); \draw[dashed]  (\x-1,0)--(\x-1,2); }
 \draw[wei2] (-0.2,0)  to [out=90,in=-90] (-0.2,2) ;      
 \draw[dashed]  (-4.4-1,0)  to [out=35,in=-150]
  (5.6-0.9,2) ;      
 \draw  (-4.4,0)   to [out=30,in=-150]
  (5.6,2) ;      
           \node [below] at (-4.4,0)  {\tiny $0$}; 

           \node [below] at (-3.5,0)  {\tiny $4$}; 
           \node [below] at (-2.6,0)  {\tiny $3$};           
           \node [below] at (-1.7,0)  {\tiny $2$};
           \node [below] at (-0.8,0)  {\tiny $1$};
                      \node [wei2,below] at (-0.2,0)  {\tiny $0$};
           \node [below] at (0.3,0)  {\tiny $0$};
                      \node [below] at (1.4,0)  {\tiny $4$}; 

           \node [below] at (0.1,0)  {\tiny $0$};
                      \node [below] at (1.2,0)  {\tiny $4$}; 
\node [below] at (2.3,0)  {\tiny $3$}; 
\node [below] at (3.4,0)  {\tiny $2$}; 
\node [below] at (4.5,0)  {\tiny $1$};  
  \end{tikzpicture}
\]
\caption{
   The  diagram  $B_\SSTU$  for $\SSTU$ as in Figure \ref{semistandard}.   
 }
\label{passing2}
\end{figure}
\end{eg}

\subsection{A vector space isomorphism over $\Bbbk$ and decomposition numbers over $\mathbb{C}$}

The purpose of this section is to establish the graded vector space isomorphisms between our subquotient algebras. We proceed in two steps. First, we show that for an adjacency-free residue set, we can construct a graded vector space isomorphism which allows us to address this question one-residue-at-a-time. We then construct the graded vector space isomorphisms between subquotients corresponding to a single residue.
This allows us to immediately deduce the decomposition numbers of these algebras over the complex field.

Given $\gamma$ an $i$-admissible multipartition,
 we  denote the addable $i$-nodes of $\gamma$ 
 by 
 $A_1,  A_2,  \dots, A_a$ 
 so that $\mathbf{i}_{A_j} \leq \mathbf{i}_{A_k}$ 
 if and only if $j<k$.   
 %
%
 Given $\lambda  \in \Gamma$ and $1\leq k \leq m$ we let 
 $\sigma_k(\lambda)$ denote the minimal number  such that
\[
 |\{ 	A_1,\ldots, A_{\sigma_k(\lambda)}	\} \cap [\lambda]|=k.
\]
%
%
 We define a length function on $\Gamma$ as follows. 
  Given
 $\lambda,\mu \in \Gamma$ such that  $\lambda\trianglerighteq \mu$, 
we define 
\[
\ell(\lambda,\mu)
= \sum_{\mathclap{1\leq k \leq m}} \; \sigma_k(\mu)-\sigma_k(\lambda).
\]

\begin{eg}
 Let $e=4$ and  $\gamma$,  $\kappa$, and $\theta$ be as in \cref{admiseg}.
 Let $i=3$,    $\lambda=((4,2^2,1^2), (5,2^2,1))$ and 
 $\mu=((4,2,1^4), (4,2^2,1^2))$.
We have  $\sigma_1(\mu)=1$, $\sigma_2(\mu)=4$, $\sigma_3(\mu)=5$ 
and $\sigma_1(\lambda)=1$, $\sigma_2(\lambda)=2$, $\sigma_3(\lambda)=3$
and therefore 
$\ell(\lambda,\mu)=4$.  
\end{eg}

\begin{defn}
Let $\SSTT\in\SStd(\lambda,\mu)$ for $\lambda,\mu\in\Gamma$.
We define the \emph{component word} $R(\SSTT)$ of 
$\SSTT$  by
\[
R(\SSTT)=(\SSTT(A_{\sigma_1(\lambda)}),\SSTT(A_{\sigma_2(\lambda)}),	 \ldots , \SSTT(A_{\sigma(\lambda)}) 	).
\]
\end{defn}
\begin{prop}
Given $\lambda,\mu\in\Gamma$, a tableau $\SSTT \in \SStd(\lambda,\mu)$ 
 is uniquely determined by its   component word $R(\SSTT)$.  
\end{prop}
 
 \begin{proof}
  Any node belonging to $\gamma$ is   simply mapped to itself under 
  $\SSTT$.  
 Therefore
 a tableau
   $\SSTT\in  \SStd(\lambda,\mu)$    is 
 an identification of  $m$ pairs of  nodes 
  $( {(r,c,k)}, {(r',c',k')})$ 
 for $(r,c,k)\in [\lambda\setminus \gamma]$ and 
 $(r',c',k')\in [ \mu\setminus \gamma]$
  such that    $\mathbf{i}_{(r',c',k')} \geq \mathbf{i}_{(r,c,k)}$.  
   It is clear that $R(\SSTT)$ uniquely determines this identification of nodes 
   in $[\lambda\setminus \gamma]$ and $ [ \mu\setminus \gamma]$ (and vice versa) and so the result follows.  
 \end{proof}

 Let $i  \in \ZZ/e\ZZ$ and $\overline i  \in \ZZ/\bar e\ZZ$. Suppose 
  $\gamma$ and     $\overline\gamma$ are $i$-admissible and $\overline  i$-admissible 
multipartitions, respectively,   
 such that  $|\Add_i(\gamma)|=|\Add_{\overline  i}( \overline\gamma)|$.
(Note that we do not assume that $\gamma$ and $\overline{\gamma}$ have the same level or degree.)
 Given $\lambda \in \Gamma$,  
 we let $\overline\lambda \in \overline\Gamma$ 
 denote the multipartition such that
  $A_{\sigma_k(\lambda)}=A_{\sigma_k(\overline\lambda)}$  
  for all $1\leq k \leq m$.  
  We define  
   a bijection $\phi: \SStd (\lambda,\mu) \to 
   \SStd (\overline\lambda,\overline\mu)$  
 which takes $\SSTT$ to the unique  $\overline\SSTT$ such that     $R(\SSTT)=R(\overline\SSTT)$.  
 We  let
  $\Phi: A_{ {\Gamma}} 
 \mapsto A_{ \overline{\Gamma}} 
$  denote the lift of $\phi$ to the cellular bases of these algebras.

\begin{eg}\label{anexampleofanisomorpshism}
Let  $e=5$, $\kappa=(0)$, $g=0.99$, $\theta=(0)$, and $i=0$.  
 The partition $\gamma=(5,1^4)$ is $0$-admissible and $\Gamma
 =\{(6,1^4), (5,2,1^3), (5,1^5)\}$.  
 Recall that there is a unique $\SSTT\in \SStd((6,1^4),(5,1^5))$ of degree  $2$, as depicted in \cref{semistandard}.  

Let  $\overline e=11$, $\overline\kappa=(1,1,1)$,   $\overline\theta=(-5,    0,  4)$, and $\overline i =1$.  
 The multipartition $\overline\gamma=(\varnothing, (2,1), \varnothing)$ is $1$-admissible and $\overline\Gamma=\{ (6,1^4), (5,2,1^3),(5,1^5)\}$.  
There is a  unique $\SSTT\in \SStd(((1), (2,1), \varnothing),$ $(\varnothing, (2,1),(1)))$ of degree  $2$.  

The image under the map $\Phi:A_{\Gamma}\to  A_{\overline\Gamma}$ of the element $B_\SSTT$ in \cref{passing} is given in \cref{isomopassing}, below.

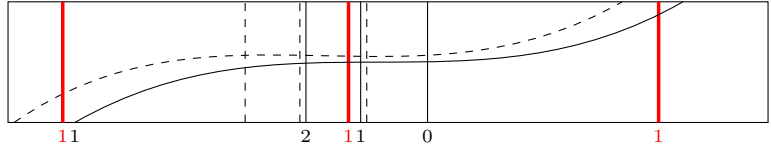
\begin{figure}[ht!]
  \[    \scalefont{0.7}  
   \begin{tikzpicture}[scale=0.8] 
     \clip(-5.6,-0.5)rectangle (7,2.2);
  \draw (-5.5,0) rectangle (7,2);
   \foreach \x in {-0.6,0.3}  
     {\draw  (\x,0)--(\x,2); \draw[dashed]  (\x-1,0)--(\x-1,2); }
       \draw[wei2] (0.1,0)  -- (0.1,2) ;  
           \node [ below] at (0.3,0)  {\tiny $1$};           
           \node [ below] at (-0.6,0)  {\tiny $2$};           
           \node [ below] at (-4.4,0)  {\tiny $1$}; 
           \node [wei2,below] at (-4.6,0)  {\tiny $1$}; 
                      \node [wei2,below] at (0.1,0)  {\tiny $1$}; 
  \draw[dashed] (0.4,0)  --(0.4,2); 
    \draw  (1.4,0)  --(1.4,2); 
               \node [ below] at (1.4,0)  {\tiny $0$};           
           \node [wei2,below] at (5.2,0)  {\tiny $1$};

   \draw[dashed]  (-4.4-1,0) to [out=35,in=-150]  (5.6-1,2) ;      
 \draw[wei2]  (-4.6,0) -- (-4.6,2) ;      
 \draw[wei2]  (5.2,0) -- (5.2,2) ;  
 \draw  (-4.4,0)  to [out=30,in=-150]  (5.6,2) ;      
         
  \end{tikzpicture}
\]
\caption{Image of the diagram  in Figure \ref{passing} under $\Phi$, 
as in \cref{anexampleofanisomorpshism}. }
\label{isomopassing}
\end{figure}
\end{eg}
   

\begin{prop}\label{vsisom}
 We have that 
$ 
A_{\Gamma}  $ and $ A_{\overline\Gamma} $ 
are isomorphic as graded vector spaces over  $\Bbbk$; the isomorphism is given by 
$\Phi(C_{\SSTS,\SSTT})= C_{\overline{\SSTS},\overline{\SSTT}}$.
  This isomorphism preserves both the length function   and  the graded 
  characters of standard   modules.  In other words
  \[
  \ell(\lambda,\mu) =   \ell(\overline\lambda,\overline\mu) 
  \quad\quad
  \Dim{(  \Delta_\mu(\lambda))} = \Dim{ (\Delta_{\overline\mu}(\overline\lambda))}
  \] for all $\lambda,\mu\in\Gamma$.  
\end{prop}

\begin{proof}
We begin by explicitly describing the effect of $\Phi$ on basis elements.  
Given $(\SSTS,\SSTT)\in\SStd(\lambda,\mu)\times \SStd(\lambda,\nu)$, the diagram 
of $C_{\overline{\SSTS},\overline{\SSTT} }$ may be obtained from that of 
$C_{\SSTS, \SSTT}$  as follows.
\begin{enumerate}
\item Take the diagram corresponding to $\SSTT\in\SStd(\lambda,\mu)$ and simply ``forget" all black strands (and their ghosts) corresponding to nodes of $[\gamma]$, as well as all red strands (which are at $x$-coordinates given by $\theta$). What remains is a diagram  involving $m$ $i$-strands   whose northern and southern points   belong to the set $\{x\mid \mathcal{D}_x \text{ is visible}\}$.
\item Isotopically deform the $i$-strands (along with their ghosts) 
and their northern and southern end points (which are initially   given by the loadings $\mathbf{i}_\lambda$ and $\mathbf{i}_\mu$ respectively)  until the 
northern and southern end points are given by the corresponding loadings
 $\mathbf{i}_{\overline\lambda}$ and $\mathbf{i}_{\overline\mu}$ respectively. Now change the  label of all of these  strands from $i$ to $\overline i$.
\item 
 Finally, add the  
black vertical strands (and their ghosts) corresponding to nodes of $[\overline\gamma]$, as well as all red strands (which are at $x$-coordinates given by $\overline\theta$).
\end{enumerate}
That the map $\Phi$ is an isomorphism of vectors spaces is clear from the fact that the corresponding semistandard tableaux are in bijection.

 We now wish to show that $\Phi$ is  degree preserving.   
Let   $(\SSTS, \SSTT) \in \SStd(\lambda,\mu)\times \SStd(\lambda,\nu)$ and 
let $A$ denote any strand in the diagram $C_{\SSTS, \SSTT}$ which corresponds to a removable $i$-node of $\lambda$ for $\lambda\in\Gamma$.  
By assumption (and the definition of $\Phi$), the strand $A$ is common to the diagrams of both $C_{\SSTS, \SSTT}$ and $  C_{\overline{\SSTS}, \overline{\SSTT}}$;  
we wish to count  the  degree contribution of $A$ in each case.  
 Degree contributions are made whenever the strand $A$ passes through an
 $i$-diagonal in $\gamma$ (respectively  $\overline i$-diagonal in $\overline{\gamma}$) or an $i$-strand (respectively $\overline i$-strand) corresponding to a node
 in $\lambda\setminus \gamma$ (respectively $\overline{\lambda}\setminus \overline{\gamma}$).
   The strands labelled by  nodes
 in $\lambda\setminus \gamma$ and $\overline{\lambda}\setminus \overline{\gamma}$
  are common to both  
  $C_{\SSTS, \SSTT}$ and $  C_{\overline{\SSTS}, \overline{\SSTT}}$ (with appropriate relabelling of residues), so we need 
 only consider the degree contributions arising from $A$ and its ghost crossing the $i$-diagonals in $\gamma$ or the $\overline i$-diagonals in $\overline{\gamma}$.

 If  $A$  passes through a $\mathbf{B}_1$ brick 
  then the degree contribution of this crossing is 0.  
If  $A$  passes through a brick $\mathbf{B}_k$ for $k=4,5,6$  then the degree contribution of this crossing is $+1$.  
If  $A$  passes through a brick $\mathbf{B}_k$ for $k=2,3$  then the degree contribution of this crossing is $-1$.

Let  $\mathbf{D}$ be a  diagonal in the diagram $B_{\SSTT}$  and suppose that 
 $A$  passes through $\mathbf{D}$.   
A visible   diagonal is built out of 
  a single  $\mathbf{B}_k$ brick  for $k\in\{4,5,6\}$ and
  some number (possibly zero) of $\mathbf{B}_1$ bricks.
 An invisible diagonal  has an extra  
  single   $\mathbf{B}_k$ brick    for $k\in\{2,3\}$.  
  Summing over the degrees, we conclude that the 
  crossing of  $A$  with a visible (respectively invisible) 
  $i$-diagonal has degree $+1$ (respectively 0).  Similarly for the crossing of $A$ with a $\overline i$-diagonal in $ C_{\overline{\SSTS}, \overline{\SSTT}}$. 

That $A_{\Gamma}$ and $A_{\overline\Gamma}$ are isomorphic as graded vector spaces now  follows as $\Phi$ maps visible (respectively invisible) $i$-diagonals to 
visible (respectively invisible) $\overline i$-diagonals.  
 That the length function is preserved is clear.  
That the 
 graded dimensions of standard modules are preserved is clear from
 the definition of $\Phi$ on the level of semistandard tableaux of a given shape and weight.  
  \end{proof}

We now momentarily focus our attention on the representation theory of these algebras over 
$\mathbb{C}$.  In this case the diagrammatic Cherednik algebas are Koszul, and so 
 graded decomposition numbers
 are particularly easy to calculate.  In particular, we have the following theorem.  

\begin{thm}\label{tanteolironcrispy2}
  Let $\gamma$  
 be an   $i$-admissible multipartition.  
The graded decomposition numbers   of $A(n,\theta,\kappa)$  
over $\mathbb{C}$ can be given in terms of nested sign sequences 
  as follows
\[
d_{\lambda\mu}(t)=\sum_{\mathclap{\omega \in \Omega(\lambda,\mu)}} \; t^{\|\omega\|}
\]
for $\lambda,\mu\in\Gamma$ such that $\mu\trianglelefteq_\theta \lambda$.
\end{thm}

\begin{proof}
For   $\lambda,\mu \in\Gamma$, recall that  
 $ \Dim{(\Delta_\mu(\lambda))}  \in \ZZ_{\geq 0}[t,t^{-1}]$ (by the definition of the grading on basis elements),
  $ \Dim{(L_\mu(\lambda))} \in \ZZ_{\geq 0}[t+t^{-1}]$ (see \cref{humathasprop}),
 and for $\lambda \neq \mu$ we have that
  $d_{\lambda\mu} (t) \in t\ZZ_{\geq 0}[t]$ (see \cref{step1}).    
  It is clear that a necessary condition for the multiplicity of 
 $L(\mu)$ in $\Delta(\lambda)$ to be non-zero is that $\SStd(\lambda,\mu)\neq \emptyset$.  It is also clear that   $\Dim{(\Delta_\lambda(\lambda))}=1=\Dim{(L_\lambda(\lambda))}$.  Therefore the first five conditions of \cite[Theorem 3.8]{KN10} are satisfied, and so
 \[
\Dim{(\Delta_\mu(\lambda)})= \;
\sum_
{  \mathclap{\begin{subarray}c \nu \neq \mu \\ 
   \SStd(\nu,\mu)\neq \emptyset\\
  \SStd(\lambda,\nu)\neq \emptyset \end{subarray} }} \; d_{\lambda\nu}(t)\Dim{(L_\mu(\nu))}+d_{ \lambda\mu}(t).
  \]
Therefore, one can calculate the graded characters of simple modules and the decomposition numbers of $A_{\Gamma}$   by induction on the 
distance, $\ell(\lambda,\mu)$, for $\lambda,\mu\in \Gamma$  
 exactly as in \cite[Main Algorithm]{KN10} and \cite[Theorem 1.18]{bcs15}.  
 By \cref{vsisom}, if we do this for $\lambda,\mu\in\Gamma$ or 
 $\overline\lambda,\overline\mu\in \overline\Gamma$ we get exactly the same answer!  Therefore the decomposition numbers of these algebras and the graded characters of simple modules are the same regardless of the weighting, $e$-multicharge, rank and level. 
In particular,  we can run this algorithm for $\overline\gamma$ a level 1 partition   with $m$ addable   $i$-nodes.  
The result now follows by \cite[Theorem 4.4]{tanteo13} and \cref{3.6rmk}.
\end{proof}

\subsection{Some useful results on moving $i$-strands through diagonals} 
 
There are several sequences of relations which we will often apply in 
particular order 
during the course of the proof.  For brevity, we shall now define these
 as Moves 1, $1^*$, and 2.  
 
\begin{move1}
Suppose we have a diagram in which two $j$-ghosts  are not separated by a black $(j+1)$-strand, and the corresponding two black $j$-strands are separated by a $(j-1)$-ghost. We apply relation \ref{rel10} to write this diagram as the difference of two diagrams in which the $j$-strands cross and the $(j-1)$strand bypasses the crossing to the left or right; see Figure \ref{e1} for an example.
\end{move1}

\begin{move1'}  
Suppose we have a diagram in which two black $j$-strands  are not separated by a $(j-1)$-ghost, and the corresponding two $j$-ghosts are separated by a black $(j+1)$-strand. One can repeat the above using relation \ref{rel9} in place of \ref{rel10}.
\end{move1'}

\begin{move2}
Suppose we have a diagram with a pair of  crossing $j$-strands. 
We may use relations \ref{rel3} and \ref{rel4} to rewrite the single crossing as a double crossing with a dot on the leftmost strand located between the two crossings; see the first equality in each of \cref{e2,e3} for an example.
\end{move2}

%

  \begin{eg}
 Let $e=3$, $\kappa=(2,0)$, $g=0.99$ and $\theta=(0,1)$.
 The leftmost diagram in Figure \ref{e1} is the idempotent corresponding to the loading of  the bicomposition $(\varnothing, (1,2))$. 
 Applying Move 1 to the two adjacent  $0$-ghosts, we obtain the difference of two diagrams depicted in  \cref{e1}.

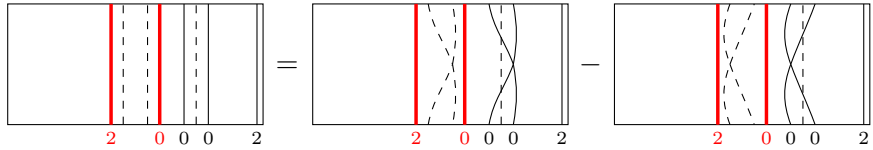
\begin{figure}[ht!]
\[\scalefont{0.6}
    \begin{minipage}{3.5cm}  \begin{tikzpicture}[scale=0.8] 
       \clip(-1.1,-0.5)rectangle (3.4,2.5);
   \draw (-1,0) rectangle (3.2,2);   
   \draw[wei2](0.7,0) -- (0.7,2);
 
       \draw[wei2](1.5,0) -- (1.5,2);
       \draw(1.9,0)--(1.9,2);       \draw(2.3,0)--(2.3,2);       \draw(3.1,0)--(3.1,2);
      \draw[dashed](-1+1.9,0)--(-1+1.9,2);       \draw[dashed](-1+2.3,0)--(-1+2.3,2);       \draw[dashed](-1+3.1,0)--(-1+3.1,2);
         \node [below] at (2.3,0)  {  $0$};     
                  \node [below] at (1.9,0)  {  $0$};     
                           \node [below] at (3.1,0)  {  $2$};     
              \node [wei2,below] at (0.7,0)  {  $2$};           
                                       \node [wei2,below] at (1.5,0)  {  $0$};      
             \end{tikzpicture}\end{minipage}
             =
                \begin{minipage}{3.5cm}  \begin{tikzpicture}[scale=0.8] 
       \clip(-1.1,-0.5)rectangle (3.4,2.5);
   \draw (-1,0) rectangle (3.2,2);   
   \draw[wei2](0.7,0) -- (0.7,2);
 
       \draw[wei2](1.5,0) -- (1.5,2);
                   \draw[dashed]  (-1+1.9,0)  to [out=80,in=-110] (-1+2.3,1) ;   
                  \draw[dashed]  (-1+2.3,1)  to [out=80,in=-80] (-1+2.3,2) ;   
         \draw[dashed]  (-1+1.9,2)  to [out=-80,in=110] (-1+2.3,1) ;   
                  \draw[dashed]  (-1+2.3,1)  to [out=-80,in=80] (-1+2.3,0) ;
         \draw  (1.9,0)  to [out=80,in=-110] (2.3,1) ;   
                  \draw  (2.3,1)  to [out=80,in=-80] (2.3,2) ;   
         \draw  (1.9,2)  to [out=-80,in=110] (2.3,1) ;   
                  \draw  (2.3,1)  to [out=-80,in=80] (2.3,0) ;  
     \draw[dashed](2.1,0)--(2.1,2);   
          \draw (3.1,0)--(3.1,2);   
         \node [below] at (2.3,0)  {  $0$};     
                  \node [below] at (1.9,0)  {  $0$};     
                           \node [below] at (3.1,0)  {  $2$};     
              \node [wei2,below] at (0.7,0)  {  $2$};           
                                       \node [wei2,below] at (1.5,0)  {  $0$};      
              \end{tikzpicture}\end{minipage}
-
             \begin{minipage}{3.5cm}  \begin{tikzpicture}[scale=0.8] 
       \clip(-1.1,-0.5)rectangle (3.4,2.5);
   \draw (-1,0) rectangle (3.2,2);   
   \draw[wei2](0.7,0) -- (0.7,2);
 
       \draw[wei2](1.5,0) -- (1.5,2);
          \draw  (1.9,0)  to [out=110,in=-110] (1.9,1) ;   
                  \draw  (1.9,1)  to [out=70,in=-110] (2.3,2) ;   
                           \draw  (1.9,2)  to [out=-110,in=110] (1.9,1) ;   
                  \draw  (1.9,1)  to [out=-70,in=110] (2.3,0) ;   
                           \draw[dashed]  (-1+1.9,0)  to [out=110,in=-110] (-1+1.9,1) ;   
                  \draw[dashed]  (-1+1.9,1)  to [out=70,in=-110] (-1+2.3,2) ;   
                           \draw[dashed]  (-1+1.9,2)  to [out=-110,in=110] (-1+1.9,1) ;   
                  \draw[dashed]  (-1+1.9,1)  to [out=-70,in=110] (-1+2.3,0) ;   
     \draw[dashed](2.1,0)--(2.1,2);   
       \draw (3.1,0)--(3.1,2);   
         \node [below] at (2.3,0)  {  $0$};     
                  \node [below] at (1.9,0)  {  $0$};     
                           \node [below] at (3.1,0)  {  $2$};     
              \node [wei2,below] at (0.7,0)  {  $2$};           
                                       \node [wei2,below] at (1.5,0)  {  $0$};      
             \end{tikzpicture}\end{minipage}
\]
\caption{Rewriting the  idempotent corresponding to the loading of $(\varnothing,(1,2))$
using relation \ref{rel10}.}
\label{e1}
\end{figure}

  In \cref{e2,e3} we first rewrite the  right-hand side of the equality in
   \cref{e1} using Move 2.  
 We then use relation \ref{rel6} (whose error term is zero by relation \ref{rel4})
in each case to  obtain an element  which factors through the idempotent 
  $((1^2), (1))$ or an unsteady idempotent, respectively.

 \begin{figure}[ht!]
\[\scalefont{0.6}
                         \begin{minipage}{3.5cm}  \begin{tikzpicture}[scale=0.8] 
       \clip(-1.1,-0.5)rectangle (3.4,2.5);
   \draw (-1,0) rectangle (3.2,2);   
   \draw[wei2](0.7,0) -- (0.7,2);
 
       \draw[wei2](1.5,0) -- (1.5,2);
                   \draw[dashed]  (-1+1.9,0)  to [out=80,in=-110] (-1+2.3,1) ;   
                  \draw[dashed]  (-1+2.3,1)  to [out=80,in=-80] (-1+2.3,2) ;   
         \draw[dashed]  (-1+1.9,2)  to [out=-80,in=110] (-1+2.3,1) ;   
                  \draw[dashed]  (-1+2.3,1)  to [out=-80,in=80] (-1+2.3,0) ;
         \draw  (1.9,0)  to [out=80,in=-110] (2.3,1) ;   
                  \draw  (2.3,1)  to [out=80,in=-80] (2.3,2) ;   
         \draw  (1.9,2)  to [out=-80,in=110] (2.3,1) ;   
                  \draw  (2.3,1)  to [out=-80,in=80] (2.3,0) ;  
     \draw[dashed](2.1,0)--(2.1,2);   
          \draw (3.1,0)--(3.1,2);   
         \node [below] at (2.3,0)  {  $0$};     
                  \node [below] at (1.9,0)  {  $0$};     
                           \node [below] at (3.1,0)  {  $2$};     
              \node [wei2,below] at (0.7,0)  {  $2$};           
                                       \node [wei2,below] at (1.5,0)  {  $0$};      
              \end{tikzpicture}\end{minipage}=        \begin{minipage}{3.5cm}  \begin{tikzpicture}[scale=0.8] 
       \clip(-1.1,-0.5)rectangle (3.4,2.5);
   \draw (-1,0) rectangle (3.2,2);   
   \draw[wei2](0.7,0) -- (0.7,2);
 
       \draw[wei2](1.5,0) -- (1.5,2);
                   \draw  (2.3,0)  -- (2.3,2) ;   
                    \node [below] at (2.3,0)  {};  
                            \draw[dashed]  (-1+1.9,2)  to [out=-80,in=90] (-1.1+2.5,1) ;   
                  \draw[dashed]  (-1+1.9,0)  to [out=80,in=-90] (-1.1+2.5,1) ;   
                  \draw[dashed]  (-1+2.3,0)  to [out=90,in=-90] (-1+2.3,2) ;   
         \draw  (1.9,2)  to [out=-80,in=90] (2.5,1) ;   
                  \draw  (1.9,0)  to [out=80,in=-90] (2.5,1) ;   
                                \fill (2.3,1)  circle (3pt); 
      \draw[dashed](2.1,0)--(2.1,2);   
          \draw (3.1,0)--(3.1,2);   
         \node [below] at (2.3,0)  {  $0$};     
                  \node [below] at (1.9,0)  {  $0$};     
                           \node [below] at (3.1,0)  {  $2$};     
              \node [wei2,below] at (0.7,0)  {  $2$};           
                                       \node [wei2,below] at (1.5,0)  {  $0$};      
              \end{tikzpicture}\end{minipage}
=        \begin{minipage}{3.5cm}  \begin{tikzpicture}[scale=0.8] 
       \clip(-1.1,-0.5)rectangle (3.4,2.5);
   \draw (-1,0) rectangle (3.2,2);   
   \draw[wei2](0.7,0) -- (0.7,2);
       \draw[wei2](1.5,0) -- (1.5,2);
                   \draw  (2.3,2)  to [out=-120,in=90]  (0,1) ;   
                   \draw  (2.3,0)  to [out=120,in=-90]  (0,1) ;   
                    \node [below] at (2.3,0)  {};  
         \draw  (1.9,2)  -- (1.9,0) ;   
                  \draw[dashed]  (1.9-1,2)  -- (1.9-1,0) ;   
      \draw[dashed](2.1,0)to [out=160,in=-90] (-0.1,1);
     \draw[dashed](2.1,2)to [out=-160,in=90] (-0.1,1);
          \draw[dashed](.5,2)to [out=-150,in=90] (0.2-1,1);
                    \draw[dashed](.5,0)to [out=150,in=-90] (0.2-1,1);
 \draw  (0.8,1)  to [out=80,in=-90] (3.1,2);
 \draw  (0.8,1)  to [out=-80,in=90] (3.1,0);
         \node [below] at (2.3,0)  {  $0$};     
                  \node [below] at (1.9,0)  {  $0$};     
                           \node [below] at (3.1,0)  {  $2$};     
              \node [wei2,below] at (0.7,0)  {  $2$};           
                                       \node [wei2,below] at (1.5,0)  {  $0$};      
              \end{tikzpicture}\end{minipage}
              \]
              \caption{Rewriting one of the diagrams in Figure \ref{e1} 
              as an element which factors through the idempotent labelled by
              $((1^2),(1))$.}
              \label{e2}
\end{figure}
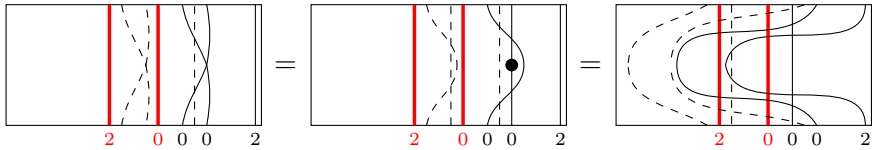

%

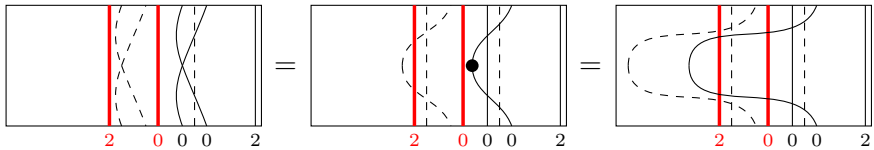
\begin{figure}[ht!]
\[\scalefont{0.6}
     \begin{minipage}{3.5cm}  \begin{tikzpicture}[scale=0.8] 
       \clip(-1.1,-0.5)rectangle (3.4,2.5);
   \draw (-1,0) rectangle (3.2,2);   
   \draw[wei2](0.7,0) -- (0.7,2);
 
       \draw[wei2](1.5,0) -- (1.5,2);
          \draw  (1.9,0)  to [out=110,in=-110] (1.9,1) ;   
                  \draw  (1.9,1)  to [out=70,in=-110] (2.3,2) ;   
                           \draw  (1.9,2)  to [out=-110,in=110] (1.9,1) ;   
                  \draw  (1.9,1)  to [out=-70,in=110] (2.3,0) ;   
 \draw[dashed]  (-1+1.9,0)  to [out=110,in=-110] (-1+1.9,1) ;   
                  \draw[dashed]  (-1+1.9,1)  to [out=70,in=-110] (-1+2.3,2) ;   
                           \draw[dashed]  (-1+1.9,2)  to [out=-110,in=110] (-1+1.9,1) ;   
                  \draw[dashed]  (-1+1.9,1)  to [out=-70,in=110] (-1+2.3,0) ;   
     \draw[dashed](2.1,0)--(2.1,2);   
       \draw (3.1,0)--(3.1,2);   
         \node [below] at (2.3,0)  {  $0$};     
                  \node [below] at (1.9,0)  {  $0$};     
                           \node [below] at (3.1,0)  {  $2$};     
              \node [wei2,below] at (0.7,0)  {  $2$};           
                                       \node [wei2,below] at (1.5,0)  {  $0$};      
             \end{tikzpicture}\end{minipage}
=
  \begin{minipage}{3.5cm}  \begin{tikzpicture}[scale=0.8] 
       \clip(-1.1,-0.5)rectangle (3.4,2.5);
   \draw (-1,0) rectangle (3.2,2);   
   \draw[wei2](0.7,0) -- (0.7,2);
 
       \draw[wei2](1.5,0) -- (1.5,2);
           \draw  (1.9,0)  -- (1.9,2) ; 
                                            \fill (1.65,1)  circle (3pt); 
     \draw  (1.65,1)  to [out=-90,in=110] (2.3,0) ;   
     \draw  (1.65,1)  to [out=90,in=-110] (2.3,2) ;  
         \draw[dashed]  (-1+1.9,0)  -- (-1+1.9,2) ; 
      \draw[dashed]  (-1+1.5,1)  to [out=-90,in=110] (-1+2.3,0) ;   
     \draw[dashed]  (-1+1.5,1)  to [out=90,in=-110] (-1+2.3,2) ;   
      \draw[dashed](2.1,0)--(2.1,2);   
       \draw (3.1,0)--(3.1,2);   
          \node [below] at (2.3,0)  {  $0$};     
                  \node [below] at (1.9,0)  {  $0$};     
                           \node [below] at (3.1,0)  {  $2$};     
              \node [wei2,below] at (0.7,0)  {  $2$};           
                                       \node [wei2,below] at (1.5,0)  {  $0$};      
             \end{tikzpicture}\end{minipage}
             =
  \begin{minipage}{3.5cm}  \begin{tikzpicture}[scale=0.8] 
       \clip(-1.1,-0.5)rectangle (3.4,2.5);
   \draw (-1,0) rectangle (3.2,2);   
   \draw[wei2](0.7,0) -- (0.7,2);
 
       \draw[wei2](1.5,0) -- (1.5,2);
           \draw  (1.9,0)  -- (1.9,2) ; 
      \draw  (0.2,1)  to [out=-90,in=110] (2.3,0) ;   
     \draw  (0.2,1)  to [out=90,in=-110] (2.3,2) ;  
         \draw[dashed]  (-1+1.9,0)  -- (-1+1.9,2) ; 
      \draw[dashed]  (-1+0.2,1)  to [out=-90,in=110] (-1+2.3,0) ;   
     \draw[dashed]  (-1+0.2,1)  to [out=90,in=-110] (-1+2.3,2) ;   
      \draw[dashed](2.1,0)--(2.1,2);   
       \draw (3.1,0)--(3.1,2);   
          \node [below] at (2.3,0)  {  $0$};     
                  \node [below] at (1.9,0)  {  $0$};     
                           \node [below] at (3.1,0)  {  $2$};     
              \node [wei2,below] at (0.7,0)  {  $2$};           
                                       \node [wei2,below] at (1.5,0)  {  $0$};      
             \end{tikzpicture}\end{minipage}
             \]
              \caption{Rewriting the other   diagram  in Figure \ref{e1} 
              as an element which factors through an unsteady idempotent.}
              \label{e3}
\end{figure}

\end{eg}

The following lemmas shall also be useful in what follows.  

\begin{lem}
\label{remove}       
If   $d\in A(n,\theta,\kappa)$ factors through some loading $\mathbf{i}$ such that 
 $ \mathbf{i} \vartriangleright \mathbf{i}_{\gamma^+}$,  then   $d=0$ in $A_{\Gamma}$. 
 In particular, 
 if $d$ factors through $1_\lambda$ with 
   $\lambda \vartriangleright_{(\theta,j)} \gamma^+$  
 for some $j\neq i$, then $d=0$ in $A_{\Gamma}$.  
 \end{lem}
\begin{proof}
%
 The result  is immediate by the definition of $A_{\Gamma}$, the definition of the dominance order,  and \cref{215}.    
 \end{proof}

\begin{lem}\label{move3}
Let  
 $A\in [\gamma]$ be at point $x\in \RR$ and suppose $A \not \in \Remb
 (\gamma)$. 
 We   let $d\in A(n,\theta, \kappa)$ 
  and suppose that there is a  neighbourhood
   $(x-n\epsilon,x+n\epsilon)\times [0,1]$
   in which 
   \[d \cap ( (x-n\epsilon,x+n\epsilon) \times [0,1])=(1_{\lambda \setminus A}) \cap  ( (x-n\epsilon,x+n\epsilon) \times [0,1]),\]   
for some $\lambda \in \Gamma$.   Then $d =0$ in $A_{\Gamma}$.  
\end{lem}

\begin{proof}
Suppose $A=(r,c,k)$ is a $j$-node.  
All the relations we shall apply only involve moving strands a distance less than $n\epsilon$ to the left or right.
 Such relations applied to one component do not affect any other components, due to the fact that $\epsilon$ is very small compared to the $\theta_a-\theta_b$ for $1\leq a,b \leq l$.  We therefore need only focus on the interaction within the $k$th component.

 First note that as $(r,c,k)$ is not a removable node, there exists a node   
      $(r+1,c,k)$ or $(r,c+1,k)$  in $[\gamma]$.   
       We shall argue for the former case, but the latter is similar.

  If $c=1$   and there is no node $(r,2,k)$
  then the result follows as we need only move the nodes 
$(r+a,1,k)$ for $a\geq 1$  
to the left to obtain a loading   that $(\theta,h)$-dominates $\mathbf{i}_{\gamma^+}$ for some $h\neq i$; the result then follows by \cref{215}. 
 For $c>1$, we now provide an algorithm  for showing that
 $d=0$ in $A_{\Gamma}$.  
This  
involves   procedures on strands which we describe by the corresponding nodes in the Young diagram.  If at any point in the algorithm the node to which we refer does not exist, then we have reached the first row or column of our partition;  in which case terminate the algorithm and proceed to the  end of the proof.

\begin{itemize}
\item[Step 1] 
 The   $(j-1)$-ghosts corresponding to $(j-1)$-nodes $(r+1,c,k)$  and $(r,c-1,k)$ are not separated by a black $j$-strand;  
 we can apply Move 1 to 
 (the  strands corresponding to)
 this pair of  nodes  and the $(j-2)$-node $(r+1,c-1,k)$.  
 The result is the difference of two distinct diagrams, in which
  the $(j-2)$-strand (labelled by node $(r+1,c-1,k)$) 
  bypasses the $(j-1)$-crossing to the left and right.  Now proceed to Step 2.
%
%
%
%
%
    \item[Step 2] 
\begin{itemize}
\item[$(a)$] Consider the diagram in which   the $(j-2)$-strand bypasses to the left.
 Observe that the    $(j-2)$-ghosts 
 labelled by nodes $(r+1,c-1,k)$ and $(r,c-2,k)$  have no black strand separating them.  The black     $(j-2)$-strands 
 labelled by nodes $(r+1,c-1,k)$ and $(r,c-2,k)$ are separated by the 
 $(j-3)$-ghost strand labelled by the node $(r+1,c-2,k)$.  
 We 
 now set $\bar{j}:=j-1$ and $(\bar{r},\bar{c},k):=(r,c-1,k)$ and
 (using  the barred residues and node labels as the input) proceed to Step 1.  
 \item[$(b)$]
 Consider the diagram in which   the $(j-2)$-strand bypasses to the right.
  Apply Move 2 to the crossing $(j-1)$-strands.  
Transpose the labels of the $(j-1)$-strands  corresponding to nodes
 $(r,c-1,k)$  and  $(r+1,c,k)$ 
 (as their order when read from left to right has switched); 
 this results in the dotted strand being labelled by $(r,c-1,k)$.  
    Push the ghost of the 
 dotted   $(j-1)$-strand through the black $j$-strand immediately to its left by relation \ref{rel6} (observe that the error term in \ref{rel6} is zero by relation \ref{rel4}).  
  Observe that the    $(j-1)$-ghosts labelled by nodes $(r,c-1,k)$ and $(r-1,c-2,k)$ 
 have no black strand separating them.
   The black strands 
  labelled by nodes $(r,c-1,k)$ and $(r-1,c-2,k)$ are separated by a   $(j-2)$-ghost 
  labelled by the node $(r,c-2,k)$.
 Therefore we relabel $(\bar r,\bar c,\bar k):=(r-1,c-1,k)$ and $\bar j:=j$ and proceed to Step 1.
  \end{itemize}
 %
%
%
%
%

     \end{itemize}
     The algorithm terminates  if at the end of   Step 2 
     case $(a)$ we set  $\bar c=0$
 and in case $(b)$ we set  $\bar r=0$ or  $\bar c=0$.  
If we terminate in  case $(a)$, then our diagram has a crossing pair of 
black $(\bar j-1)$-nodes labelled by $(\bar r,1)$ and $(\bar r+1,2)$ bypassed by a ghost 
$(\bar j-2)$-strand to the left.  
 We can pull the $(\bar j-2)$-strand  at least $n\epsilon$ units to the left; 
  we then  apply Move 2 to the crossing $(\bar j-1)$-strands and pull the ghost 
  dotted $(\bar j-1)$-strand through the black $\bar j$-strand immediately to its left.
 The loading at $y=1/2$ in the resulting diagram  $(\theta,h)$-dominates $\mathbf{i}_{\gamma^+}$ for $h=\bar j-1,\bar j-2$.  
 The result follows from \cref{remove}.  
  Case $(b)$ is similar.  
\end{proof}

\begin{lem}\label{move2.5}
Given $\lambda \in \Gamma$,  
 if we add a dot to any of the strands in $1_\lambda$ corresponding to a node  in $\gamma$, then
 the resulting diagram is zero in $A_{\Gamma}$.    \end{lem}
   
   \begin{proof}
 Let $A=(r,c,k)$ denote a 
 $j$-node in $\gamma$  
 with a dot on the corresponding strand.  
  We proceed by induction on $r+c$; in the case $(r,c,k)=(1,1,k)$, $A$ can pass through the red $j$-strand immediately to its left using relation  \ref{rel11}.  
     If $j\neq i$ then the diagram is zero by \cref{remove}.  
  If $j=i$, then our assumption that $(r,c,k)\in \gamma$ for $\gamma$ 
   $i$-admissible implies that there is either an $(i-1)$-node $(2,1,k)$
   or an $(i+1)$-node $(1,2,k)$.  In either case,  the diagram is zero by  \cref{remove}.  
    
We now assume that   $r+c\geq 2$. 
We can pull $A$ through the $(j+1)$-ghost to its left, labelled by $(r-1,c,k)$, at the expense of losing the dot (we also obtain an error term $1_\lambda$ with a dot on the $(j+1)$-strand labelled by $(r-1,c,k)$, which is zero by induction).  
We now apply Move 1 to the $j$-ghosts labelled by nodes $(r,c,k)$ and $(r-1,c-1,k)$  and the $(j+1)$-strand labelled by $(r-1,c,k)$, to obtain two terms. The term in which the $(j+1)$-strand bypasses the crossing of $j$-ghosts to the left is zero by \cref{move3}.

Now consider the remaining term in which
  the $(j+1)$-strand bypasses the crossing of $j$-ghosts to the right. If $j\neq i$, the result follows by \cref{remove}. If $j=i$, we continue by applying Move 2 and pulling the dotted strand to the left. Repeating this argument we can pull $A$ through all the $i$-strands and onwards outside of the region in \cref{move3} and the result follows.  
\end{proof}

We denote the $i$-diagonals in $\gamma$ by $\mathbf{D}_{x_1}, \mathbf{D}_{x_2}, \dots$ so that $x_a = x(\mathbf{D}_{x_a})$ and $x_a < x_b$ whenever $a<b$.
We let $b_a:=b_a(\mathbf{D})$ denote the total number of $\mathbf{B}_a$ bricks in the $i$-diagonal $\mathbf{D}$ for $a=1,\ldots , 6$.

%

\begin{prop}\label{lem:crossingthroughdiagonal}
  We can pull an $i$-crossing through an $i$-diagonal $\mathbf{D}$ at the expense of 
 an  error term, as illustrated in \cref{fig:crossingthroughdiagonal}.  
  \end{prop}
 \begin{figure}[ht!]
 
\[\scalefont{0.6} \begin{minipage}{2.8cm}  \begin{tikzpicture}[scale=0.8] 
     \clip(-1.9,-0.5)rectangle (1.9,2.2);
   \draw (-1.75,0) rectangle (1.5,2);
   \node[cylinder,draw=black,thick,aspect=0.4,
        minimum height=1.65cm,minimum width=0.3cm,
        shape border rotate=90,
        cylinder uses custom fill,
        cylinder body fill=blue!30,
        cylinder end  fill=blue!10]
  at (0,0.93) (A) { };
 \draw [smooth] (-1.25,2) to[out=-90,in=115] (1,1) to[out=-65,in=90] (1.25,0);
 \draw [smooth] (-1.25,0) to[out=90,in=-115] (1,1) to[out=65,in=-90] (1.25,2);
 \draw [dashed,smooth] (-0.4++-1.25,2) to[out=-90,in=130] (-0.4++1,1) to[out=-60,in=90] (-0.4++1.25,0);
 \draw [dashed,smooth] (-0.4++-1.25,0) to[out=90,in=-130] (-0.4++1,1) to[out=60,in=-90] (-0.4++1.25,2);
             \node [below] at (-1.25,0)  {  $i$};  \node [below] at (1.25,0)  {  $i$}; 
             \node [below] at (0,0)  {$\tiny\mathbf{D}$};  
 \end{tikzpicture}\end{minipage}
 =\begin{minipage}{2.85cm}
  \begin{tikzpicture}[scale=0.8] 
     \clip(-1.9,-0.5)rectangle (1.9,2.2);
   \draw (-1.75,0) rectangle (1.5,2);
   \node[cylinder,draw=black,thick,aspect=0.4,
        minimum height=1.65cm,minimum width=0.3cm,
        shape border rotate=90,
        cylinder uses custom fill,
        cylinder body fill=blue!30,
        cylinder end  fill=blue!10]
  at (0,0.93) (A) { };
 \draw [smooth] (-1.25,2) to[out=-90,in=160] (-0.7,1) to[out=-30,in=70] (1.25,0);
 \draw [smooth] (-1.25,0) to[out=90,in=-160] (-0.7,1) to[out=30,in=-70] (1.25,2);
 \draw [dashed,smooth] (-0.4+-1.25,2) to[out=-90,in=150] (-0.4+-0.7,1) to[out=-40,in=70] (-0.4+1.25,0);
 \draw [dashed,smooth] (-0.4+-1.25,0) to[out=90,in=-150] (-0.4+-0.7,1) to[out=40,in=-70] (-0.4+1.25,2);
             \node [below] at (-1.25,0)  {  $i$};  \node [below] at (1.25,0)  {  $i$}; 
             \node [below] at (0,0)  {$\tiny\mathbf{D}$};  
 \end{tikzpicture}\end{minipage}
 + (-1)^{b_1+b_5}
 \begin{minipage}{2.5cm}
  \begin{tikzpicture}[scale=0.8] 
     \clip(-1.9,-0.5)rectangle (1.9,2.2);
   \draw (-1.85,0) rectangle (1.5,2);
   \node[cylinder,draw=black,thick,aspect=0.4,
        minimum height=1.65cm,minimum width=0.3cm,
        shape border rotate=90,
        cylinder uses custom fill,
        cylinder body fill=blue!30,
        cylinder end  fill=blue!10]
  at (0,0.93) (A) { };
 \draw [smooth] (-1.25,2) to[out=-90,in=90] (-1.25,0);
      \draw [smooth] (1.25,0) to[out=90,in=-90]  (1.25,2); 
 \draw [dashed,smooth] (-0.4+-1.25,2) to[out=-90,in=90] (-0.4+-1.25,0);
      \draw [dashed,smooth] (-0.4+1.25,0) to[out=90,in=-90]  (-0.4+1.25,2); 
             \node [below] at (-1.25,0)  {  $i$};  \node [below] at (1.25,0)  {  $i$}; 
             \node [below] at (0,0)  {$\tiny\mathbf{D}$};  
 \end{tikzpicture}\end{minipage}
\]
\caption{Pulling an $i$-crossing through an $i$-diagonal $\mathbf{D}$.
Recall, $b_k:=b_k(\mathbf{D})$ is the total number of $\mathbf{B}_k$ bricks in  $\mathbf{D}$.
}
\label{fig:crossingthroughdiagonal}
\end{figure}
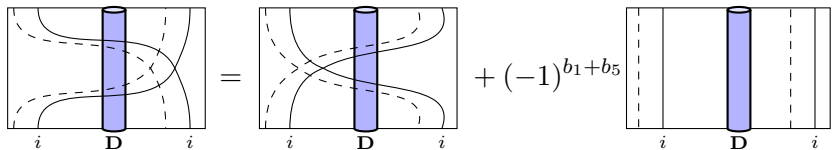

We shall prove the proposition via a series of small lemmas representing easy cases.  
Recalling \cref{remarkonmovingfromrighttoleft}, we proceed from right-to-left through the possible bricks that form 
an $i$-diagonal, and check what happens as the 
$i$-crossing passes each successive brick.

If the $i$-diagonal is invisible, we first must pass the $i$-crossing through either a $\mathbf{B}_3$ or $\mathbf{B}_2$ brick. We shall show that the $i$-crossing passes through this brick without cost.

\begin{lem}\label{316}
We can pull an $i$-crossing through a $\mathbf{B}_2$ or $\mathbf{B}_3$ brick without cost.  
\end{lem}

\begin{proof}
    In the former (respectively latter) case, 
  we first apply relation \ref{rel9} (respectively \ref{rel10}) 
  to the 
     ghost $i$-crossing and the black $(i+1)$-strand
(respectively 
       $i$-crossing and the  $(i-1)$-ghost)
 to push the $i$-crossing through to the left  at the expense of an error term.
 In both cases, the error term is zero by relation \ref{rel4}. 
 The 
  $\mathbf{B}_2$ case  is illustrated in  Figure \ref{fig:crossingB2B3}.
  We may now pull the   $i$-crossing through the black $i$-strand without cost (by relation \ref{rel8}).  
 We therefore obtain the required diagram.  \qedhere
  
\begin{figure}[ht!]

\[\scalefont{0.6}
         \begin{minipage}{4.1cm}  \begin{tikzpicture}[scale=0.8] 
       \clip(-2.6,-0.5)rectangle (2.6,2.5);
   \draw (-2.5,0) rectangle (2.5,2);   
   \draw  (-1.35,0)  to [out=30,in=-110] (2,1) ;      
   \draw  (-1.35,2)  to [out=-30,in=110] (2,1) ;      
\draw  (2,1)  to [out=80,in=-80] (2.1,2) ;      
\draw  (2,1)  to [out=-80,in=80] (2.1,0) ; 
  \draw[dashed]  (-1+-1.35,0)  to [out=90,in=-110] (-1+2,1) ;      
   \draw[dashed]  (-1+-1.35,2)  to [out=-90,in=110] (-1+2,1) ;      
\draw[dashed]  (-1+2,1)  to [out=80,in=-80] (-1+2.1,2) ;      
\draw[dashed]  (-1+2,1)  to [out=-80,in=80] (-1+2.1,0) ;      
      \node [below] at (-1.35,0)  {  $i$}; 
              \node [below] at (1.3,0)  {$\tiny i-1$};  
  \draw[dashed]  (0.4+0,0)  to [out=90,in=-90] (0.4,2) ;      
    \draw[dashed]  (0.2+-0.9,0)  to [out=90,in=-90] (0.2+-0.9,2) ;      
     \draw   (1.3,0)  to [out=90,in=-90] (1.3,2) ;      
        \draw   (0.2,0)  to [out=90,in=-90] (0.2,2) ;   \node [below] at (2.3,0)  {$\tiny i$};      
              \node [below] at (0.2,0)  {  $i$};     
              \end{tikzpicture}\end{minipage}
=
       \begin{minipage}{4.1cm}  \begin{tikzpicture}[scale=0.8] 
       \clip(-2.6,-0.5)rectangle (2.6,2.5);
   \draw (-2.5,0) rectangle (2.5,2);   
   \draw  (-1.35,0)  to [out=40,in=-143] (2.1,2) ;      
   \draw  (-1.35,2)  to [out=-40,in=142] (2.1,0) ;      
   \draw[dashed]  (-2.35,0)  to [out=40,in=-143] (-0.8+2.1,2) ;      
   \draw[dashed]  (-2.35,2)  to [out=-40,in=142] (2.1-0.8,0) ;      
  
     
      \node [below] at (-1.35,0)  {  $i$}; 
              \node [below] at (1.3,0)  {$\tiny i-1$};  
  \draw[dashed]  (0.4+0,0)  to [out=90,in=-90] (0.4,2) ;      
    \draw[dashed]  (0.2+-0.9,0)  to [out=90,in=-90] (0.2+-0.9,2) ;      
     \draw   (1.3,0)  to [out=90,in=-90] (1.3,2) ;      
        \draw   (0.2,0)  to [out=90,in=-90] (0.2,2) ;      
              \node [below] at (0.2,0)  {  $i$};    \node [below] at (2.3,0)  {$\tiny i$};    
              \end{tikzpicture}\end{minipage}       
 -
\begin{minipage}{4.1cm}  \begin{tikzpicture}[scale=0.8] 
       \clip(-2.6,-0.5)rectangle (2.6,2.5);
   \draw (-2.5,0) rectangle (2.5,2);   
   \draw  (-1.35,0)  to [out=10,in=-90] (0.25,1) ;      
   \draw  (-1.35,2)  to [out=-10,in=90] (0.25,1) ;      
    \node [below] at (2.3,0)  {$\tiny i$};  
\draw  (0.5,1)  to [out=90,in=-170] (2.1,2) ;      
\draw  (0.5,1)  to [out=-90,in=170] (2.1,0) ; 
    \draw[dashed]  (-0.8+0.25,1)  to [out=-90,in=10] (-1+-1.35,0)  ;      
   \draw[dashed]  (-0.8+0.25,1)  to [out=90,in=-10] (-1+-1.35,2)  ;      
\draw[dashed]  (-0.8+0.5,1)  to [out=90,in=-170] (-1+2.1,2) ;      
\draw[dashed]  (-0.8+0.5,1)  to [out=-90,in=170] (-1+2.1,0) ; 
     
      \node [below] at (-1.35,0)  {  $i$}; 
              \node [below] at (1.3,0)  {$\tiny i-1$};  
  \draw[dashed]  (0.4+0,0)  to [out=90,in=-90] (0.4,2) ;      
    \draw[dashed]  (0.2+-0.9,0)  to [out=100,in=-100] (0.2+-0.9,2) ;      
     \draw   (1.3,0)  to [out=90,in=-90] (1.3,2) ;      
        \draw   (0.2,0)  to [out=100,in=-100] (0.2,2) ;      
              \node [below] at (0.2,0)  {  $i$};     
              \end{tikzpicture}\end{minipage}                                        \]
\caption{Pulling an $i$-crossing  through a $\mathbf{B}_2$ brick.
On the right-hand side of the equality the first diagram can now be pulled to the left through the $i$-strand at no cost.  The second diagram is zero by relation \ref{rel4}.}
\label{fig:crossingB2B3}
\end{figure}
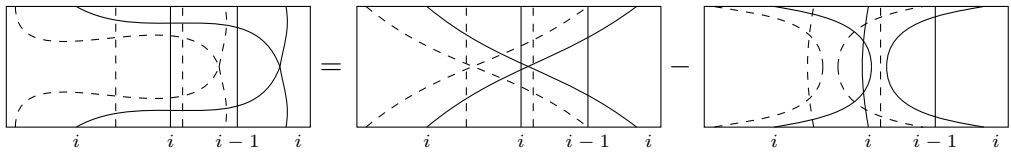
\end{proof}

We have seen that we  can pull an $i$-crossing pair through a $\mathbf{B}_2$ or
  $\mathbf{B}_3$ brick without cost.  Therefore, we now consider what happens when we pull an $i$-crossing through some number (possibly zero) of $\mathbf{B}_1$ bricks.  
 We first deal with the case that $b_1=0$.

 \begin{lem}\label{cross0b1}
Let $\mathbf{D}$ be an  $i$-diagonal with $b_1=0$.
  We can pull an $i$-crossing through  $\mathbf{D}$ at the expense of 
 an  error term, as illustrated in \cref{fig:crossingthroughdiagonal}.  
 \end{lem}
 
 \begin{proof}
By \cref{316}  we need only consider pulling an $i$-crossing 
 through a $\mathbf{B}_4$, $\mathbf{B}_5$, or
 $\mathbf{B}_6$ brick.
 This can be done at the expense of an error term as in relations \ref{rel9}, \ref{rel10} and \ref{rel12}, giving the required form.  
\end{proof}

 \begin{lem}\label{cross1b1}
 Let $\mathbf{D}$ be an  $i$-diagonal with $b_1=1$.
  We can pull an $i$-crossing through  $\mathbf{D}$ at the expense of 
 an  error term, as illustrated in \cref{fig:crossingthroughdiagonal}.
 \end{lem}  
\begin{proof}
  
As in \cref{cross0b1}, we need only consider pulling an $i$-crossing 
 through a $\mathbf B_1$ brick followed by a $\mathbf{B}_4$, $\mathbf{B}_5$, or
 $\mathbf{B}_6$ brick. We shall prove this via a series of steps.  
  
\noindent\textbf{Step 1.}  We first  pull the $i$-crossing  through 
 the   $(i-1)$-ghost at the expense of an error term (in which we undo the crossing) using relation \ref{rel10}.  The error term is the leftmost diagram depicted in   \cref{errorstep1}.

\noindent\textbf{Step 2.} We can now apply relation \ref{rel9} to pass the ghost $i$-crossing through the $(i+1)$-strand and obtain a further error term. This error term
  is easily seen to be zero by applying relation \ref{rel4}. 
   This gives us the  first term after the equality in \cref{fig:2crossingb1}.

\noindent\textbf{Step 3.} We now turn our attention to the error term from Step 1.  
 We pull the non-vertical   $i$-ghost through the vertical black $(i+1)$-strand immediately to its left.
   The result is a diagram with a double crossing of black $i$-strands with a dot on the rightmost 
 of the two, this is depicted in \cref{errorstep1}.  
 We also obtain an error term with a dot on the $(i+1)$-strand; however this error term is zero by relation \ref{rel4} and so is not depicted in \cref{errorstep1}.

 \noindent\textbf{Step 4.} Continuing from Step 3, we can apply relations \ref{rel3} and \ref{rel4} to rewrite the dotted double $i$-crossing   as a single crossing without decoration at the expense of multiplication by the scalar $-1$.  
 This diagram can then be deformed isotopically to obtain the rightmost diagram in Figure \ref{fig:2crossingb1}.  
 
 \begin{figure}[ht!]
 \[
        \scalefont{0.6}  \begin{minipage}{4.15cm}  \begin{tikzpicture}[scale=0.8] 
       \clip(-2.7,-0.5)rectangle (2.6,2.5);
   \draw (-2.65,0) rectangle (2.5,2);   
   \draw  (0.3,2)   to [out=-90,in=90]  (0.3,0);
  \draw  (-1.75,2)  to [out=-110,in=90] (0.375,1);
    \draw  (-1.75,0)  to [out=110,in=-90] (0.375,1);

       \draw[dashed]  (-1+0.3,2)   to [out=-90,in=90]  (-1+0.3,0);
  \draw[dashed]  (-0.8+-1.75,2)  to [out=-90,in=90] (-0.9+0.38,1);
    \draw[dashed]  (-0.8+-1.75,0)  to [out=90,in=-90] (-0.9+0.38,1);

  \draw  (0.55,1)  to [out=90,in=-90] (2.3,2) ;      
  \draw  (0.55,1)  to [out=-90,in=90] (2.3,0) ;      
  \draw[dashed]  (-0.9+0.5,1)  to [out=90,in=-90] (-0.9++2.3,2) ;      
  \draw[dashed]  (-0.9+0.5,1)  to [out=-90,in=90] (-0.9++2.3,0) ;      
                      \node [below] at (-1.7,0)  {  $i$}; 
              \node [below] at (1.3,0)  {$\tiny i-1$};  
  \draw[dashed]  (0.4+0,0)  to [out=85,in=-85] (0.4+0,2) ;      
    \draw[dashed]  (0.2+-1.7,0)  to [out=90,in=-90] (0.2+-1.7,2) ;      
    \draw   (1.3,0)  to [out=85,in=-85] (1.3,2) ;      
            \draw   (0.2+-0.8,0)  to [out=90,in=-90] (0.2+-0.8,2) ;  
             \node [below] at (0.2+0.1,0)  {  $i$};     
                  \node [below] at (0.2+-0.9,0)  {  $i+1$}; 
             \end{tikzpicture}\end{minipage}
=
    \begin{minipage}{4.15cm}  \begin{tikzpicture}[scale=0.8] 
       \clip(-2.7,-0.5)rectangle (2.6,2.5);
   \draw (-2.65,0) rectangle (2.5,2);   
   \draw  (0.3,2)   to [out=-95,in=95]  (0.3,0);
  \draw  (-1.75,2)  to [out=-110,in=90] (0.375,1);
    \draw  (-1.75,0)  to [out=110,in=-90] (0.375,1);

       \draw[dashed]  (-1+0.3,2)   to [out=-90,in=90]  (-1+0.3,0);
  \draw[dashed]  (-0.8+-1.75,2)  to [out=-90,in=90] (-0.9+0.3,1);
    \draw[dashed]  (-0.8+-1.75,0)  to [out=90,in=-90] (-0.9+0.3,1);

  \draw  (2.3,0)  to [out=85,in=-85] (2.3,2) ;      
       \node [below] at (2.3,0)  {$\tiny i$};  
  \draw[dashed]  (-0.9++2.3,0)  to [out=85,in=-85] (-0.9++2.3,2) ;      
    \fill (0.355,1) circle (2.3pt);   
                      \node [below] at (-1.7,0)  {  $i$}; 
              \node [below] at (1.3,0)  {$\tiny i-1$};  
  \draw[dashed]  (0.4+0,0)  to [out=83,in=-83] (0.4+0,2) ;      
    \draw[dashed]  (0.2+-1.7,0)  to [out=80,in=-80] (0.2+-1.7,2) ;      
    \draw   (1.3,0)  to [out=85,in=-85] (1.3,2) ;      
            \draw   (0.2+-0.8,0)  to [out=80,in=-80] (0.2+-0.8,2) ;  
             \node [below] at (0.2+0.1,0)  {  $i$};     
                  \node [below] at (0.2+-0.9,0)  {  $i+1$}; 
             \end{tikzpicture}\end{minipage}
\]
\caption{The diagram on the left-hand side of the equality is the error term from Step 1.
The diagram on the right-hand side is obtained by  
applying relation \ref{rel7}; the resulting error term is then   zero by relation \ref{rel4}.
}
\label{errorstep1}
\end{figure}
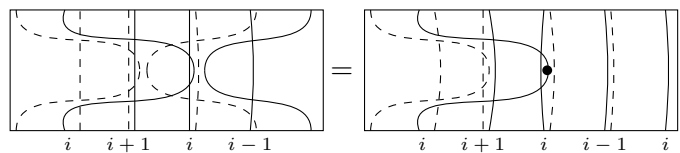
 
Applying steps 1--4 pulls the $i$-crossing through the $\mathbf B_1$ brick as depicted in \cref{fig:2crossingb1}. Finally, we may pull the $i$-crossing through the $\mathbf B_4$, $\mathbf B_5$ or $\mathbf B_6$ brick. Doing so for the first term after the equality in \cref{fig:2crossingb1} yields the first term in \cref{fig:crossingthroughdiagonal} and an error term which is zero by relation \ref{rel4}. Doing so for the second term after the equality in \cref{fig:2crossingb1} yields a term which is zero by \cref{move3} and an error term which is the second term in \cref{fig:crossingthroughdiagonal}. \qedhere

  \begin{figure}[ht!]

\[\scalefont{0.6}
         \begin{minipage}{4.15cm}  \begin{tikzpicture}[scale=0.8] 
       \clip(-2.7,-0.5)rectangle (2.6,2.5);
   \draw (-2.65,0) rectangle (2.5,2);   
   \draw[dashed]  (-2.45,0)  to [out=40,in=-90] (1.5,2) ;      
  \draw  (-1.75,0)  to [out=30,in=-90] (2.3,2) ;      
   \draw[dashed]  (-2.45,2)  to [out=-40,in=90] (1.5,0) ;      
  \draw  (-1.75,2)  to [out=-30,in=90] (2.3,0) ;      
 \node [below] at (2.3,0)  {$\tiny i$};  
                      \node [below] at (-1.7,0)  {  $i$}; 
              \node [below] at (1.3,0)  {$\tiny i-1$};  
  \draw[dashed]  (0.4+0,0)  to [out=90,in=-90] (0.4+0,2) ;      
    \draw[dashed]  (0.2+-0.9,0)  to [out=90,in=-90] (0.2+-0.9,2) ;      
    \draw[dashed]  (0.2+-1.7,0)  to [out=90,in=-90] (0.2+-1.7,2) ;      
    \draw   (1.3,0)  to [out=90,in=-90] (1.3,2) ;      
        \draw   (0.2+0.1,0)  to [out=90,in=-90] (0.2+0.1,2) ;      
            \draw   (0.2+-0.8,0)  to [out=90,in=-90] (0.2+-0.8,2) ;  
             \node [below] at (0.2+0.1,0)  {  $i$};      \node [below] at (2.3,0)  {$\tiny i$};  
                  \node [below] at (0.2+-0.9,0)  {  $i+1$}; 
             \end{tikzpicture}\end{minipage}
=
 \begin{minipage}{4.15cm}  \begin{tikzpicture}[scale=0.8] 
       \clip(-2.7,-0.5)rectangle (2.6,2.5);
   \draw (-2.65,0) rectangle (2.5,2);   
   \draw[dashed]  (-2.45,0)  to [out=110,in=-90-80] (1.5,2) ;      
   \draw[dashed]  (-2.45,2)  to [out=-110,in=90+80] (1.5,0) ;      
  \draw  (-1.75,2)  to [out=-110,in=90+80] (2.3,0) ;      
  \draw  (-1.75,0)  to [out=110,in=-90-80] (2.3,2) ;      

                      \node [below] at (-1.7,0)  {  $i$}; 
              \node [below] at (1.3,0)  {$\tiny i-1$};  
  \draw[dashed]  (0.4+0,0)  to [out=90,in=-90] (0.4+0,2) ;      
    \draw[dashed]  (0.2+-0.9,0)  to [out=90,in=-90] (0.2+-0.9,2) ;      
    \draw[dashed]  (0.2+-1.7,0)  to [out=90,in=-90] (0.2+-1.7,2) ;      
    \draw   (1.3,0)  to [out=90,in=-90] (1.3,2) ;      
        \draw   (0.2+0.1,0)  to [out=90,in=-90] (0.2+0.1,2) ;      
            \draw   (0.2+-0.8,0)  to [out=90,in=-90] (0.2+-0.8,2) ;  
             \node [below] at (0.2+0.1,0)  {  $i$};     
                  \node [below] at (0.2+-0.9,0)  {  $i+1$}; 
             \end{tikzpicture}\end{minipage}
             -
                       \begin{minipage}{4.15cm}  \begin{tikzpicture}[scale=0.8] 
       \clip(-2.7,-0.5)rectangle (2.6,2.5);
   \draw (-2.65,0) rectangle (2.5,2);   
  \draw  (0.3,2)   to [out=-170,in=110]  (-1.75,0);
  \draw  (-1.75,2)  to [out=-110,in=90+80] (0.3,0);
  \draw  (2.3,1)  to [out=90,in=-90] (2.3,2) ;      
  \draw  (2.3,1)  to [out=-90,in=90] (2.3,0) ;     
   \node [below] at (2.3,0)  {$\tiny i$};   
  \draw[dashed]  (-0.9+ 0.3,2)   to [out=-170,in=110]  (-0.7++-1.75,0);
  \draw[dashed]  (-0.7+ -1.75,2)  to [out=-110,in=90+80] (-0.9++0.3,0);
  \draw[dashed]  (-0.9++2.3,1)  to [out=90,in=-90] (-0.9++2.3,2) ;      
  \draw[dashed]  (-0.9++2.3,1)  to [out=-90,in=90] (-0.9++2.3,0) ;      
                      \node [below] at (-1.7,0)  {  $i$}; 
              \node [below] at (1.3,0)  {$\tiny i-1$};  
  \draw[dashed]  (0.4+0,0)  to [out=90,in=-90] (0.4+0,2) ;      
    \draw[dashed]  (0.2+-1.7,0)  to [out=90,in=-90] (0.2+-1.7,2) ;      
    \draw   (1.3,0)  to [out=90,in=-90] (1.3,2) ;      
            \draw   (0.2+-0.8,0)  to [out=90,in=-90] (0.2+-0.8,2) ;  
             \node [below] at (0.2+0.1,0)  {  $i$};     
                  \node [below] at (0.2+-0.9,0)  {  $i+1$}; 
             \end{tikzpicture}\end{minipage}\]
\caption{Pulling an $i$-crossing through a  single $\mathbf{B}_1$ brick.  We have pulled the $i$-crossing as far to the left as possible in order to illustrate that it
 has passed through the brick; however,
  in practice it will have to pass through other bricks.  
}
\label{fig:2crossingb1}
\end{figure}
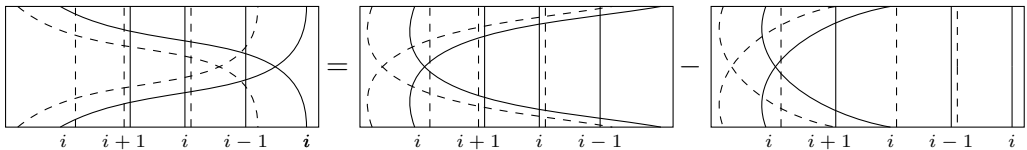
 \end{proof}

We now turn our attention to proving \cref{lem:crossingthroughdiagonal} in full generality. We refer to the rightmost diagram in \cref{fig:2crossingb1} as having \emph{an $i$-crossing attached to the $i$-node} in the $\mathbf{B}_1$ brick.  

\begin{proof}[Proof of \cref{lem:crossingthroughdiagonal}]
By \cref{316,cross0b1,cross1b1} we may assume that we start by passing the $i$-crossing through  $b_1$  $\mathbf B_1$ bricks for $b_1\geq 2$.

Repeating the first step in the argument in \cref{cross1b1} yields a leading term and an error term. By repeatedly applying relations \ref{rel9} and \ref{rel10}, we can push the $i$-crossing in the leading term through all $b_1$ $\mathbf B_1$ bricks; each error term along the way is zero by relation \ref{rel4}. We can then proceed to push the $i$-crossing through the $\mathbf B_4$, $\mathbf B_5$ or $\mathbf B_6$ brick, yielding the first term in \cref{fig:crossingthroughdiagonal} and an error term which is again zero by relation \ref{rel4}.


We now deal with the error term from our first step. As in Steps 3 and 4 of the proof of 
\cref{cross1b1}, we can 
  rewrite this as $-1$ multiplied by the diagram with a crossing attached to the $i$-node (say $(r,c,k)$) in the top $\mathbf B_1$ brick.

We now diverge from the proof of  \cref{cross1b1}, as we need to 
consider what happens when we 
pull the $i$-crossing attached to $(r,c,k)$ to the left.  
 Firstly, we must  pull this $i$-crossing through the next $\mathbf B_1$ brick.  
 We pull the $i$-crossing through the $(i-1)$-ghost labelled by the node $(r,c-1,k)$ 
 yielding two terms: the leading term $d$ and an error term $d'$.
 The term $d$  is zero, as we can now push this $i$-crossing through all remaining $\mathbf B_1$ bricks and the $\mathbf B_4$, $\mathbf B_5$, or $\mathbf B_6$ brick and apply \cref{move3}, with all error terms along the way being zero by relation \ref{rel4}.
 Now, observe that  the diagram $d'$ has a double crossing of $i$-strands. 
We can apply relation \ref{rel6} to $d'$, followed by relations \ref{rel3} and \ref{rel4} to rewrite the double crossing as an $i$-crossing attached to the node $(r-1,c-1,k)$, at the expense of scalar multiplication by $-1$ again.

We repeat the above procedure until we end up with a diagram with an $i$-crossing attached to node $(r',c',k)$ with $r'=1$ or $c'=1$ (that is, attached to the $i$-node of the bottom $\mathbf B_1$ brick) and coefficient $(-1)^{b_1}$. Finally, we pull this $i$-crossing through the $\mathbf B_4$, $\mathbf B_5$ or $\mathbf B_6$ brick to yield a leading term which is zero by \cref{move3} and an error term which is the second term in \cref{fig:crossingthroughdiagonal}. \qedhere
\end{proof}

 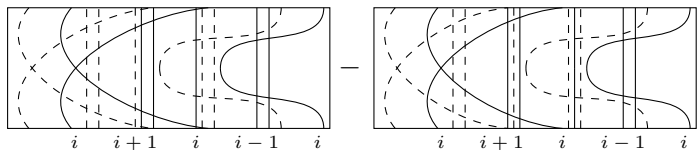
\begin{figure}[ht!]            

\[  \scalefont{0.6}
 \begin{minipage}{4.35cm}  \begin{tikzpicture}[scale=0.8] 
       \clip(-2.9,-0.35)rectangle (2.6,2.35);
   \draw (-2.8,0) rectangle (2.5,2);   
   \draw  (0.3,2)  to [out=-90,in=90] (0.3,0) ;

 \draw  (-1.75,2)  to [out=-130,in=90+85] (0.5,0) ;      
  \draw  (-1.75,0)  to [out=130,in=-90-85] (0.5,2) ;

 \draw[dashed]  (-0.7-1.75,2)  to [out=-130,in=90+85] (-0.9+0.5,0) ;      
  \draw[dashed]  (-0.7-1.75,0)  to [out=130,in=-90-85] (-0.9+0.5,2) ;

  \draw  (0.7,1)  to [out=-90,in=90] (2.4,0) ;      
  \draw  (0.7,1)  to [out=90,in=-90] (2.4,2) ;      
   \draw[dashed]  (-1+0.3,2)  to [out=-90,in=90] (-1+0.3,0) ;      

  \draw[dashed]  (-1+0.7,1)  to [out=-90,in=90] (-0.7+2.4,0) ;      
  \draw[dashed]  (-1+0.7,1)  to [out=90,in=-90] (-0.7+2.4,2) ;      

              \node [below] at (1.3,0)  {$\tiny i-1$};  
   \draw[dashed]  (0.4+0,0)  to [out=90,in=-90] (0.4+0,2) ;      
     \draw[dashed]  (0.2+-1.7,0)  to [out=90,in=-90] (0.2+-1.7,2) ;      
    \draw   (1.3,0)  to [out=90,in=-90] (1.3,2) ;      
             \draw   (0.2+-0.8,0)  to [out=90,in=-90] (0.2+-0.8,2) ;  
             \node [below] at (0.2+0.1,0)  {  $i$};     
                  \node [below] at (0.2+-0.9,0)  {  $i+1$}; 
    \draw[dashed]  (0.1+0.1+0.4+0,0)  to [out=90,in=-90] (0.1+0.1+0.4+0,2) ;      
    \draw[dashed]  (0.1+0.1+0.2+-1.7,0)  to [out=90,in=-90] (0.1+0.1+0.2+-1.7,2) ;      
    \draw   (0.1+0.1+1.3,0)  to [out=90,in=-90] (0.1+0.1+1.3,2) ;      
            \draw   (0.1+0.1+0.2+-0.8,0)  to [out=90,in=-90] (0.1+0.1+0.2+-0.8,2) ;  
                       \node [below] at (-1.7,0)  {  $i$}; 
 \node [below] at (2.3,0)  {$\tiny i$};  
             \end{tikzpicture}\end{minipage}
-
 \begin{minipage}{4.35cm}  \begin{tikzpicture}[scale=0.8] 
       \clip(-2.9,-0.35)rectangle (2.6,2.35);
   \draw (-2.8,0) rectangle (2.5,2);   
    \node [below] at (2.3,0)  {$\tiny i$};  
     \draw  (0.5,2)  to [out=-90,in=90] (0.5,1) ;      
     \draw  (0.5,0)  to [out=90,in=-90] (0.5,1) ;      

   \draw[dashed]  (-1+0.5,2)  to [out=-90,in=90] (-1+0.5,1) ;      
     \draw[dashed]  (-1+0.5,0)  to [out=90,in=-90] (-1+0.5,1) ;      

 \draw  (-1.75,2)  to [out=-130,in=90+85] (0.3,0) ;      
  \draw  (-1.75,0)  to [out=130,in=-90-85] (0.3,2) ;

 \draw[dashed]  (-0.7+-1.75,2)  to [out=-130,in=90+85] (-0.9+0.3,0) ;      
  \draw[dashed]  (-0.7+-1.75,0)  to [out=130,in=-90-85] (-0.9+0.3,2) ;

                        \node [below] at (-1.7,0)  {  $i$}; 

  \draw  (0.7,1)  to [out=-90,in=90] (2.4,0) ;      
  \draw  (0.7,1)  to [out=90,in=-90] (2.4,2) ;      
 
  \draw[dashed]  (-1+0.7,1)  to [out=-90,in=90] (-0.7+2.4,0) ;      
  \draw[dashed]  (-1+0.7,1)  to [out=90,in=-90] (-0.7+2.4,2) ;      

              \node [below] at (1.3,0)  {$\tiny i-1$};  
   \draw[dashed]  (0.4+0,0)  to [out=90,in=-90] (0.4+0,2) ;      
     \draw[dashed]  (0.2+-1.7,0)  to [out=90,in=-90] (0.2+-1.7,2) ;      
    \draw   (1.3,0)  to [out=90,in=-90] (1.3,2) ;      
             \draw   (0.2+-0.8,0)  to [out=90,in=-90] (0.2+-0.8,2) ;  
             \node [below] at (0.2+0.1,0)  {  $i$};     
                  \node [below] at (0.2+-0.9,0)  {  $i+1$}; 
    \draw[dashed]  (0.1+0.1+0.4+0,0)  to [out=90,in=-90] (0.1+0.1+0.4+0,2) ;      
    \draw[dashed]  (0.1+0.1+0.2+-1.7,0)  to [out=90,in=-90] (0.1+0.1+0.2+-1.7,2) ;      
    \draw   (0.1+0.1+1.3,0)  to [out=90,in=-90] (0.1+0.1+1.3,2) ;      
            \draw   (0.1+0.1+0.2+-0.8,0)  to [out=90,in=-90] (0.1+0.1+0.2+-0.8,2) ;  
 
             \end{tikzpicture}\end{minipage}\]
\caption{Pulling the $i$-crossing in the rightmost diagram in \cref{fig:2crossingb1}
   through a second $\mathbf{B}_1$ brick.  
   The leftmost diagram becomes zero once we push the $i$-crossing through all bricks, by \cref{move3}.}
\label{2brickB1}
\end{figure}

\begin{prop}\label{lem:dotthroughdiagonal}
  We can pull a dot through an  $i$-diagonal, $\mathbf{D}$, without cost, as illustrated in \cref{fig:dotthroughdiagonal}.  
  \end{prop}
\begin{figure}[ht!]
\[\scalefont{0.6}
 \begin{minipage}{32mm}\begin{tikzpicture}[scale=0.8] 
     \clip(-2.4,-0.5)rectangle (1.9,2.2);
   \draw (-2,0) rectangle (1.5,2);
   \node[cylinder,draw=black,thick,aspect=0.4,
        minimum height=1.65cm,minimum width=0.3cm,
        shape border rotate=90,
        cylinder uses custom fill,
        cylinder body fill=blue!30,
        cylinder end  fill=blue!10]
  at (0,0.93) (A) { };
  \draw[dashed]  (-1.25-0.5,0)  to [out=90,in=-90] (1.25-0.5,2) ;      
  \draw  (-1.25,0)  to [out=90,in=-90] (1.25,2) ;      
 \fill (1,1.45) circle (3pt);
              \node [below] at (-1.25,0)  {  $i$}; 
             \node [below] at (0,0)  {$\tiny\mathbf{D}$};  
\end{tikzpicture}\end{minipage}
=\begin{minipage}{40mm}
\begin{tikzpicture}[scale=0.8] 
     \clip(-2.4,-0.5)rectangle (1.9,2.2);
   \draw (-2,0) rectangle (1.5,2);
   \node[cylinder,draw=black,thick,aspect=0.4,
        minimum height=1.65cm,minimum width=0.3cm,
        shape border rotate=90,
        cylinder uses custom fill,
        cylinder body fill=blue!30,
        cylinder end  fill=blue!10]
  at (0,0.93) (A) { };
  \draw[dashed]  (-1.25-0.5,0)  to [out=90,in=-90] (1.25-0.5,2) ;      
  \draw  (-1.25,0)  to [out=90,in=-90] (1.25,2) ;      
 \fill (-1.15,0.35) circle (3pt);
              \node [below] at (-1.25,0)  {  $i$}; 
             \node [below] at (0,0)  {$ \tiny\mathbf{D}$};  
\end{tikzpicture}\end{minipage}
\]
\caption{Pulling a dot through an $i$-diagonal.}
\label{fig:dotthroughdiagonal}
\end{figure}
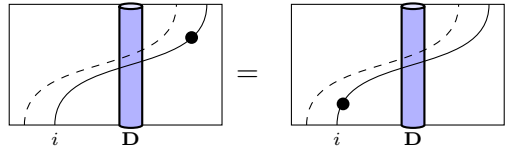
\begin{proof}
The result is clear for bricks of the form $\mathbf{B}_4, \mathbf{B}_5$ or $\mathbf{B}_6$ by applying relations \ref{rel2} and \ref{rel14}. Now assume that $\mathbf{D}$ has more than one brick.
Each  brick  of the form $\mathbf{B}_1, \mathbf{B}_2, \mathbf{B}_3$ contains a single $i$-strand.
This strand is intersected by the $i$-strand in \cref{fig:dotthroughdiagonal}.
We can pull the dot through the brick (using relation \ref{rel3}) at the expense of an error term in which 
we undo the aforementioned    $i$-crossing.  
 The resulting diagram factors through an idempotent which is zero by \cref{move3}.  
 See \cref{b1brickdot} for the  case of a single $\mathbf{B}_1$ brick.  
 \end{proof}

\begin{figure}[ht!]
  
\[\scalefont{0.6}
         \begin{minipage}{4.1cm}  \begin{tikzpicture}[scale=0.8] 
       \clip(-2.6,-0.5)rectangle (2.6,2.5);
   \draw (-2.5,0) rectangle (2.5,2);   
   \draw[dashed]  (-2.35,0)  to [out=90,in=-90] (1,2) ;      
  \draw  (-1.35,0)  to [out=90,in=-90] (2,2) ;      
  \fill (1.75,1.45) circle (3pt);                       \node [below] at (-1.7,0)  {  $i$}; 
              \node [below] at (0.2+0.9,0)  {$\tiny i-1$};  
  \draw[dashed]  (0.2+0,0)  to [out=90,in=-90] (0.2+0,2) ;      
    \draw[dashed]  (0.2+-0.9,0)  to [out=90,in=-90] (0.2+-0.9,2) ;      
    \draw[dashed]  (0.2+-1.7,0)  to [out=90,in=-90] (0.2+-1.7,2) ;      
    \draw   (0.2+0.9,0)  to [out=90,in=-90] (0.2+0.9,2) ;      
        \draw   (0.2+0.1,0)  to [out=90,in=-90] (0.2+0.1,2) ;      
            \draw   (0.2+-0.8,0)  to [out=90,in=-90] (0.2+-0.8,2) ;  
             \node [below] at (0.2+0.1,0)  {  $i$};     
                  \node [below] at (0.2+-0.9,0)  {  $i+1$}; 
             \end{tikzpicture}\end{minipage}
=
       \begin{minipage}{4.1cm}  \begin{tikzpicture}[scale=0.8] 
       \clip(-2.6,-0.5)rectangle (2.6,2.5);
   \draw (-2.5,0) rectangle (2.5,2);   
  \draw[dashed]  (-2.35,0)  to [out=90,in=-90] (1.2,2) ;      
  \draw  (-1.35,0)  to [out=90,in=-90] (2.1,2) ;      
 \fill (-1.2,.45) circle (3pt);                       \node [below] at (-1.7,0)  {  $i$}; 
              \node [below] at (0.2+0.9,0)  {$\tiny i-1$};  
  \draw[dashed]  (0.2+0,0)  to [out=90,in=-90] (0.2+0,2) ;      
    \draw[dashed]  (0.2+-0.9,0)  to [out=90,in=-90] (0.2+-0.9,2) ;      
    \draw[dashed]  (0.2+-1.7,0)  to [out=90,in=-90] (0.2+-1.7,2) ;      
    \draw   (0.2+0.9,0)  to [out=90,in=-90] (0.2+0.9,2) ;      
        \draw   (0.2+0.1,0)  to [out=90,in=-90] (0.2+0.1,2) ;      
            \draw   (0.2+-0.8,0)  to [out=90,in=-90] (0.2+-0.8,2) ;  
             \node [below] at (0.2+0.1,0)  {  $i$};     
                  \node [below] at (0.2+-0.9,0)  {  $i+1$}; 
             \end{tikzpicture}\end{minipage}
-
       \begin{minipage}{4.1cm}  \begin{tikzpicture}[scale=0.8] 
       \clip(-2.6,-0.5)rectangle (2.6,2.5);
   \draw (-2.5,0) rectangle (2.5,2);   
     \draw  (-1.35,0)  
      to [out=90,in=-90]   (0.3,2);
   \draw  (0.2+0.1,0)  
    to [out=90,in=-90]   (2,2);
     \draw[dashed]  (-1+-1.35,0)  
      to [out=90,in=-90]   (-1+0.3,2);
                     \node [below] at (-1.7,0)  {  $i$}; 
              \node [below] at (0.2+0.9,0)  {$\tiny i-1$};  
  \draw[dashed]  (0.2+0,0) 
   to [out=90,in=-90]  (0.2+0,2) ;      
  \draw[dashed]  (0.2+0.1-1,0)  
   to [out=90,in=-90]   (2.1-0.9,2);
    \draw[dashed]  (0.2+-1.7,0)  to [out=90,in=-90] (0.2+-1.7,2) ;      
    \draw   (0.2+0.9,0)  to [out=90,in=-90] (0.2+0.9,2) ;      
             \draw   (0.2+-0.8,0)  to [out=90,in=-90] (0.2+-0.8,2) ;  
             \node [below] at (0.2+0.1,0)  {  $i$};     
                  \node [below] at (0.2+-0.9,0)  {  $i+1$}; 
             \end{tikzpicture}\end{minipage}
\]
  \caption{Pulling a dot through a $\mathbf{B}_1$ brick.  The rightmost diagram
  is zero modulo in $A_{\Gamma}$ by \cref{move3}.
    }
\label{b1brickdot}
\end{figure}
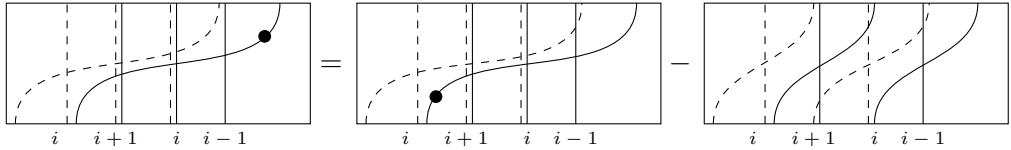

\begin{prop}\label{2crossB1247}
Suppose we have a
 double crossing of an  $i$-strand, $A$, with   an invisible  $i$-diagonal, $\mathbf{D}$.  
We   can pull $A$ through $\mathbf{D}$ 
 at the expense of multiplication by the scalar $(-1)^{b_1+b_3+b_5}$.  This is depicted in \cref{dotbtick}.
\end{prop}

 \begin{figure}[ht]
 
 \[\scalefont{0.6}
  \begin{minipage}{2.8cm}  \begin{tikzpicture}[scale=0.8] 
     \clip(-1.8,-0.5)rectangle (1.8,2.5);
   \draw (-1.75,0) rectangle (1.5,2);
   \node[cylinder,draw=black,thick,aspect=0.4,
        minimum height=1.65cm,minimum width=0.3cm,
        shape border rotate=90,
        cylinder uses custom fill,
        cylinder body fill=blue!30,
        cylinder end  fill=blue!10]
  at (0,0.93) (A) { };
  \draw  (-1.25,0)  to [out=50,in=-90] (0.95,1) ;      
 \draw  (-1.25,2)  to [out=-50,in=90] (0.95,1) ;      
 
 \draw[dashed]  (-0.4+-1.25,0)  to [out=90,in=-90] (-0.4+0.95,1) ;      
 \draw[dashed]  (-0.4+-1.25,2)  to [out=-90,in=90] (-0.4+0.95,1) ;      
 
             \node [below] at (-1.25,0)  {  $i$}; 
             \node [below] at (0,0)  {$\tiny\mathbf{D}$};  
             \end{tikzpicture}\end{minipage}=(-1)^{b_1+b_3+b_5}
               \begin{minipage}{2.8cm}  \begin{tikzpicture}[scale=0.8] 
     \clip(-1.8,-0.5)rectangle (1.8,2.5);
   \draw (-1.75,0) rectangle (1.5,2);
   \node[cylinder,draw=black,thick,aspect=0.4,
        minimum height=1.65cm,minimum width=0.3cm,
        shape border rotate=90,
        cylinder uses custom fill,
        cylinder body fill=blue!30,
        cylinder end  fill=blue!10]
  at (0,0.93) (A) { };
  \draw  (-1.25,0)  to [out=90,in=-90] (-0.5,1) ;      
 \draw  (-1.25,2)  to [out=-90,in=90] (-0.5,1) ;   
 
 \draw[dashed]  (-0.4+-1.25,0)  to [out=90,in=-90] (-0.4+-0.5,1) ;      
 \draw[dashed]  (-0.4+-1.25,2)  to [out=-90,in=90] (-0.4+-0.5,1) ;      
    
             \node [below] at (-1.25,0)  {  $i$}; 
             \node [below] at (0,0)  {$\tiny\mathbf{D}$};  
             \end{tikzpicture}\end{minipage}
 \]
             \caption{Resolving a diagram  as in   \cref{2crossB1247}  for  
              $\mathbf{D}$ an invisible   $i$-diagonal. }
             \label{dotbtick}
\end{figure}
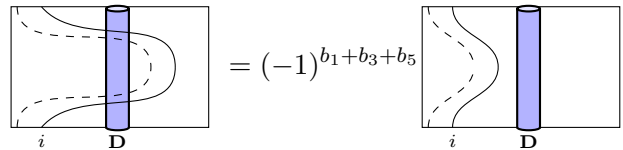

  \begin{proof}
 We first consider the double crossing with  a 
$\mathbf{B}_2$ brick (the $\mathbf{B}_3$ brick case is similar and left an exercise for the reader).
  First apply  relation \ref{rel6} to pull $A$ through the   $(i-1)$-ghost 
  to obtain two terms; one with a dot on the $i$-strand and one with a dot on the
   $(i-1)$-strand.  The diagram with a dot on the $(i-1)$-strand is zero by relation \ref{rel4}.  We apply relations \ref{rel3} and \ref{rel4} to the other diagram and obtain a diagram with a single crossing as depicted in \cref{B6}.

 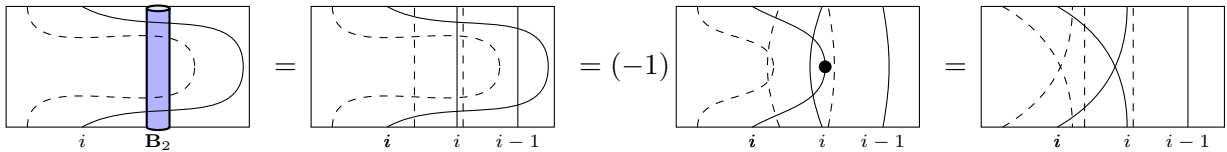
\begin{figure}[ht!]
    \[ \scalefont{0.6}
  \begin{minipage}{3.5cm}  \begin{tikzpicture}[scale=0.8] 
     \clip(-2.6,-0.5)rectangle (1.6,2.5);
   \draw (-2.5,0) rectangle (1.5,2);
   \node[cylinder,draw=black,thick,aspect=0.4,
        minimum height=1.65cm,minimum width=0.3cm,
        shape border rotate=90,
        cylinder uses custom fill,
        cylinder body fill=blue!30,
        cylinder end  fill=blue!10]
  at (0,0.93) (A) { };
   \draw[dashed]  (-0.9+-1.25,0)  to [out=90,in=-90] (-0.9+1.5,1) ;      
 \draw[dashed]  (-0.9+-1.25,2)  to [out=-90,in=90] (-0.9+1.5,1) ;      

  \draw  (-1.25,0)  to [out=30,in=-90] (1.4,1) ;      
 \draw  (-1.25,2)  to [out=-30,in=90] (1.4,1) ;      
             \node [below] at (-1.25,0)  {  $i$}; 
             \node [below] at (0,0)  {$\tiny\mathbf{B}_2$};  
             \end{tikzpicture}\end{minipage}
              =
        \begin{minipage}{3.5cm}  \begin{tikzpicture}[scale=0.8] 
     \clip(-2.6,-0.5)rectangle (1.6,2.5);
   \draw (-2.5,0) rectangle (1.5,2);
      \draw[dashed]  (-0.9+-1.25,0)  to [out=90,in=-90] (-0.9+1.5,1) ;      
 \draw[dashed]  (-0.9+-1.25,2)  to [out=-90,in=90] (-0.9+1.5,1) ;      

  \draw  (-1.25,0)  to [out=30,in=-90] (1.4,1) ;      
 \draw  (-1.25,2)  to [out=-30,in=90] (1.4,1) ;      
             \node [below] at (-1.25,0)  {  $i$}; 
              \node [below] at (-1.25,0)  {  $i$}; 
             \node [below] at (0.9,0)  {$\tiny i-1$};  
                          \node [below] at (-0.1,0)  {$\tiny i$};  
    \draw[dashed]  (0,0)  to [out=90,in=-90] (0,2) ;        \draw   (-0.1,0)  to [out=90,in=-90] (-0.1,2) ;            
    \draw   (0.9,0)  to [out=90,in=-90] (0.9,2) ;   
     \draw[dashed]   (-0.8,0)  to [out=90,in=-90] (-0.8,2) ;                \end{tikzpicture}\end{minipage}
           =
(      -1 )      \begin{minipage}{3.5cm}  \begin{tikzpicture}[scale=0.8] 
     \clip(-2.6,-0.5)rectangle (1.6,2.5);
   \draw (-2.5,0) rectangle (1.5,2);
      \draw[dashed]  (-0.9+-1.25,0)  to [out=90,in=-90] (-0.9,1) ;      
 \draw[dashed]  (-0.9+-1.25,2)  to [out=-90,in=90] (-0.9,1) ;      

  \draw  (-1.25,0)  to [out=30,in=-90] (-0.05,1) ;      
 \draw  (-1.25,2)  to [out=-30,in=90] (-0.05,1) ;      
             \node [below] at (-1.25,0)  {  $i$}; 
              \node [below] at (-1.25,0)  {  $i$}; 
             \node [below] at (0.9,0)  {$\tiny i-1$};  
                          \node [below] at (-0.1,0)  {$\tiny i$};  
    \draw[dashed]  (0,0)  to [out=80,in=-80] (0,2) ;      
      \draw   (-0.1,0)  to [out=110,in=-110] (-0.1,2) ;            
    \draw   (0.9,0)  to [out=80,in=-80] (0.9,2) ;     \fill (-0.05,1)  circle (3pt);
     \draw[dashed]   (-0.8,0)  to [out=110,in=-110] (-0.8,2) ;                \end{tikzpicture}
     \end{minipage}
                         =
                        \begin{minipage}{3.5cm}  \begin{tikzpicture}[scale=0.8] 
     \clip(-2.6,-0.5)rectangle (1.6,2.5);
   \draw (-2.5,0) rectangle (1.5,2);
     \draw  (-1.25,0)  to [out=30,in=-90] (-0.1,2) ;      
 \draw  (-1.25,2)  to [out=-30,in=90] (-0.1,0) ;   
     \draw[dashed]  (-0.9+-1.25,0)  to [out=30,in=-90] (-0.9+-0.1,2) ;      
 \draw[dashed]  (-0.9+-1.25,2)  to [out=-30,in=90] (-0.9+-0.1,0) ;    
             \node [below] at (-1.25,0)  {  $i$}; 
              \node [below] at (-1.25,0)  {  $i$}; 
             \node [below] at (0.9,0)  {$\tiny i-1$};  
                          \node [below] at (-0.1,0)  {$\tiny i$};  
    \draw[dashed]  (0,0)  to [out=90,in=-90] (0,2) ;     
     \draw   (0.9,0)  to [out=90,in=-90] (0.9,2) ;   
     \draw[dashed]   (-0.8,0)  to [out=90,in=-90] (-0.8,2) ;                \end{tikzpicture}\end{minipage} 
\]
\caption{The first equality follows by definition. The second equality follows from relation \ref{rel6} where the error term is zero by relation \ref{rel4}.  The third equality follows from relations \ref{rel3} and \ref{rel4}.  
 }
  \label{B6}
\end{figure}

Observe that the $i$-crossing is attached to the $i$-node in the $\mathbf{B}_2$ (respectively $\mathbf{B}_3$) brick  in $\gamma$.  Repeating arguments from the proof of \cref{lem:crossingthroughdiagonal} yields the result.  
%
     \end{proof}

 \begin{prop}\label{2crossB1247b}
Suppose we have a double crossing of an  $i$-strand, $A$, with   a visible  $i$-diagonal, $\mathbf{D}$.  
We can pull $A$ through $\mathbf{D}$ at the expense of scalar multiplication by $(-1)^{b_1+b_4}$ and  acquiring a dot.  This is depicted in
\cref{dotbtick11111111}.  \end{prop}

 \begin{figure}[ht]
 
 \[\scalefont{0.6}
               \begin{minipage}{2.6cm}  \begin{tikzpicture}[scale=0.8] 
     \clip(-1.8,-0.5)rectangle (1.6,2.5);
   \draw (-1.75,0) rectangle (1.4,2);
   \node[cylinder,draw=black,thick,aspect=0.4,
        minimum height=1.65cm,minimum width=0.3cm,
        shape border rotate=90,
        cylinder uses custom fill,
        cylinder body fill=blue!30,
        cylinder end  fill=blue!10]
  at (0,0.93) (A) { };
  \draw  (-1.25,0)  to [out=50,in=-90] (0.95,1) ;      
 \draw  (-1.25,2)  to [out=-50,in=90] (0.95,1) ;      
 
 \draw[dashed]  (-0.4+-1.25,0)  to [out=90,in=-90] (-0.4+0.95,1) ;      
 \draw[dashed]  (-0.4+-1.25,2)  to [out=-90,in=90] (-0.4+0.95,1) ;      
 
             \node [below] at (-1.25,0)  {  $i$}; 
             \node [below] at (0,0)  {$\tiny\mathbf{D}$};  
             \end{tikzpicture}\end{minipage}=  (-1)^{b_1+b_4}
              \begin{minipage}{2.6cm}  \begin{tikzpicture}[scale=0.8] 
     \clip(-1.8,-0.5)rectangle (1.6,2.5);
   \draw (-1.75,0) rectangle (1.4,2);
   \node[cylinder,draw=black,thick,aspect=0.4,
        minimum height=1.65cm,minimum width=0.3cm,
        shape border rotate=90,
        cylinder uses custom fill,
        cylinder body fill=blue!30,
        cylinder end  fill=blue!10]
  at (0,0.93) (A) { };
  \draw  (-1.25,0)  to [out=90,in=-90] (-0.5,1) ;      
 \draw  (-1.25,2)  to [out=-90,in=90] (-0.5,1) ;   
     \draw[fill]  (-0.5,1)   circle (2pt);       
 \draw[dashed]  (-0.4+-1.25,0)  to [out=90,in=-90] (-0.4+-0.5,1) ;      
 \draw[dashed]  (-0.4+-1.25,2)  to [out=-90,in=90] (-0.4+-0.5,1) ;      
    
             \node [below] at (-1.25,0)  {  $i$}; 
             \node [below] at (0,0)  {$\tiny\mathbf{D}$};  
             \end{tikzpicture}\end{minipage}
  \]
             \caption{Resolving a diagram  as in   \cref{2crossB1247b}  for  
 	    $\mathbf{D}$ a visible  $i$-diagonal. }
             \label{dotbtick11111111}
\end{figure}
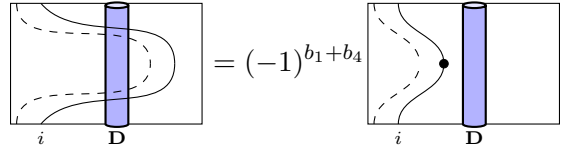

\begin{proof}
We claim that we can pull such a strand through a $\mathbf{B}_1$  at the cost of scalar multiplication by $(-1)$.  
 To see this, note that we can pull the $i$-strand through the
  $(i-1)$-ghost at the expense of acquiring a dot; we can pull the  
$i$-ghost  through the black $(i+1)$-strand at the expense of acquiring another dot (both error terms are zero by \cref{move2.5}).  We thus obtain the diagram on the left-hand side of the equality in \cref{onaprayer}.  Applying relations \ref{rel3} and \ref{rel4} several times, we obtain the diagram in which $A$ passed through $\mathbf{B}_1$ at the expense of scalar multiplication by $(-1)$ (the leftmost diagram after the equality in \cref{onaprayer})  along with  two error terms, which are both zero by \cref{move2.5};  see  \cref{onaprayer}.

\begin{figure}[ht!]
\[\scalefont{0.6}
         \begin{minipage}{3.5cm}  \begin{tikzpicture}[scale=0.8] 
       \clip(-2.8,-0.5)rectangle (1.6,2.5);
   \draw (-2.8,0) rectangle (1.5,2);   
   \draw  (-1.75,0)  to [out=60,in=-90] (0.3,1) ;    
    \draw  (-1.75,2)  to [out=-60,in=90] (0.3,1) ;      
  \draw[dashed]   (-1+-1.75,2) to [out=-90,in=90] (-1+0.3,1) ;      
    \draw[dashed]   (-1+-1.75,0)  to [out=90,in=-90] (-1+0.3,1) ;    


   \fill (0.3,1.2) circle (2.5pt);   \fill (0.3,0.8) circle (2.5pt);
                      \node [below] at (-1.7,0)  {  $i$}; 
              \node [below] at (1.3,0)  {$\tiny i-1$};  
  \draw[dashed]  (0.4+0,0)  to [out=80,in=-80] (0.4+0,2) ;      
    \draw[dashed]  (0.2+-0.9,0)  to [out=110,in=-110] (0.2+-0.9,2) ;      
    \draw[dashed]  (0.2+-1.7,0)  to [out=90,in=-90] (0.2+-1.7,2) ;      
    \draw   (1.3,0)  to [out=80,in=-80] (1.3,2) ;      
        \draw   (0.2+0.1,0)  to [out=110,in=-110] (0.2+0.1,2) ;      
            \draw   (0.2+-0.8,0)  to [out=90,in=-90] (0.2+-0.8,2) ;  
             \node [below] at (0.2+0.1,0)  {  $i$};     
                  \node [below] at (0.2+-0.9,0)  {  $i+1$}; 
             \end{tikzpicture}\end{minipage}
             =
           (-1)        \begin{minipage}{3.5cm}  \begin{tikzpicture}[scale=0.8] 
       \clip(-2.8,-0.5)rectangle (1.6,2.5);
   \draw (-2.8,0) rectangle (1.5,2);   
   \draw  (-1.75,0)  to [out=60,in=-90] (-0.3,1) ;    
    \draw  (-1.75,2)  to [out=-60,in=90] (-0.3,1) ;      
  \draw[dashed]   (-1+-1.75,2) to [out=-90,in=90] (-1-0.3,1) ;      
    \draw[dashed]   (-1+-1.75,0)  to [out=90,in=-90] (-1-0.3,1) ;    


                      \node [below] at (-1.7,0)  {  $i$}; 
              \node [below] at (1.3,0)  {$\tiny i-1$};  
  \draw[dashed]  (0.4+0,0)  to [out=90,in=-90] (0.4+0,2) ;      
    \draw[dashed]  (0.2+-0.9,0)  to [out=90,in=-90] (0.2+-0.9,2) ;      
    \draw[dashed]  (0.2+-1.7,0)  to [out=90,in=-90] (0.2+-1.7,2) ;      
    \draw   (1.3,0)  to [out=90,in=-90] (1.3,2) ;      
        \draw   (0.2+0.1,0)  to [out=90,in=-90] (0.2+0.1,2) ;      
            \draw   (0.2+-0.8,0)  to [out=90,in=-90] (0.2+-0.8,2) ;  
             \node [below] at (0.2+0.1,0)  {  $i$};     
                  \node [below] at (0.2+-0.9,0)  {  $i+1$}; 
             \end{tikzpicture}\end{minipage}
 -
                \begin{minipage}{3.5cm}  \begin{tikzpicture}[scale=0.8] 
       \clip(-2.8,-0.5)rectangle (1.6,2.5);
   \draw (-2.8,0) rectangle (1.5,2);   
 \draw  (-1.75,0)  to [out=30,in=-90] (0.3,2) ;      
 \draw  (-1.75,2)  to [out=-30,in=90] (0.3,0) ;      
\draw[dashed]  (-0.9+-1.75,0)  to [out=30,in=-90] (-1+0.3,2) ;      
\draw[dashed] (-0.9+-1.75,2)  to [out=-30,in=90] (-1+0.3,0) ;      

    \fill (0.155,0.58) circle (3pt);   
                      \node [below] at (-1.7,0)  {  $i$}; 
              \node [below] at (1.3,0)  {$\tiny i-1$};  

  \draw[dashed]  (0.4+0,0)  to [out=90,in=-90] (0.4+0,2) ;      
    \draw[dashed]  (0.2+-1.7,0)  to [out=90,in=-90] (0.2+-1.7,2) ;      
    \draw   (1.3,0)  to [out=90,in=-90] (1.3,2) ;      
            \draw   (0.2+-0.8,0)  to [out=90,in=-90] (0.2+-0.8,2) ;  
             \node [below] at (0.2+0.1,0)  {  $i$};     
                  \node [below] at (0.2+-0.9,0)  {  $i+1$}; 
             \end{tikzpicture}\end{minipage}
             -
                \begin{minipage}{3.5cm}  \begin{tikzpicture}[scale=0.8] 
       \clip(-2.8,-0.5)rectangle (1.6,2.5);
   \draw (-2.8,0) rectangle (1.5,2);   
 \draw  (-1.75,0)  to [out=30,in=-90] (0.3,2) ;      
 \draw  (-1.75,2)  to [out=-30,in=90] (0.3,0) ;      
\draw[dashed]  (-0.9+-1.75,0)  to [out=30,in=-90] (-1+0.3,2) ;      
\draw[dashed] (-0.9+-1.75,2)  to [out=-30,in=90] (-1+0.3,0) ;      

 \fill (0.155,1.42) circle (3pt);                         \node [below] at (-1.7,0)  {  $i$}; 
              \node [below] at (1.3,0)  {$\tiny i-1$};  

  \draw[dashed]  (0.4+0,0)  to [out=90,in=-90] (0.4+0,2) ;      
    \draw[dashed]  (0.2+-1.7,0)  to [out=90,in=-90] (0.2+-1.7,2) ;      
    \draw   (1.3,0)  to [out=90,in=-90] (1.3,2) ;      
            \draw   (0.2+-0.8,0)  to [out=90,in=-90] (0.2+-0.8,2) ;  
             \node [below] at (0.2+0.1,0)  {  $i$};     
                  \node [below] at (0.2+-0.9,0)  {  $i+1$}; 
             \end{tikzpicture}\end{minipage}
             \]
\caption{The effect of pulling a double crossing $i$-strand through a $\mathbf{B}_1$ brick.}
\label{onaprayer}
\end{figure}
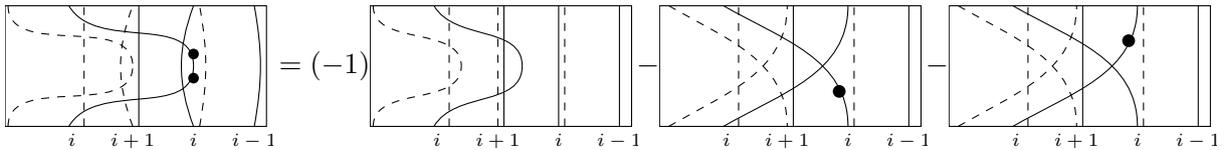

 Thus we can pull $A$ through all the  $\mathbf{B}_1$ bricks at the expense of multiplication by $(-1)^{b_1}$.  
We now wish to consider what happens when we pull $A$ through a 
$\mathbf{B}_4$, $\mathbf{B}_5$, or $\mathbf{B}_6$ brick.

We  can pull $A$ through a $\mathbf{B}_6$ brick  at the expense of acquiring a single dot on $A$  (by relation \ref{rel11}), as required.  
  We can pull $A$ through a 
 $\mathbf{B}_5$  or $\mathbf{B}_4$ 
 brick 
 at the expense of scalar multiplication by $(-1)^{b_4}$
 and an error term in which there is a dot on  the $(i+1)$-strand (respectively $(i-1)$-strand)   in the $\mathbf{B}_5$  (respectively  $\mathbf{B}_4$)  brick.
 This error term is   zero by \cref{move2.5}.  We thus obtain the required result.  
 \end{proof}

\begin{prop}\label{2crossB1247backwards}
Suppose we have an $i$-strand $A$ (respectively a dotted $i$-strand $A'$) next to an invisible (respectively visible) $i$-diagonal $\mathbf{D}$ (respectively $\mathbf{D}'$). We can pull $A$ (respectively $A')$ through $\mathbf{D}$ (respectively $\mathbf{D}'$) at the expense of scalar multiplication by $(-1)^{b_1+b_4}$.
This is depicted in \cref{backwards}.
\end{prop}
 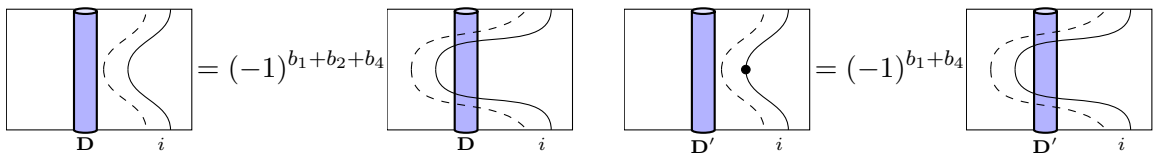
\begin{figure}[ht!]
 \[  
       \scalefont{0.6}       
                 \begin{minipage}{2.4cm}  \begin{tikzpicture}[scale=0.8] 
     \clip(-1.3,-0.5)rectangle (1.8,2.5);
   \draw (-1.3,0) rectangle (1.75,2);
   \node[cylinder,draw=black,thick,aspect=0.4,
        minimum height=1.65cm,minimum width=0.3cm,
        shape border rotate=90,
        cylinder uses custom fill,
        cylinder body fill=blue!30,
        cylinder end  fill=blue!10]
  at (0,0.93) (A) { };
  \draw  (+0.7,1)  to [out=90,in=-90] (1.4,2) ;      
 \draw  (+0.7,1)  to [out=-90,in=90] (1.4,0) ;      
 
 \draw[dashed]  (-0.4+0.7,1)  to [out=90,in=-90] (-0.4+1.4,2) ;      
 \draw[dashed]  (-0.4+0.7,1)  to [out=-90,in=90] (-0.4+1.4,0) ;      
              \node [below] at (1.25,0)  {  $i$}; 
             \node [below] at (0,0)  {$\tiny\mathbf{D}$};  
             \end{tikzpicture}\end{minipage}
             =
              (-1)^{b_1+b_2+b_4}
  \begin{minipage}{2.0cm}  \begin{tikzpicture}[scale=0.8] 
     \clip(-1.3,-0.5)rectangle (1.8,2.5);
   \draw (-1.3,0) rectangle (1.75,2);
   \node[cylinder,draw=black,thick,aspect=0.4,
        minimum height=1.65cm,minimum width=0.3cm,
        shape border rotate=90,
        cylinder uses custom fill,
        cylinder body fill=blue!30,
        cylinder end  fill=blue!10]
  at (0,0.93) (A) { };
  \draw  (-0.5,1)  to [out=90,in=-90] (1.4,2) ;      
 \draw  (-0.5,1)  to [out=-90,in=90] (1.4,0) ;      
 
 \draw[dashed]  (-0.4+-0.5,1)  to [out=90,in=-130] (-0.4+1.4,2) ;      
 \draw[dashed]  (-0.4+-0.5,1)  to [out=-90,in=130] (-0.4+1.4,0) ;      
 
             \node [below] at (1.25,0)  {  $i$}; 
             \node [below] at (0,0)  {$\tiny\mathbf{D}$};  
             \end{tikzpicture}\end{minipage}
             \quad
     \quad                               \quad             \quad
                 \begin{minipage}{2.4cm}  \begin{tikzpicture}[scale=0.8] 
     \clip(-1.3,-0.5)rectangle (1.8,2.5);
   \draw (-1.3,0) rectangle (1.75,2);
   \node[cylinder,draw=black,thick,aspect=0.4,
        minimum height=1.65cm,minimum width=0.3cm,
        shape border rotate=90,
        cylinder uses custom fill,
        cylinder body fill=blue!30,
        cylinder end  fill=blue!10]
  at (0,0.93) (A) { };
  \draw  (+0.7,1)  to [out=90,in=-90] (1.4,2) ;      
 \draw  (+0.7,1)  to [out=-90,in=90] (1.4,0) ;      
 
 \draw[dashed]  (-0.4+0.7,1)  to [out=90,in=-90] (-0.4+1.4,2) ;      
 \draw[dashed]  (-0.4+0.7,1)  to [out=-90,in=90] (-0.4+1.4,0) ;      
   \draw[fill]  (+0.7,1)   circle (2pt);       
             \node [below] at (1.25,0)  {  $i$}; 
             \node [below] at (0,0)  {$\tiny\mathbf{D}'$};  
             \end{tikzpicture}\end{minipage}
             =
              (-1)^{b_1+b_4}
  \begin{minipage}{2.2cm}  \begin{tikzpicture}[scale=0.8] 
     \clip(-1.3,-0.5)rectangle (1.8,2.5);
   \draw (-1.3,0) rectangle (1.75,2);
   \node[cylinder,draw=black,thick,aspect=0.4,
        minimum height=1.65cm,minimum width=0.3cm,
        shape border rotate=90,
        cylinder uses custom fill,
        cylinder body fill=blue!30,
        cylinder end  fill=blue!10]
  at (0,0.93) (A) { };
  \draw  (-0.5,1)  to [out=90,in=-90] (1.4,2) ;      
 \draw  (-0.5,1)  to [out=-90,in=90] (1.4,0) ;      
 
 \draw[dashed]  (-0.4+-0.5,1)  to [out=90,in=-130] (-0.4+1.4,2) ;      
 \draw[dashed]  (-0.4+-0.5,1)  to [out=-90,in=130] (-0.4+1.4,0) ;      
 
             \node [below] at (1.25,0)  {$i$}; 
             \node [below] at (0,0)  {$\tiny\mathbf{D}'$};  
             \end{tikzpicture}\end{minipage}
\]
\caption{Pulling a dotted $i$-strand through an invisible 
$i$-diagonal $\mathbf{D}$ (respectively visible $i$-diagonal $\mathbf{D}'$).}
\label{backwards}
\end{figure}

\begin{proof}
We shall first show that $A$ can pass through a $\mathbf{B}_2$ brick at the expense of multiplication by $-1$ (respectively $+1$) and acquiring a dot.  The $\mathbf B_3$ case is similar.  

We can pull  $A$ through  the $(i-1)$-ghost in $\mathbf{B}_2$ 
 and then apply Move $1$ 
  to yield two diagrams each with an $i$-crossing. 
   In the case that the $i$-crossing is bypassed to the left by the $(i-1)$-ghost, 
   we can push the crossing to the right and apply \cref{move3} to see that this diagram is zero.  
 It remains to consider the diagram with an $i$-crossing attached to the $i$-node in $\mathbf{B}_2$ 
    bypassed to the right by the $(i-1)$-ghost 
      multiplied by the scalar $-1$. 
       By Move 2 we obtain the required result.

We have seen that pulling $A$ through a $\mathbf{B}_2$  or $\mathbf{B}_3$ adds a dot to the strand. Therefore it will suffice to show that the result holds for
  $A'$  and $\mathbf{D}'$ as in the rightmost diagram of \cref{backwards}. 
 First, we show that we can pull $A'$ through a $\mathbf{B}_1$ brick at the expense of multiplication by the scalar $(-1)$.

By relation \ref{rel7} we can pull  $A'$ through  the $(i-1)$-ghost  at the expense of multiplication by $-1$ and losing the dot (the error term is zero by \cref{move2.5}).  
We now apply Move 1, followed by \cref{move3}, followed by Move 2  as above, to obtain the result.


Finally, we can pull the dotted $A'$ through a $\mathbf{B}_4$,
$\mathbf{B}_5$, or $\mathbf{B}_6$ brick using relation \ref{rel7}, \ref{rel6} or \ref{rel11} at the expense of multiplication by $(-1)^{b_4}$ (note that the error term in \ref{rel6} and \ref{rel7} is zero by \cref{move2.5}). The result follows.
\end{proof}

%

\subsection{The algebra isomorphism   for two subquotients each  with a   single residue}
We now associate a sequence $\chi(\gamma)$ to the multipartition $\gamma$. 
This sequence records the occurrences of visible and invisible  $i$-diagonals, and the
bricks from which the diagonals are built. We define  
\[\chi(\mathbf{D}_x)=
(-1)^{b_1}\mathbf{d}^{j}_{k}
\] 
where $k=4,5,$ or $6$ if the bottom brick of $\mathbf{D}_x$ is $\mathbf{B}_4,\mathbf{B}_5$ or $\mathbf{B}_6$, respectively and where
$j=2$ or $3$ if $\mathbf{D}_x$ is invisible with  top brick $\mathbf{B}_2$ or $\mathbf{B}_3$, respectively, and $j=0$ if $\mathbf{D}_x$ is visible.  
%
We then define $\chi(\gamma)$ to be the sequence $(\chi(\mathbf D_{x_1}),\chi(\mathbf D_{x_2}), \dots)$.

%
%
%
%
%

\begin{eg}
Continuing with Example \ref{notchiexample}, we have that $
\chi(\gamma)=(\mathbf{d}_4^0,\mathbf{d}_4^3,
\mathbf{d}_6^0,\mathbf{d}_5^0,\mathbf{d}_5^0)$.
\end{eg}

Let $-\emptyset$ and $\emptyset$ denote two formal symbols in what follows.

\begin{defn}\label{equivalence}We say that two sequences $\chi$ and $\overline{\chi}$ are \emph{equivalent}, and write $\chi \sim \overline\chi$,  if one can be obtained from the other by applying the following local identifications within the sequence (i.e.~to individual elements or adjacent pairs of elements in the sequence):
 \begin{itemize}
\item[$(i)$]  $(+\mathbf{d}_4^j ) =( -\mathbf{d}_5^j)$ and $(-\mathbf{d}_4^j ) =( +\mathbf{d}_5^j)$ for $j=0,2,3$;
 \item[$(ii)$] $(-\emptyset) = (+\mathbf{d}_k^2,+\mathbf{d}_k^3) = (-\mathbf{d}_k^2,-\mathbf{d}_k^3)= (+\mathbf{d}_k^3,+\mathbf{d}_k^2) = (-\mathbf{d}_k^3,-\mathbf{d}_k^2)$ for $k=4,5$;
\item[$(iii)$]  $(+\mathbf{d}_k^j,- \mathbf{d}_k^j)= (-\mathbf{d}_k^j,+\mathbf{d}_k^j)$ for $j=2,3$ and  $k=4,5$;
\item[$(iv)$]  $(\emptyset) = (+\mathbf{d}_6^j,+\mathbf{d}_6^j) = (-\mathbf{d}_6^j,-\mathbf{d}_6^j)$ for $j=2,3$;
\item[$(v)$]  $(-\emptyset,-\emptyset)=(\emptyset)$ and we may delete any  $\emptyset$ from our sequence $\chi$ without cost.
\end{itemize}
\end{defn}

\begin{eg}\label{schurisomorhishi}
Let $e=\overline e=5$, $i=\overline{i}=0$, $\kappa=\overline\kappa=(0)$ and $\theta=\overline\theta=(0)$.  
The partition
\[
\gamma = (30^6,28,20,19^2,15,11,9,7,3^6) \in \mptn 1{326}
\]
has  four addable $0$-nodes and 
\[
\chi(\gamma)=
 (+\mathbf{d}_4^0, \underbrace{+\mathbf{d}_4^2,+ \mathbf{d}_4^3},\underbrace{+\mathbf{d}_4^3,+\mathbf{d}_4^2}, +\mathbf{d}_4^0,-\mathbf{d}_6^0,-\mathbf{d}_5^3, -\mathbf{d}_5^3,+\mathbf{d}_5^2,+\mathbf{d}_5^0),
\]
where we have indicated the pairs of adjacent elements to which we can apply the local identification $(ii)$.  
The partition 
\[
\overline{\gamma} = (10^4,9,5^4,3^3,1^8) \in \mptn 1{86}
\]
 also has four addable $0$-nodes and 
\[
\chi(\overline\gamma)=(+\mathbf{d}_4^0,   +\mathbf{d}_4^0,-\mathbf{d}_6^0,-\mathbf{d}_5^3, -\mathbf{d}_5^3,+\mathbf{d}_5^2,+\mathbf{d}_5^0).
\]
These two sequences are easily seen to be equivalent using $(ii)$ and $(v)$ above.  
\end{eg}

\begin{thm}\label{main}\label{perfgh}
If $\gamma$ and $\overline\gamma$ are two multipartitions which are $i$- and $\overline i$-admissible, respectively, with $\chi(\gamma) \sim \chi(\overline\gamma)$, then $A_{\Gamma} \cong A_{\overline{\Gamma}}$ as graded $\Bbbk$-algebras; the isomorphism is given by $\Phi$ from \cref{vsisom}.
\end{thm}

\begin{proof}
By Proposition \ref{vsisom} we need only check that  
$\Phi( C_{\SSTS, \SSTT}C_{\SSTU, \SSTV}) = \Phi( C_{\SSTS, \SSTT})\Phi(C_{\SSTU, \SSTV})$.  
We suppose that our  product diagram 
$D= C_{\SSTS, \SSTT}C_{\SSTU, \SSTV}$ contains a neighbourhood in which some $i$-strand, $A$,  passes through an $i$-diagonal, $\mathbf D$. 
 We shall show that   
 resolving the crossing  and then applying $\Phi$,
  we obtain the same result 
 as if we resolve the crossing in $\Phi(D)$.
 There are five possible cases that we need to check for each diagonal.
The five cases are: $(a)$ moving a dot through a crossing of  $A$  with $\mathbf{D}$   
  $(b\&c)$  an $i$-crossing passes through $\mathbf{D}$ 
  $(d \& e)$ a  double crossing of   $A$ with  $\mathbf{D}$.
These are the same cases as those considered in 
 \cref{possibles}.

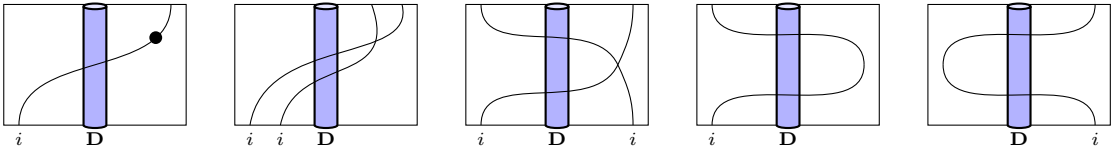
\begin{figure}[ht!]
      \[ \scalefont{0.6}
        \begin{tikzpicture}[scale=0.8] 
     \clip(-1.9,-0.5)rectangle (1.9,2.2);
   \draw (-1.5,0) rectangle (1.5,2);
   \node[cylinder,draw=black,thick,aspect=0.4,
        minimum height=1.65cm,minimum width=0.3cm,
        shape border rotate=90,
        cylinder uses custom fill,
        cylinder body fill=blue!30,
        cylinder end  fill=blue!10]
  at (0,0.93) (A) { };
  \draw  (-1.25,0)  to [out=90,in=-90] (1.25,2) ;      
 \fill (1,1.45) circle (3pt);
              \node [below] at (-1.25,0)  {  $i$}; 
             \node [below] at (0,0)  {$\tiny\mathbf{D}$};  
             \end{tikzpicture}
  \begin{tikzpicture}[scale=0.8] 
  \clip(-1.9,-0.5)rectangle (1.9,2.2);
    \draw (-1.5,0) rectangle (1.5,2);
   \node[cylinder,draw=black,thick,aspect=0.4,
        minimum height=1.65cm,minimum width=0.3cm,
        shape border rotate=90,
        cylinder uses custom fill,
        cylinder body fill=blue!30,
        cylinder end  fill=blue!10]
  at (0,0.93) (A) { };
   \draw (-0.75,0)  .. controls (-0.5,1.2)  and (1.25,0.6) ..  (0.75,2);
   \draw (-1.25,0)  .. controls (-1,1.5)  and (1.5,1) ..  (1.25,2);
             \node [below] at (-1.25,0)  {  $i$};  
             \node [below] at (-0.75,0)  {  $i$}; 
             \node [below] at (0,0)  {$\tiny\mathbf{D}$};  
\end{tikzpicture}
 \begin{tikzpicture}[scale=0.8] 
     \clip(-1.9,-0.5)rectangle (1.9,2.2);
   \draw (-1.5,0) rectangle (1.5,2);
   \node[cylinder,draw=black,thick,aspect=0.4,
        minimum height=1.65cm,minimum width=0.3cm,
        shape border rotate=90,
        cylinder uses custom fill,
        cylinder body fill=blue!30,
        cylinder end  fill=blue!10]
  at (0,0.93) (A) { };
 \draw [smooth] (-1.25,2) to[out=-90,in=115] (1,1) to[out=-180+115,in=90] (1.25,0);
 \draw [smooth] (-1.25,0) to[out=90,in=-115] (1,1) to[out=180-115,in=-90] (1.25,2);
             \node [below] at (-1.25,0)  {  $i$};  \node [below] at (1.25,0)  {  $i$}; 
             \node [below] at (0,0)  {$\tiny\mathbf{D}$};  
 \end{tikzpicture}
\begin{tikzpicture}[scale=0.8] 
     \clip(-1.9,-0.5)rectangle (1.9,2.2);
   \draw (-1.5,0) rectangle (1.5,2);
   \node[cylinder,draw=black,thick,aspect=0.4,
        minimum height=1.65cm,minimum width=0.3cm,
        shape border rotate=90,
        cylinder uses custom fill,
        cylinder body fill=blue!30,
        cylinder end  fill=blue!10]
  at (0,0.93) (A) { };
  \draw  (-1.25,0)  to [out=90,in=-90] (1.25,1) ;      
 \draw  (-1.25,2)  to [out=-90,in=90] (1.25,1) ;      
             \node [below] at (-1.25,0)  {  $i$}; 
             \node [below] at (0,0)  {$ \mathbf{D}$};  
\end{tikzpicture}
\begin{tikzpicture}[scale=0.8] 
     \clip(-1.9,-0.5)rectangle (1.8,2.2);
   \draw (-1.5,0) rectangle (1.5,2);
   \node[cylinder,draw=black,thick,aspect=0.4,
        minimum height=1.65cm,minimum width=0.3cm,
        shape border rotate=90,
        cylinder uses custom fill,
        cylinder body fill=blue!30,
        cylinder end  fill=blue!10]
  at (0,0.93) (A) { };
  \draw  (1.25,0)  to [out=90,in=-90] (-1.25,1) ;      
 \draw  (1.25,2)  to [out=-90,in=90] (-1.25,1) ;     
             \node [below] at (1.25,0)  {  $i$}; 
             \node [below] at (0,0)  {$ \mathbf{D}$};  
\end{tikzpicture}
\]
 \caption{The five possible diagrams we need to consider, for $\mathbf{D}$ an  arbitrary $i$-diagonal.}
\label{possibles}
\end{figure}

First note that cases $(a)$ and $(b)$ are trivial for all identifications $(i)$--$(v)$ in \cref{equivalence},  from \cref{lem:dotthroughdiagonal} and the defining relations of $A(n,\theta,\kappa)$, respectively.

For the identification $(i)$, we now look at cases $(c)$, $(d)$ and $(e)$. 
In case $(c)$, the result follows by \cref{lem:crossingthroughdiagonal} and the fact that $(i)$ preserves the parity of $b_1+b_5$; in case $(d)$, the result follows by \cref{2crossB1247,2crossB1247backwards} and the fact that $(i)$ preserves the parity of $b_1+b_4$ and $b_1+b_3+b_5$; in case $(e)$, the result follows by \cref{2crossB1247b} and the fact that $(i)$ preserves the parity of $b_1+b_4$ and $b_1+b_2+b_4$ -- these all follow as our equivalence is defined so that these parities are all preserved by $\Phi$.

 
We now address the identifications in $(ii)$.   
We aim to show that we can resolve a diagram with a pair of (consecutive) $i$-diagonals $\mathbf{D}_{x_1}$ and $\mathbf{D}_{x_2}$, with $\chi(\mathbf{D}_{x_1}) = \pm \mathbf{d}_k^2$ and $\chi(\mathbf{D}_{x_2}) = \pm \mathbf{d}_k^3$ (or vice versa), for $k=4$ or $5$, and an $i$-crossing as in case $(c)$ without cost 
or a double crossing as in 
cases $(d\&e)$ at the expense of multiplication by scalar $-1$.

First, consider what happens when we resolve an $i$-crossing as in case $(c)$.
We first pull the $i$-crossing through $\mathbf{D}_{x_2}$, to obtain an error term in which the crossing is undone, and then through $\mathbf{D}_{x_1}$ to obtain another  error term in which the crossing is undone, applying \cref{lem:crossingthroughdiagonal} twice.
The resulting sum of three diagrams is depicted in \cref{dottymoveystrandypandy} in the case that $(\chi(\mathbf{D}_{x_1}), \chi(\mathbf{D}_{x_2})) = (\mathbf{d}_4^2, \mathbf{d}_4^3)$.

Resolving the two error terms using \cref{2crossB1247,2crossB1247backwards}, we get the diagram with both $i$-strands vertical, with   coefficient
\begin{equation}\tag{$\dagger$}
(-1)^{b_1(x_1) + b_5(x_1) + b_1(x_2) + b_3(x_1) + b_5(x_2)} + (-1)^{b_1(x_1) + b_5(x_1) + b_1(x_2) + b_2(x_2) + b_4(x_2)}
\end{equation}
where $b_k(x_j)$ denotes the number of bricks $\mathbf B_k$ in the $i$-diagonal $\mathbf D_{x_j}$.
It is simple to check that this coefficient is zero in all cases covered in $(ii)$, and therefore we can  pass an $i$-crossing through the pair $(\mathbf{D}_{x_1},\mathbf{D}_{x_2})$ without cost.

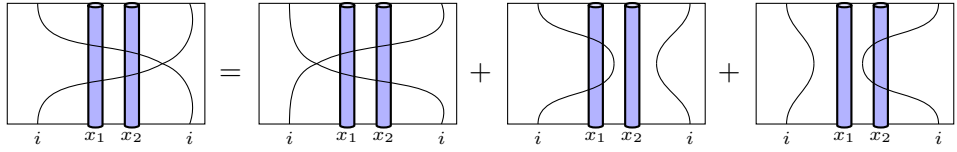
\begin{figure}[ht!]
\[
\scalefont{0.6} 
 \begin{minipage}{2.8cm}  \begin{tikzpicture}[scale=0.8] 
     \clip(-1.9,-0.5)rectangle (1.9,2.2);
   \draw (-1.75,0) rectangle (1.5,2);

   \node[cylinder,draw=black,thick,aspect=0.3,
        minimum height=1.65cm,minimum width=0.1cm,
        shape border rotate=90,
        cylinder uses custom fill,
        cylinder body fill=blue!30,
        cylinder end  fill=blue!10]
  at (-0.3,0.93) (A) { };
               \node [below] at (-0.3,-0)  {$x_1$};  

   \node[cylinder,draw=black,thick,aspect=0.3,
        minimum height=1.65cm,minimum width=0.1cm,
        shape border rotate=90,
        cylinder uses custom fill,
        cylinder body fill=blue!30,
        cylinder end  fill=blue!10]
  at (0.3,0.93) (A) { };
               \node [below] at (0.3,0)  {$x_2$};

 \draw [smooth] (-1.25,2) to[out=-90,in=150] (0.8,1) to[out=-20,in=70] (1.25,0);
 \draw [smooth] (-1.25,0) to[out=90,in=-150] (0.8,1) to[out=20,in=-70] (1.25,2);
             \node [below] at (-1.25,0)  {  $i$};  \node [below] at (1.25,0)  {  $i$}; 
  \end{tikzpicture}\end{minipage}
=
 \begin{minipage}{2.8cm}  \begin{tikzpicture}[scale=0.8] 
     \clip(-1.9,-0.5)rectangle (1.9,2.2);
   \draw (-1.75,0) rectangle (1.5,2);

   \node[cylinder,draw=black,thick,aspect=0.3,
        minimum height=1.65cm,minimum width=0.1cm,
        shape border rotate=90,
        cylinder uses custom fill,
        cylinder body fill=blue!30,
        cylinder end  fill=blue!10]
  at (-0.3,0.93) (A) { };
               \node [below] at (-0.3,-0)  {$x_1$};

   \node[cylinder,draw=black,thick,aspect=0.3,
        minimum height=1.65cm,minimum width=0.1cm,
        shape border rotate=90,
        cylinder uses custom fill,
        cylinder body fill=blue!30,
        cylinder end  fill=blue!10]
  at (0.3,0.93) (A) { };
               \node [below] at (0.3,0)  {$x_2$};

 \draw [smooth] (-1.25,2) to[out=-90,in=160] (-0.8,1) to[out=-30,in=70] (1.25,0);
 \draw [smooth] (-1.25,0) to[out=90,in=-160] (-0.8,1) to[out=30,in=-70] (1.25,2);    
             \node [below] at (-1.25,0)  {  $i$};  \node [below] at (1.25,0)  {  $i$}; 
  \end{tikzpicture}\end{minipage}
+
  \begin{minipage}{2.8cm}  \begin{tikzpicture}[scale=0.8] 
     \clip(-1.9,-0.5)rectangle (1.9,2.2);
   \draw (-1.75,0) rectangle (1.5,2);
 
   \node[cylinder,draw=black,thick,aspect=0.3,
        minimum height=1.65cm,minimum width=0.1cm,
        shape border rotate=90,
        cylinder uses custom fill,
        cylinder body fill=blue!30,
        cylinder end  fill=blue!10]
  at (-0.3,0.93) (A) { };
               \node [below] at (-0.3,-0)  {$x_1$};  
   \node[cylinder,draw=black,thick,aspect=0.3,
        minimum height=1.65cm,minimum width=0.1cm,
        shape border rotate=90,
        cylinder uses custom fill,
        cylinder body fill=blue!30,
        cylinder end  fill=blue!10]
  at (0.3,0.93) (A) { };
               \node [below] at (0.3,0)  {$x_2$};  

\draw [smooth] (-1.25,2) to[out=-90,in=90] (0,1) ; 
\draw [smooth] (-1.25,0) to[out=90,in=-90] (0,1) ; 
\draw [smooth] (1.25,2) to[out=-90,in=90] (0.7,1) ; 
\draw [smooth] (1.25,0) to[out=90,in=-90] (0.7,1) ; 
             \node [below] at (-1.25,0)  {  $i$};  \node [below] at (1.25,0)  {  $i$}; 
  \end{tikzpicture}\end{minipage}
  +
  \begin{minipage}{2.8cm}  \begin{tikzpicture}[scale=0.8] 
     \clip(-1.9,-0.5)rectangle (1.9,2.2);
   \draw (-1.75,0) rectangle (1.5,2);
 
   \node[cylinder,draw=black,thick,aspect=0.3,
        minimum height=1.65cm,minimum width=0.1cm,
        shape border rotate=90,
        cylinder uses custom fill,
        cylinder body fill=blue!30,
        cylinder end  fill=blue!10]
  at (-0.3,0.93) (A) { };
               \node [below] at (-0.3,-0)  {$x_1$};  

   \node[cylinder,draw=black,thick,aspect=0.3,
        minimum height=1.65cm,minimum width=0.1cm,
        shape border rotate=90,
        cylinder uses custom fill,
        cylinder body fill=blue!30,
        cylinder end  fill=blue!10]
  at (0.3,0.93) (A) { };
               \node [below] at (0.3,0)  {$x_2$};

\draw [smooth] (-1.25,2) to[out=-90,in=90] (-0.8,1) ; 
\draw [smooth] (-1.25,0) to[out=90,in=-90] (-0.8,1) ; 
\draw [smooth] (1.25,2) to[out=-90,in=90] (0,1) ; 
\draw [smooth] (1.25,0) to[out=90,in=-90] (0,1) ; 
             \node [below] at (-1.25,0)  {  $i$};  \node [below] at (1.25,0)  {  $i$}; 
  \end{tikzpicture}\end{minipage}
\]
\caption{The result of pulling an $i$-crossing through a pair of $i$-diagonals 
$\mathbf{D}_{x_1}$ and $\mathbf{D}_{x_2}$ with $\chi(\mathbf{D}_{x_1})  = \mathbf{d}_4^2$, $\chi(\mathbf{D}_{x_2})  = \mathbf{d}_4^3$.}
\label{dottymoveystrandypandy}
\end{figure}

Next, consider what happens when we pull a double $i$-crossing as in case $(d)$ (respectively $(e)$) through through the pair of (consecutive) $i$-diagonals $\mathbf{D}_{x_1}$, $\mathbf{D}_{x_2}$ with $\chi(\mathbf{D}_{x_1}) = \pm \mathbf{d}_k^2$ and $\chi(\mathbf{D}_{x_2}) = \pm \mathbf{d}_k^3$ (or vice versa), for $k=4$ or $5$.

We can pull the $i$-strand through both diagonals, applying \cref{2crossB1247} (respectively \cref{2crossB1247backwards}) twice, at the expense of multiplication by
\begin{equation}\tag{$\dagger\dagger$}
(-1)^{b_1(x_1+x_2) + b_3(x_1+x_2) + b_5(x_1+x_2)}
\quad \text{or} \quad
(-1)^{b_1(x_1 + x_2) + b_2(x_1+x_2) + b_4(x_1+x_2)}\end{equation}
respectively, where $b_k(x_1 + x_2) := b_k(x_1) + b_k(x_2)$. 
In all cases covered in $(ii)$, we have that $b_1(x_1) = b_1(x_2)$,  $b_3(x_1) = b_3(x_2) \pm 1$ (respectively $b_2(x_1) = b_2(x_2) \pm 1$) and $b_5(x_1) = b_5(x_2)$ (respectively $b_4(x_1) = b_4(x_2)$), and so the scalar is $-1$ for  both $(d)$ and $(e)$, as claimed.

We have shown that we can pull an $i$-crossing through any pair of consecutive 
invisible $i$-diagonals as in case  $(ii)$  without cost.  We have also shown
that we can pull a single 
$i$-strand through at the cost of multiplication by $-1$; and therefore $\Phi$ respects the identifications in $(ii)$.
Moreover, we can clearly   pull an $i$-crossing or a single $i$-strand through two such pairs without cost (as $(-1)^2=1$), and therefore $\Phi$ also respects the identifications in $(v)$. 

 In case $(iii)$ one can argue as above independently of $k$,
 and obtain coefficient $+2$  or $-2$  in $(\dagger)$ for $j=3$ or $j=2$, respectively.  
 The coefficients in  $(\dagger\dagger)$ are easily seen to  both be $-1$ for $j=2$ or $3$.
The identifications in $(iv)$ follow similarly to those in $(ii)$, but with scalars in $(\dagger\dagger)$ both being equal to $+1$.
The result follows as $\Phi$ respects these coefficients.
\end{proof}
 
\begin{cor}\label{hatty}
Let $A(n,\theta,\kappa)$ and $A(\overline n,\overline \theta,\overline \kappa)$ be two  
 diagrammatic Cherednik algebras. 
  Let $\gamma$ be an $i$-admissible multipartition  and $\overline{\gamma}$   an $\overline i$-admissible multipartition, with $\chi(\gamma) \sim \chi(\overline\gamma)$.  
We have that
\[
d_{\lambda\mu}(t)=d_{\overline\lambda\overline\mu}(t)
\]
for  all $\lambda,\mu\in\Gamma$ and $\overline{\lambda},\overline{\mu}\in\overline\Gamma$.
 We also have that
 \[
 \Ext^j_{A(n,\theta,\kappa)}(\Delta(\lambda),\Delta(\mu))\cong
  \Ext^j_{A(\overline n,\overline \theta,\overline \kappa)}(\Delta(\overline{\lambda}),\Delta(\overline{\mu}))
 \]
 for all $\lambda,\mu\in\Gamma$, $\overline{\lambda},\overline{\mu}\in\overline\Gamma$ and $j\geq 0$.  
\end{cor}
  
\begin{proof}
The graded decomposition numbers and higher extension groups are preserved under the  isomorphisms in the proof of \cref{main}  and by (graded analogues of) the  results for (co-)saturated idempotent sub- and quotient algebras in  \cite[Appendix]{Donkin}.  
\end{proof}

 \begin{thm}\label{charp}
   Let  $\gamma$  
 be an     $i$-admissible multipartition and suppose that 
 the  $e$-multicharge $\kappa\in(\ZZ/e\ZZ)^l$ contains $i\in\ZZ/e\ZZ$  as a constituent with multiplicity 0 or 1. 
The graded decomposition numbers   of $A(n,\theta,\kappa)$ over $\Bbbk$ can be given in terms of nested sign sequences as follows:
\[
d_{  \lambda\mu}(t)=\sum_{\mathclap{\omega \in \Omega(\lambda,\mu)}} \; t^{\|\omega\|} 
\]
for $\lambda,\mu\in\Gamma$ such that $\mu\trianglelefteq_\theta \lambda$.
\end{thm}
 
 \begin{proof}
  Suppose that $\chi(\gamma)$ can be broken into 3 parts
 $(i)$ a sequence of $\pm \mathbf{d}_4^j$ for $j=0,2,3$ 
 $(ii)$ either one or zero $\pm \mathbf{d}_6^j$ for $j=0,2,3$ 
 $(iii)$ a sequence of $\pm \mathbf{d}_5^j$ for $j=0,2,3$ 
 in order.  Then there clearly exists a level 1 partition $\overline{\gamma}$
 such that $\chi(\gamma) \sim \chi(\overline\gamma)$ and the result follows from \cref{main} and \cite[Theorem 4.4]{tanteo13}.  

By assumption,  $\chi(\gamma)$ is a sequence
 which 
   has a maximum of one 
 entry  $\pm \mathbf{d}_6^{j}$ for $j=0,2,3$.
 Therefore  one can apply the identification $(i)$ of \cref{equivalence} to swap
 all the subscripts  of entries   $\mathbf{d}_k^{j}$ for $j=0, 2,3$, $k=4,5$ to obtain some 
 $\overline{\chi} \sim\chi(\gamma)$ which has the above form.  Thus, the result follows.
 \end{proof}

%

\begin{eg}
Let $e=3$ and $\kappa=(0)$ and $i=2$. We want to calculate $d_{\lambda\mu}(t)$ for $\mu = (10^2,9^2,8,7^2,6^3,5,4^3,3,2^3,1^2)$ and $\lambda = (10^2,9^2,8^2,7^2,6,5^3,4^2,3,2^2,1^2)$.
By \cref{newbyiui}, we have that 
\[
d_{\lambda\mu}(t) = t^{11} + 2t^9 + 2t^7 + t^5.
\]
\end{eg}
 
\begin{eg}\label{level1ismany}
Let ${e}=3$, $\kappa=(2,1)$, $i=0$, and 
 let  $\theta=(0,1)$, $g=2$ (note that $\theta$ is a FLOTW weighting).  
The bipartition
\[
{\gamma}=((7,5,3,1^2),(5^2,4,2^2,1^2))
\]
is $0$-admissible.
We have that $\chi(\gamma) = (+\mathbf{d}^0_4, +\mathbf{d}^0_4, -\mathbf{d}^0_4,+\mathbf{d}^0_4, +\mathbf{d}^0_4, +\mathbf{d}^0_5, -\mathbf{d}^0_5, -\mathbf{d}^0_5,+\mathbf{d}^0_5,+\mathbf{d}^0_5)  $.  
Let $\overline e=4$, $\overline \kappa=(1)$, $\overline \theta=(0)$ and $\overline i=0$. The partition 
\[
\gamma=(19,18,17^3,16,13,12,11,8^3,7,6,5,2^2)
\]
is $0$-admissible with $\chi(\gamma)=\chi(\overline \gamma)$.  

The algebras $A_{\Gamma}$ and $A_{\overline\Gamma}$ are isomorphic (as graded $\Bbbk$-algebras)  and
 the graded decomposition numbers of the latter may be calculated using \cite[Theorem 4.4]{tanteo13} 
   in terms of nested sequences.  
Given $\lambda=((8,5,3,1^3),(6,5^2,3,2,1^3)),\mu=((7,5,4,2,1^2),(5^3,2^3,1^2))\in \Gamma$  
 we leave it as an exercise for the reader to  show that the unique well-nested path in this case
 is the generic latticed path with norm 11.  Therefore, 
\[
d_{((8,5,3,1^3),(6,5^2,3,2,1^3)),((7,5,4,2,1^2),(5^3,2^3,1^2)) }(t)=
t^{11}
\]
for $A(43, (0,1) , (2,1))$.
 \end{eg}
%
%
%
 
 \begin{eg}\label{magicexample}
Continuing with \cref{schurisomorhishi}, we let 
%
%
%
$\gamma=(30^6,28,20,19^2,15,11,9,7,3^6) \in \mptn 1{326}
$    
and $\overline{\gamma} = (10^4,9,5^4,3^3,1^8) \in \mptn 1{86}$. 
By \cref{schurisomorhishi}, we have that $\chi(\gamma) \sim \chi(\overline\gamma)$.  
The partially ordered sets $\Gamma$ and $\overline\Gamma$ each consist of six  partitions   
and define natural subquotients of the corresponding Schur algebras (as in \cite{Donkin}).  By \cite[Corollary 3.11]{Webster} these algebras are Morita equivalent to the corresponding subquotients $A_{\Gamma}$ and $A_{\overline\Gamma}$ of the diagrammatic Cherednik algebras.

Therefore by \cref{main}, the  subquotients of the classical Schur algebras of degrees $328 $ and  $88$ labelled by the
sets $\Gamma$  and $\overline\Gamma$ are Morita equivalent (to each  another).
\end{eg}

\begin{rmk}
By \cref{main}, one can calculate  decomposition numbers 
for more general $\kappa$ if one puts restrictions on $\theta$; 
see \cref{cref}, below.
\end{rmk}

\begin{eg}\label{cref}
Let $e=5$, $\kappa = (0,0)$, and $\theta\in\RR^2$ denote a FLOTW weighting. 
The bipartition $\gamma = ((10,8,7,5^3,3^3),(5,4,3^5,2,1^2))$ is  $0$-admissible with 
\[
\chi(\gamma) = (+\mathbf{d}_4^0, +\mathbf{d}_4^0, +\mathbf{d}_4^0, +\mathbf{d}_6^2, +\mathbf{d}_6^2, +\mathbf{d}_5^2, +\mathbf{d}_5^0, +\mathbf{d}_5^0).
\]
Now let $e=5$, $\kappa= (1)$ and $\theta=(0)$. If $\overline\gamma = (14,12,11,9,8,5^2,3,2,1^2)$, then
\[
\chi(\overline\gamma) = (+\mathbf{d}_4^0, +\mathbf{d}_4^0, +\mathbf{d}_4^0, +\mathbf{d}_5^2, +\mathbf{d}_5^0, +\mathbf{d}_5^0)
\]
and $\chi(\gamma) \sim \chi(\overline\gamma)$, by identifications $(iv)$ and $(v)$ in  \cref{equivalence}. Thus, we can still use \cref{main} to calculate $d_{\lambda\mu}(t)$ for $\lambda, \mu\in\Gamma$.
For instance, by \cref{hatty}, 
\[
d_{((11,8^2,5^3,3^3),(5,4,3^6,1^2))   ((10,8,7,5^3,3^3),(6,4,3^6,1^3))}(t) = d_{(15,12^2,9,8,5^2,3^2,1^2),(14,12,11,9^2,5^2,3^2,1^3)}(t),
\]
where the decomposition numbers are taken in the relevant algebra. We can now use \cref{charp} to find that
\[
d_{(15,12^2,9,8,5^2,3^2,1^2)(14,12,11,9^2,5^2,3^2,1^3)}(t) = t^5.
\]
\end{eg}

\section{Tensor product factorisation for non-adjacent residues }\label{sec5}

Throughout this paper, we have constructed isomorphisms between subquotient algebras
 corresponding to   subsets  of $\mptn l n$ consisting of  multipartitions which differ by moving nodes of a single, fixed residue.  
 In this section, we generalise these results to subsets of   multipartitions which differ by moving nodes of  many distinct residues, as long as these residues are nonadjacent.  
 We prove that one can factorise these algebras as a tensor product of the smaller, single residue subalgebras.  This lifts results of \cite{ct16}  to an isomorphism of algebras which holds over fields of arbitrary characteristic. In particular, we obtain the graded decomposition numbers of these algebras as products of the graded  decomposition numbers of the smaller algebras (over arbitrary fields).  

We suppose that $\mathcal M$ is a multiset of residues from  an adjacency-free residue set $S\subset I$ as in \cref{subqsec}, and let $\mathcal{M}=\mathcal{M}_0 \cup \mathcal{M}_1 \cup \dots \cup \mathcal{M}_{e-1} $ denote the disjoint decomposition of the multiset $\mathcal{M}$ into distinct residues; we let $m_{r}=|\mathcal{M}_r|$ for $0\leq r \leq e-1$. Note that since $S$ is adjacency-free, some of these multisets are empty.
We let $\Gamma=\Gamma(\mathcal M)$ and ${\Gamma}_r := \Gamma(\mathcal{M}_r$) for $0\leq r\leq e-1$.

\begin{lem}
For  an adjacency-free set $\mathcal{M}$, we have a bijection
\[
\psi : \Gamma \longrightarrow \Gamma_0 \times  \Gamma_1 \times \dots \times  \Gamma_{e-1}
\]
which is given by $\psi(\lambda)= \psi_0(\lambda) \times  \psi_1(\lambda) \times \dots \times \psi_{e-1}(\lambda)$ where $\psi_i(\lambda) $ is obtained from $\lambda$ by deleting all nodes of $\lambda\setminus\gamma$ whose residue  is not equal to $i\in \ZZ/e\ZZ$.
\end{lem}

\begin{proof}
The adjacency-free condition ensures that no two nodes in $\lambda\setminus \gamma$ appear in the same row or column for any $\lambda \in \Gamma$. The result follows.
\end{proof}

\begin{prop}\label{tensordecomp}
For  an adjacency-free set  $\mathcal{M}$, we have a bijection
\[
\psi :\SStd(\lambda,\mu) \longrightarrow \prod_{r=0}^{e-1} \SStd(\psi_r(\lambda),\psi_r(\mu))
\]
for $\lambda,\mu\in \Gamma$.    
This is given by setting $\psi (\SSTT) = \psi_0 (\SSTT)\times \psi_1 (\SSTT) \times \dots \times \psi_{e-1}  (\SSTT)$ where $\psi_i(\SSTT)$ is simply obtained by restriction of the domain, $\psi_i(\SSTT) = \SSTT{\downarrow}_{\psi_i(\lambda)} $.
This lifts to a graded vector space isomorphism over $\Bbbk$, given by
\[
\Psi: A_\Gamma(\calm, \theta  ) \to 
 A_{\Gamma_0}(\calm_0, \theta )
  \otimes_\Bbbk \dots \otimes_\Bbbk 
 A_{\Gamma_{e-1}}(\calm_{e-1}, \theta )
\]
where
\[
\Psi (C_{\SSTS\SSTT}) = C_{\psi_0(\SSTS)\psi_0(\SSTT)}  \otimes
 C_{\psi_1(\SSTS)\psi_1(\SSTT)} \otimes \dots \otimes 
C_{\psi_{e-1}(\SSTS)\psi_{e-1}(\SSTT)}.
\]
\end{prop}
 
\begin{proof}
Given $\lambda,\mu \in \Gamma_m$, 
 the set $\SStd(\lambda,\mu)$ is particularly simple to describe.  Namely, 
  $  \SStd(\lambda,\mu)$ consists of the bijective  residue-preserving maps such that
 \begin{itemize}
    \item 
    $\SSTT (r,c,k) = \mathbf{i}_{(r,c,k)}$ 
    for $(r,c,k) \in [\gamma]$;  
 \item   
    $\SSTT(r,c,k) =  \mathbf{i}_{(r',c',k')}$ 
    for $(r,c,k) \in [\lambda \setminus \gamma]$, 
   $(r',c',k') \in   [\mu \setminus \gamma]$ 
    such that 
$\mathbf{i}_{(r',c',k')}\geq \mathbf{i}_{(r,c,k)}$.
\end{itemize}
The adjacency-free condition ensures that no two nodes in $\lambda\setminus \gamma$ (or $\mu\setminus \gamma$)  for any $\lambda,\mu \in \Gamma$ appear in the same row or column.
It follows from \cref{sdfjhkhjkfhjdsgfhjkhjkhjgkfjshkdkjhfgkjhgdfs} that the map $\Psi$ is a bijection.  
 
As the bases are in bijection, the map $\Psi$ is clearly a vector space isomorphism.
The non-adjacency condition ensures that for distinct residues $i,j \in S$, the crossings of $i$-strands (or their ghosts) with $j$-strands (or their ghosts) do not provide any non-zero contributions to the grading and so the grading is preserved.
\end{proof}

\begin{eg}\label{runner}Let $e=4$, $\gamma=(\varnothing, \ldots ,\varnothing)\in \mptn 70$,  ${\kappa} = (3,1,3,3,3,1,3)$,   $\mathcal{M} =\{1^1,3^3\}$, $g=0.99$ and $\theta =(-3,-1,1,3,5,9,11) $.  
There are $2\times {5\choose 3}=20$ simple modules for this algebra. 
Consider the    space $\Delta_\mu(\lambda)$ for 
$\lambda=((1),(1),\varnothing,(1),(1),\varnothing,\varnothing)$ and $\mu=(\varnothing,\varnothing,\varnothing,(1),(1),(1),(1))$.  
We have that $\SStd(\lambda,\mu)$ has two elements, $\SSTS$ and $\SSTT$, of degrees 2 and 4, respectively.
Therefore $\Dim{(\Delta_\mu(\lambda))}=t^4+t^2$.  The element $B_\SSTS $ of degree $2$ is pictured in Figure \ref{apicforrunner}, below. 

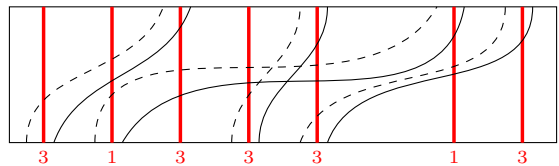
\begin{figure}[ht!]  \[    \scalefont{0.7}    \begin{tikzpicture}[scale=0.9] 
  \draw (-2,0) rectangle (6,2);
   \foreach \x in {-1.5,-0.5,0.5,1.5,2.5,4.5,5.5}  
     {\draw[wei2] (\x,0)--(\x,2); }
     \node [wei2,below] at (-0.5,0)  {\tiny $1$};
 
          \node [wei2,below] at (4.5,0)  {\tiny $1$};
          \node [wei2,below] at (-1.5,0)  {\tiny $3$};
          \node [wei2,below] at (1.5,0)  {\tiny $3$};
                    \node [wei2,below] at (0.5,0)  {\tiny $3$};
                    \node [wei2,below] at (2.5,0)  {\tiny $3$};
          \node [wei2,below] at (5.5,0)  {\tiny $3$};                                        
   \draw (0.15-0.5,0)  to [out=60,in=-100] (4.65,2) ; 
    \draw (-1.35,0)  to [out=70,in=-110] (0.65,2) ; 
      \draw (1.65,0)  to [out=90,in=-90] (2.65,2) ; 
      \draw (2.65,0)  to [out=70,in=-90] (5.65,2) ;   
        \draw[dashed] (-0.4+0.15-0.5,0)  to [out=90,in=-120] (-0.4+4.65,2) ; 
    \draw[dashed] (-0.4+-1.35,0)  to [out=90,in=-110] (-0.4+0.65,2) ; 
      \draw[dashed] (-0.4+1.65,0)  to [out=90,in=-90] (-0.4+2.65,2) ; 
      \draw[dashed] (-0.4+2.65,0)  to [out=70,in=-90] (-0.4+5.65,2) ;   
  \end{tikzpicture}
\]
\caption{
The basis element $B_\SSTS$ in Example \ref{runner}.
}
\label{apicforrunner}
\end{figure}

Under the graded vector space isomorphism in \cref{tensordecomp}, $B_\SSTS$ maps to the tensor product of diagrams depicted in \cref{apicforrunner2}.
\begin{figure}[ht!]
\[
\begin{minipage}{74mm}\scalefont{0.7}
 \begin{tikzpicture}[scale=0.9] 
  \draw (-2,0) rectangle (6,2);
    \foreach \x in {-1.5,-0.5,0.5,1.5,2.5,4.5,5.5}  
     {\draw[wei2] (\x,0)--(\x,2); }
     \node [wei2,below] at (-0.5,0)  {\tiny $1$};
  \node [white,above] at (4.5,2)  {\tiny $1$};
          \node [wei2,below] at (4.5,0)  {\tiny $1$};
          \node [wei2,below] at (-1.5,0)  {\tiny $3$};
          \node [wei2,below] at (1.5,0)  {\tiny $3$};
                    \node [wei2,below] at (0.5,0)  {\tiny $3$};
                    \node [wei2,below] at (2.5,0)  {\tiny $3$};
          \node [wei2,below] at (5.5,0)  {\tiny $3$};                                        
   \draw (0.15-0.5,0)  to [out=60,in=-100] (4.65,2) ; 
          \draw[dashed] (-0.4+0.15-0.5,0)  to [out=90,in=-120] (-0.4+4.65,2) ; 
   \end{tikzpicture}
  \end{minipage} 
  \otimes \ \ 
  \begin{minipage}{85mm}
 \begin{tikzpicture}[scale=0.9] 
  \draw (-2,0) rectangle (6,2);
   \node [white,above] at (4.5,2)  {\tiny $1$};   \foreach \x in {-1.5,-0.5,0.5,1.5,2.5,4.5,5.5}  
     {\draw[wei2] (\x,0)--(\x,2); }
     \node [wei2,below] at (-0.5,0)  {\tiny $1$};
 
          \node [wei2,below] at (4.5,0)  {\tiny $1$};
          \node [wei2,below] at (-1.5,0)  {\tiny $3$};
          \node [wei2,below] at (1.5,0)  {\tiny $3$};
                    \node [wei2,below] at (0.5,0)  {\tiny $3$};
                    \node [wei2,below] at (2.5,0)  {\tiny $3$};
          \node [wei2,below] at (5.5,0)  {\tiny $3$};                                        
                                   
     \draw (-1.35,0)  to [out=70,in=-110] (0.65,2) ; 
      \draw (1.65,0)  to [out=90,in=-90] (2.65,2) ; 
      \draw (2.65,0)  to [out=70,in=-90] (5.65,2) ;   
     \draw[dashed] (-0.4+-1.35,0)  to [out=90,in=-110] (-0.4+0.65,2) ; 
      \draw[dashed] (-0.4+1.65,0)  to [out=90,in=-90] (-0.4+2.65,2) ; 
      \draw[dashed] (-0.4+2.65,0)  to [out=70,in=-90] (-0.4+5.65,2) ;  
  \end{tikzpicture}  
    \end{minipage} 
\]
\caption{Decomposing the diagram in \cref{apicforrunner} as a tensor product.}
\label{apicforrunner2}
\end{figure}
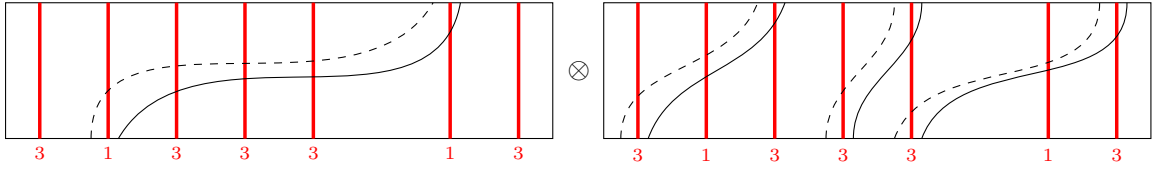
\end{eg}

%

\begin{thm}\label{tensordecomp2}
If $\mathcal{M}$ is adjacency-free, then
\[
A_\Gamma(\calm, \theta  )
 \cong 
 A_{\Gamma_0}(\calm_0, \theta )
  \otimes_\Bbbk \dots \otimes_\Bbbk 
  A_{\Gamma_{e-1}}(\calm_{e-1}, \theta )
 \]
 as  graded $\Bbbk$-algebras. This isomorphism is given by $\Psi$  from \cref{tensordecomp}.  
\end{thm}
 
\begin{proof}
By Proposition \ref{tensordecomp} we need only check that  
\begin{align*}
\Psi(C_{\SSTS\SSTT}C_{\SSTU\SSTV}) &= 
\Psi\big(\sum_{\sf W,X}r_{\sf W,X} C_{\sf W,X } \big)\\
&=
 \sum_{\sf W,X}
 r_{\sf W,X} ( C_{
 \psi_0 ({\sf  W}) \psi_0 (\sf X  )} \otimes \dots \otimes C_{
 \psi_{e-1} ({\sf  W}) \otimes \psi_{e-1} (\sf X  )}
 )\\
&=
 \sum_{\sf W,X}
 ( r^0_{\sf W,X} C_{
 \psi_0 ({\sf  W}) \psi_0 (\sf X  )}) \otimes \dots\otimes
 ( r^{e-1}_{\sf W,X} C_{
 \psi_{e-1} ({\sf  W}) \psi_{e-1} (\sf X  )}
 )\\
  &=
\Psi(C_{\SSTS\SSTT})\Psi(C_{\SSTU\SSTV})
\end{align*}
where 
$ r_{\sf W,X} = \prod_{0\leq i \leq e-1} r^i_{\sf W,X}  \in \Bbbk$ and 
the sum is over tableaux ${\sf W, X}$ 
 whose shape and weight are multipartitions belonging to $\Gamma$.  

Given a strand $A$ in a diagram $D$, we say that the strand  is left-justifiable 
if, upon  applying one of the relations \ref{rel1} to \ref{rel15} to a local neighbourhood of $A$, we can move the strand $A$ further to the left.  Otherwise, we say that the strand $A$ is left-justified.  
  In \cite[Proof of Lemmas 2.5 and 2.20]{Webster}, Webster shows that the process of left-justifying  every strand in $D$ eventually terminates, and that the result is a  linear combination of the basis elements from $\mathcal{C}$.  
  

Therefore, by left-justifying every strand in the diagram 
$C_{\SSTS\SSTT}C_{\SSTU\SSTV}$ we can rewrite this  product  as a linear    combination  of basis elements of $A_{\Gamma}(\calm, \theta  )$.  
Similarly, by  left-justifying every strand in each of  the diagrams 
$C_{\psi_i(\SSTS)\psi_i(\SSTT)} C_{\psi_i(\SSTU)\psi_i(\SSTV)}  $ for $0\leq i \leq e-1$ we shall obtain a linear combination of basis elements of 
$ A_{\Gamma_0}(\calm_0, \theta )
  \otimes_\Bbbk \dots \otimes_\Bbbk 
  A_{\Gamma_{e-1}}(\calm_{e-1}, \theta )$.  
  
 We shall proceed one tensor  component at a time.  
 Given a strand, $A$, of residue $i\in \ZZ/e\ZZ$ in $C_{\SSTS\SSTT}C_{\SSTU\SSTV}$,
 there is a corresponding strand, $\Psi_i(A)$,  of residue $i\in \ZZ/e\ZZ$ in 
 $C_{\psi_i(\SSTS)\psi_i(\SSTT)} C_{\psi_i(\SSTU)\psi_r(\SSTV)}  $ 
 which has the  same northern and southern  terminating points.  We shall say that these  strands are paired.  Similarly, for any fixed  $0\leq i \leq e-1$ and any 
  given  node of $\gamma$, we have  corresponding vertical strands in  each of the diagrams 
 $C_{\SSTS\SSTT}C_{\SSTU\SSTV}$ and 
 $C_{\psi_i(\SSTS)\psi_i(\SSTT)} C_{\psi_i(\SSTU)\psi_i(\SSTV)}  $.
  We shall say that these  strands are paired, and denote them by $A$ and $\Psi_i(A)$ as before.  
We shall proceed one tensor  component at a time.   
When considering the $i$th  component, we shall proceed by    applying 
 local relations in unison  to a    neighbourhood  of $A$ and the corresponding neighbourhood of $\Psi_i(A)$.  
 Of course, there are strands in  $C_{\SSTS\SSTT}C_{\SSTU\SSTV}$ (of residue not equal to $i\in \ZZ/e\ZZ$) which do not have counterparts in 
  $C_{\psi_i(\SSTS)\psi_i(\SSTT)} C_{\psi_i(\SSTU)\psi_i(\SSTV)}  $.
  We shall address this problem separately below.

If $i\in \ZZ/e\ZZ$ is such that $m_i=0$, then we have that $r^i_{\sf WX}=1$ 
if $\SSTT^\gamma=\psi_i(\SSTS)=\psi_i(\SSTT)=\psi_i(\SSTU)=\psi_i(\SSTV)=\sf W =\sf X$, and 0 otherwise.  
 This is simply because 
$C_{\psi_i(\SSTS)\psi_i(\SSTT)} C_{\psi_i(\SSTU)\psi_i(\SSTV)}= 1_\gamma 1_\gamma$, and all strands in the diagram are already left-justified (in particular $1_\gamma$ is itself  a basis element).  
Now, pick any $i \in \ZZ/e\ZZ$ such that $m_{i}\neq 0$.  
 Pick   a left-justifiable  $i$-strand, $A$, in the diagram $  C_{\SSTS  \SSTT}C_{\SSTU \SSTV} $, then the paired 
 strand $\Psi_i(A)$ in the diagram 
$C_{\psi_i(\SSTS)\psi_i(\SSTT)} C_{\psi_i(\SSTU)\psi_i(\SSTV)}  $ 
 is also left-justifiable (because the only non-trivial relations are between strands of adjacent residues, which are common to both diagrams).  
 We   pull these strand to the left in unison.  Along the way we shall deal with 
 \begin{itemize} 
  \item[$(i)$]    neighbourhoods in which $A$ encounters a strand in  $C_{\SSTS\SSTT}C_{\SSTU\SSTV}$ which is not paired with any strand in  $C_{\psi_i(\SSTS)\psi_i(\SSTT)} C_{\psi_i(\SSTU)\psi_i(\SSTV)}  $;
   \item[$(ii)$]      $j$-diagonals of $\gamma$ which are common  to both diagrams for $j\neq i,i\pm 1$;  
 \item[$(iii)$]      $i$-diagonals common to both diagrams.  
 \end{itemize}
 (Note that the multipartition $\gamma$ is a union of the nodes in its  $i$-diagonals  and its 
  $j$-diagonals for $j\neq i,i\pm 1$, so this list is exhaustive.)

In case $(i)$, we note that these strands are all of residue not equal to $i-1$, $i$, or $i+1$.  Therefore, 
we   apply  relations \ref{rel5} and \ref{rel8} to   pull   $A$ through the $j$-strand  of the diagram in 
 $  A_{\Gamma }(\calm , \theta )$
    or simply apply relation \ref{rel1}
to pull
 $\Psi_i(A)$ through  the corresponding empty neighbourhood of the diagram in 
    $A_{\Gamma_i}(\calm_i, \theta )$, without cost.  
    
    Now, for $(ii)$ and $(iii)$ it is clear that pulling    $A$ through the neighbourhood of the diagram in 
 $  A_{\Gamma }(\calm , \theta )$ and 
 $\Psi_i(A)$ through the corresponding neighbourhood of the diagram in  $A_{\Gamma_i}(\calm_i, \theta )$ 
 we obtain the same linear combination of diagrams in both cases.  
 This is simply because the diagrams are locally identical!    However, we must also note the following extra information, 
 \begin{itemize}
\item  in  case $(ii)$, we can apply relations \ref{rel5} and \ref{rel8} to pull $A$ through the  
 $j$-diagonal  in 
  the diagram in 
 $  A_{\Gamma }(\calm , \theta )$ and 
 $\Psi_i(A)$ through the 
 corresponding  $j$-diagonal in the     diagram in 
 $A_{\Gamma_i}(\calm_i, \theta )$, without cost.   
 \item in case $(iii)$, we can apply  \cref{lem:dotthroughdiagonal,lem:crossingthroughdiagonal,2crossB1247,2crossB1247b,2crossB1247backwards}
 to pull $A$
  through the  $i$-diagonal 
   of 
 $  A_{\Gamma }(\calm , \theta )$ 
  (respectively  
 $\Psi_i(A)$ through the
corresponding    $i$-diagonal    in the  diagram in  
 $A_{\Gamma_i}(\calm_i, \theta )$)  
 to obtain a linear combination of  diagrams 
   which  differ  from the original diagram \emph{only} in the position and decorations of  strands  which do not correspond to nodes in $\gamma$.  
\end{itemize}
In particular, in both cases all the strands labelled by nodes of $\gamma$ remain exactly as before.  
 The up-shot of this is that we can left-justify every single $i$-strand 
 in  $  A_{\Gamma }(\calm , \theta )$
    and its paired $i$-strand 
in 
    $A_{\Gamma_i}(\calm_i, \theta )$ to obtain identical linear combinations 
%
     of diagrams 
{\em and} we can do this   without affecting any 
 strands labelled by  
  nodes    
of $\gamma$.  
Therefore the strands   labelled by  
  nodes   of  $\gamma$ continue to be  vertical lines  in the configuration of a multipartition 
 and so are  left-justified.  Therefore, the above process terminates with an element 
 $$
  \sum_{\sf W,X}
  r^{i}_{\sf W,X} C_{
 \psi_{i} ({\sf  W})   \psi_{i} (\sf X  )} \in A_{\Gamma_i }(\calm_i , \theta )
$$
and a corresponding element of $ A_{\Gamma }(\calm , \theta )$ (with the same coefficients, but with diagrams which are not yet basis elements, as we must still consider the other residues $j\neq i$ for $j \in \ZZ/e\ZZ$).    We remark that at this point we know only that ${\sf W,X}$ must exist (simply because we can rewrite any product as a linear combination  of basis elements) but we have only determined  
 the semistandard tableaux $ \psi_{i} ({\sf  W})$ and $\psi_{i} (\sf X  ) $.

Continuing in this fashion through all the residues $i \in \ZZ/e\ZZ$   we obtain,
\begin{align*} 
\Psi\big(\sum_{\sf W,X}(r^0_{\sf W,X}\times \dots \times r^{e-1}_{\sf W,X}) C_{\sf W,X } \big)
&=
 \sum_{\sf W,X}
 ( r^0_{\sf W,X} C_{
 \psi_0 ({\sf  W}) \psi_0 (\sf X  )}) \otimes \dots\otimes
 ( r^{e-1}_{\sf W,X} C_{
 \psi_{e-1} ({\sf  W})   \psi_{e-1} (\sf X  )}
 )  \end{align*}
and setting   $r_{\sf W, X}=  r^0_{\sf W,X}\times \dots \times r^{e-1}_{\sf W,X} $ we obtain the required result.  
\end{proof}
%

\bibliographystyle{amsalpha}
\bibliography{master}

\end{document}